\documentclass[12pt]{article}
\usepackage{amsfonts}
\usepackage{amsmath}
\usepackage{array}
\usepackage{graphicx}
\usepackage{multirow}
\usepackage{xcolor}
\usepackage{float}
\usepackage{appendix}
\usepackage{algorithm}
\usepackage{algpseudocode}

\newtheorem{theorem}{Theorem}

\newtheorem{case}{Case}

\newtheorem{definition}{Definition}

\newtheorem{proposition}[theorem]{Proposition}
\newtheorem{remark}{Remark}

\newenvironment{proof}[1][Proof]{\noindent\textbf{#1.} }{\ \rule{0.5em}{0.5em}}
\algrenewcommand\algorithmicrequire{\textbf{Input:}}
\algrenewcommand\algorithmicensure{\textbf{Output:}}
\input{tcilatex}
\begin{document}

\title{\textbf{Robust approach for comparing two dependent normal
populations through Wald-type tests based on R\'{e}nyi's pseudodistance
estimators }}
\date{\today}
\author{Mar\'{\i}a Jaenada$^{a}$,Elena Castilla$^{b}$, Nirian Mart\'{\i}n$^{c,}$\thanks{$a$, Dep. of
		Statistics and Operations Research, Complutense University of Madrid;$b$, Dep. of Applied Mathematics, Materials Science and Engineering, and Electronic Technology,
		Rey Juan Carlos University; $c$,
		Dep. of Financial and Actuarial Economics \& Statistics, Complutense
		University of Madrid; $\ast $, Correponding author: Nirian Mart\'{\i}n,
		nirian@estad.ucm.es.} and Leandro Pardo$^{a}$}
\maketitle

\begin{abstract}
Since the two seminal papers by Fisher (1915, 1921) were published, the test
under a fixed value correlation coefficient null hypothesis for the
bivariate normal distribution constitutes an important statistical problem.
In the framework of asymptotic robust statistics, it remains being a topic
of great interest to be investigated. For this and other tests, focused on
paired correlated normal random samples, R\'{e}nyi's pseudodistance
estimators are proposed, their asymptotic distribution is established and an
iterative algorithm is provided for their computation. From them the
Wald-type test statistics are constructed for different problems of interest
and their influence function is theoretically studied. For testing null
correlation in different contexts, an extensive simulation study and two
real data based examples support the robust properties of our proposal.
\end{abstract}

\section{Introduction\protect\medskip}

In parametric estimation the role of divergence measures is very intuitive:
minimizing a suitable divergence measure between the data and the assumed
model in order to estimate the unknown parameters. These estimators are
called \textquotedblleft minimum divergence estimators\textquotedblright\
(MDEs). There is a growing body of literature that recognizes the importance
of MDEs on the basis of their robustness, without a significant loss of
efficiency, in comparison with the maximum likelihood estimator (MLE). \cite{Beran1977} showed that the minimum Hellinger distance estimator that minimizes
Hellinger distance between the modelled parametric density and its
non-parametric estimator is robust against small perturbation in the
underlying model. Other interesting results in relation to the MDEs can be
seen in \cite{tamura1986}, \cite{Simpson1987, Simpson1989}, \cite{Lindsay1994}, \cite{Pardo2006}, \cite{basu2011}, \cite{Broniatowski2012} and references therein.

In the case of continuous models, it is convenient to consider families of
divergence measures for which non-parametric estimators of the unknown
density function are needed. For instance, the theory developed by the cited
paper of Beran needs a non-parametric estimator of the unknown density
function. From this perspective, the density power divergence (DPD) family,
leading to the minimum density power divergence estimators (MDPDEs), is a
good example. For more details see \cite{basu2011}. However, there is
another important family of divergence measures which neither needs
non-parametric estimators, the R\'{e}nyi's pseudodistances (RPDs). This
family of pseudodistances will be considered in this paper.

Let $X_{1},\ldots ,X_{n}$ be a random sample of size $n$ from a population $X
$, having true and unknown density function $g,$ modelled by a parametric
family of densities $f_{\boldsymbol{\theta }}$ with $\boldsymbol{\theta }\in
\Theta \subset \mathbb{R}^{p}$. The RPD between the densities $f_{%
	\boldsymbol{\theta }}$ and $g$ is given, for a tuning parameter $\alpha >0$,
by 
\begin{align}
	R_{\alpha }\left( f_{\boldsymbol{\theta }},g\right) & =\frac{1}{\alpha +1}%
	\log \int_{-\infty }^{+\infty }f_{\boldsymbol{\theta }}^{\alpha +1}(x)dx 
	\notag \\
	& +\frac{1}{\alpha \left( \alpha +1\right) }\log \int_{-\infty }^{+\infty
	}g^{\alpha +1}(x)dx-\frac{1}{\alpha }\log \int_{-\infty }^{+\infty }f_{%
		\boldsymbol{\theta }}^{\alpha }(x)g(x)dx.  \label{1.1}
\end{align}%
The RPD was considered for the first time in \cite{Jones2001}. \cite{Fujisawa2008} used the RPD under the name of $\gamma $-cross entropy.
Due to the resemblance with the R\'{e}nyi divergence (\cite{Renyi1961}),
\cite{Broniatowski2012} named it RPD.

The RPD can be extended for $\alpha =0$ taking continuous limits on the left
yielding the expression 
\begin{equation*}
	R_{\alpha =0}\left( f_{\boldsymbol{\theta }},g\right) =\lim_{\alpha
		\downarrow 0}R_{\alpha }\left( f_{\boldsymbol{\theta }},g\right)
	=\int_{-\infty }^{+\infty }g(x)\log \frac{g(x)}{f_{\theta }(x)}dx,
\end{equation*}%
i.e., the RPD coincides with the Kullback-Leibler divergence (KLD) between $%
g $ and $f_{\boldsymbol{\theta }}$, at $\alpha =0$ (see \cite{Pardo2006}).

\cite{Broniatowski2012} established that the RPD is positive for any two
densities and for all values of the tuning parameter $\alpha >0$, $R_{\alpha
}\left( f_{\boldsymbol{\theta }},g\right) \geq 0$ and further $R_{\alpha
}\left( f_{\boldsymbol{\theta }},g\right) =0$ if and only if $f_{\boldsymbol{%
		\theta }}=g$. This property suggests the definition of the minimum RPD
estimators (MRPDEs) as the minimizer of the RPD between the assumed
distribution and the empirical distribution of the data. Therefore, the
MRPDE for the unknown parameter $\boldsymbol{\theta }$, based on the random
sample $X_{1},\ldots ,X_{n}$, $\widehat{\boldsymbol{\theta }}_{R,\alpha }=%
\widehat{\boldsymbol{\theta }}_{R,\alpha }(X_{1},\ldots ,X_{n})$, is given,
for a tuning parameter $\alpha >0$, by 
\begin{equation}
	\widehat{\boldsymbol{\theta }}_{R,\alpha }=\arg \sup_{\boldsymbol{\theta }\in
		\Theta }\sum\limits_{i=1}^{n}w_{\alpha }(\boldsymbol{\theta })f_{\boldsymbol{%
			\theta }}^{\alpha }(X_{i}),  \label{1.2}
\end{equation}%
where the weight is defined as $w_{\alpha }(\boldsymbol{\theta })=\kappa
_{\alpha }^{-\frac{\alpha }{\alpha +1}}(\boldsymbol{\theta })$ with%
\begin{equation}
	\kappa _{\alpha }(\boldsymbol{\theta })=\mathrm{E}[f_{\boldsymbol{\theta }%
	}^{\alpha }(X)]=\int_{-\infty }^{+\infty }f_{\boldsymbol{\theta }}^{\alpha
		+1}(x)dx.  \label{kappa}
\end{equation}%
Note that the value $\alpha =0$ was defined as the KLD and hence, the MRPDE
coincides with the MLE at $\alpha =0$.

The estimating equations, based on (\ref{1.2}), are given by 
\begin{equation}
	\sum_{i=1}^{n}\boldsymbol{\Psi }_{\alpha }(x_{i};\boldsymbol{%
		\theta })=\boldsymbol{0}_{p},  \label{eq}
\end{equation}%
where $\boldsymbol{0}_{p}$ is the null column vector of dimension $p$ and 
\begin{align}
	\boldsymbol{\Psi }_{\alpha }(x_{i};\boldsymbol{\theta })& =f_{\boldsymbol{%
			\theta }}^{\alpha }(x_{i})\left( \boldsymbol{u}_{\boldsymbol{\theta }%
	}(x_{i})-\boldsymbol{c}_{\alpha }\left( \boldsymbol{\theta }\right) \right) ,
	\notag \\
	\boldsymbol{u}_{\boldsymbol{\theta }}(x_{i})& =\tfrac{\partial }{\partial 
		\boldsymbol{\theta }}\log f_{\boldsymbol{\theta }}(x_{i})=\frac{\tfrac{%
			\partial }{\partial \boldsymbol{\theta }}f_{\boldsymbol{\theta }}(x_{i})}{f_{%
			\boldsymbol{\theta }}(x_{i})},  \notag \\
	\boldsymbol{c}_{\alpha }\left( \boldsymbol{\theta }\right) & =\frac{\tfrac{%
			\partial }{\partial \boldsymbol{\theta }}\log \kappa _{\alpha }(\boldsymbol{%
			\theta })}{\alpha +1}=\frac{\boldsymbol{\xi }_{\alpha }(\boldsymbol{\theta })%
	}{\kappa _{\alpha }(\boldsymbol{\theta })}=\left( c_{\alpha ,1}\left( 
	\boldsymbol{\theta }\right) ,\ldots ,c_{\alpha ,p}\left( \boldsymbol{\theta }%
	\right) \right) ^{T},  \label{cc}
\end{align}%
where $\kappa _{\alpha }(\boldsymbol{\theta })$ is given by (\ref{kappa}) and%
\begin{equation}
	\boldsymbol{\xi }_{\alpha }(\boldsymbol{\theta })=\frac{1}{\alpha +1}\tfrac{%
		\partial }{\partial \boldsymbol{\theta }}\kappa _{\alpha }(\boldsymbol{%
		\theta })=\mathrm{E}[f_{\boldsymbol{\theta }}^{\alpha }(X)\boldsymbol{u}_{%
		\boldsymbol{\theta }}(X)]=\int_{-\infty }^{+\infty }f_{\boldsymbol{\theta }%
	}^{\alpha +1}(x)\boldsymbol{u}_{\boldsymbol{\theta }}(x)dx.  \label{eps}
\end{equation}%
The MRPDE is an $M$-estimator and thus it asymptotic distribution and
influence function (IF) can be obtained based on the asymptotic theory of
the $M$-estimators. \cite{Broniatowski2012} studied the asymptotic
properties and robustness of the MRPDEs. In relation with the asymptotic
distribution they got 
\begin{equation}
	\sqrt{n}(\widehat{\boldsymbol{\theta }}_{R,\alpha }-\boldsymbol{\theta }_{0})%
	\underset{n\rightarrow \infty }{\overset{\mathcal{L}}{\rightarrow }}\mathcal{%
		N}\left( \boldsymbol{0}_{p},\boldsymbol{V}_{\alpha }\left( \boldsymbol{%
		\theta }_{0}\right) \right) ,  \label{asymptTheta}
\end{equation}%
where $\boldsymbol{\theta }_{0}$ is the true unknown value of $\boldsymbol{%
	\theta }$ and 
\begin{equation}
	\boldsymbol{V}_{\alpha }\left( \boldsymbol{\theta }\right) =\boldsymbol{S}%
	_{\alpha }^{-1}\left( \boldsymbol{\theta }\right) \boldsymbol{K}_{\alpha
	}\left( \boldsymbol{\theta }\right) \boldsymbol{S}_{\alpha }^{-1}\left( 
	\boldsymbol{\theta }\right) ,  \label{V}
\end{equation}%
with 
\begin{align}
	\boldsymbol{S}_{\alpha }\left( \boldsymbol{\theta }\right) & =-\mathrm{E}%
	\left[ \frac{\partial \boldsymbol{\Psi }_{\alpha }^{T}\left( X;\boldsymbol{%
			\theta }\right) }{\partial \boldsymbol{\theta }}\right] ,  \label{S} \\
	\boldsymbol{K}_{\alpha }\left( \boldsymbol{\theta }\right) & =\mathrm{E}%
	\left[ \boldsymbol{\Psi }_{\alpha }\left( X;\boldsymbol{\theta }\right) 
	\boldsymbol{\Psi }_{\alpha }^{T}\left( X;\boldsymbol{\theta }\right) \right]
	.  \label{K}
\end{align}%
The new result given in Section \ref{Sec1} provides a simplified version
which is very useful in practice.

At the same time \cite{Broniatowski2012} established that the IF of the
functional of the MRPDE of $\boldsymbol{\theta }$, $\boldsymbol{T}_{\alpha }$%
, is given by $\mathcal{IF}\left( x,\boldsymbol{T}_{\alpha },F_{\boldsymbol{%
		\theta }}\right) =\boldsymbol{S}_{\alpha }^{-1}\left( \boldsymbol{\theta }%
\right) \boldsymbol{\Psi }_{\alpha }\left( x,\boldsymbol{\theta }\right) $.
In aforementioned paper an application was presented to the multiple
regression model (MRM) with random covariates. \cite{toma2013}
used RP in order to define new robustness and efficiency measures. In the
same vein, \cite{castilla2020a} introduced Wald-type tests based on the
minimum RPD estimators for the MRM and its extension for Generalized Linear
models was presented in \cite{jaenada2021}. Further, \cite{castilla2020b} studied the MRPDE for the linear regression model in the ultra-high
dimensional set-up.

\section{Simplified version of the asymptotic variance-covariance matrix of R%
	\'{e}nyi's pseudodistance estimators\label{Sec1}}

This is a short but very important section as it establishes for the first
time new and short expressions of $\boldsymbol{S}_{\alpha }\left( 
\boldsymbol{\theta }\right) $ and $\boldsymbol{K}_{\alpha }\left( 
\boldsymbol{\theta }\right) $, given in (\ref{S}) and (\ref{K}), in terms of
a scalar $\kappa _{\alpha }(\boldsymbol{\theta })$, a vector $\boldsymbol{c}%
_{\alpha }\left( \boldsymbol{\theta }\right) $, and a matrix $\boldsymbol{J}%
_{\alpha }\left( \boldsymbol{\theta }\right) $, whose calculation of any
distribution is exactly the same as the one developed for MDPDEs, so the
complexity of the construction of the theory based on MRPDEs is not higher than the MDPDEs.

\begin{theorem}
	\label{Th1}The expression of the variance-covariance matrix in the
	asymptotic distribution, (\ref{asymptTheta}), is given by (\ref{V}) where 
	\begin{align}
		\boldsymbol{S}_{\alpha }\left( \boldsymbol{\theta }\right) & =\boldsymbol{J}%
		_{\alpha }\left( \boldsymbol{\theta }\right) -\kappa _{\alpha }(\boldsymbol{%
			\theta })\boldsymbol{c}_{\alpha }(\boldsymbol{\theta })\boldsymbol{c}%
		_{\alpha }^{T}(\boldsymbol{\theta }),  \label{S-1} \\
		\boldsymbol{K}_{\alpha }\left( \boldsymbol{\theta }\right) & =\boldsymbol{S}%
		_{2\alpha }\left( \boldsymbol{\theta }\right) +\kappa _{2\alpha }(%
		\boldsymbol{\theta })\left( \boldsymbol{c}_{2\alpha }(\boldsymbol{\theta })-%
		\boldsymbol{c}_{\alpha }(\boldsymbol{\theta })\right) \left( \boldsymbol{c}%
		_{2\alpha }(\boldsymbol{\theta })-\boldsymbol{c}_{\alpha }(\boldsymbol{%
			\theta })\right) ^{T},  \label{M}
	\end{align}%
	with%
	\begin{equation}
		\boldsymbol{J}_{\alpha }\left( \boldsymbol{\theta }\right) =\mathrm{E}[f_{%
			\boldsymbol{\theta }}^{\alpha }(X)\boldsymbol{u}_{\boldsymbol{\theta }}(X)%
		\boldsymbol{u}_{\boldsymbol{\theta }}^{T}(X)]=\int_{-\infty }^{+\infty }f_{%
			\boldsymbol{\theta }}^{\alpha +1}(x)\boldsymbol{u}_{\boldsymbol{\theta }}(x)%
		\boldsymbol{u}_{\boldsymbol{\theta }}^{T}(x)dx,  \label{J}
	\end{equation}%
	and the expressions of $\kappa _{\alpha }(\boldsymbol{\theta })$ and $%
	\boldsymbol{c}_{\alpha }\left( \boldsymbol{\theta }\right) $ were given by (%
	\ref{kappa}) and (\ref{cc})\ respectively.
\end{theorem}

\begin{proof}
	See Appendix \ref{A1}.\medskip
\end{proof}

\section{Minimum R\'{e}nyi pseudodistance estimators for two dependent
	populations with normal distribution}

In the previous results univariate case was considered, but it is
straightforward to extend it for the multivariate set-up. In this paper we
are considering the bidimensional normal distribution model, and so in the
following the role of $x$ is replaced by $(x,y)$ and all the integrals
are in $%
\mathbb{R}
^{2}$. In addition, we are going to get Wald-type test statistics for
testing different composite null hypothesis regarding the model parameters.

Let $(X,Y)$ be a bidimensional normal model with density function 
\begin{equation}
	f_{\boldsymbol{\theta }}(x,y)=\tfrac{1}{2\pi \sigma _{1}\sigma _{2}\sqrt{%
			1-\rho ^{2}}}\exp \left\{ -\tfrac{1}{2(1-\rho ^{2})}\left[ \left( \tfrac{%
		x-\mu _{1}}{\sigma _{1}}\right) ^{2}+\left( \tfrac{y-\mu _{2}}{\sigma _{2}}%
	\right) ^{2}-2\rho \left( \tfrac{x-\mu _{1}}{\sigma _{1}}\right) \left( 
	\tfrac{y-\mu _{2}}{\sigma _{2}}\right) \right] \right\} ,  \label{model}
\end{equation}%
$\sigma _{1},\sigma _{2}>0$, $\mu _{1},\mu _{2}\in \mathbb{R}$ and $-1<\rho
<1$, and we shall denote by 
\begin{equation}
	\boldsymbol{\theta }=\left( \mu _{1},\mu _{2},\sigma _{1},\sigma _{2},\rho
	\right) ^{T}  \label{theta}
\end{equation}%
the model parameters belonging to the parameter space $\Theta =\mathbb{R}%
^{2}\times \mathbb{R}_{+}^{2}\times (-1,1)$.

We are interested, on the basis of a random sample of size $n$, $%
(X_{1},Y_{1}),\ldots ,\allowbreak (X_{n},Y_{n})$, in obtaining the MRPDE for 
$\boldsymbol{\theta }$, as well as the asymptotic distribution. Further, we
aim to develop Wald-type tests, in the bidimensional normal model, based on
MRPDE. Some preliminary results from which proofs the reader could find many
clues were presented in \cite{martin2020}.

\begin{proposition}
	\label{Prop2}For the bidimensional normal model, (\ref{model}), the vector
	of score functions is given by%
	\begin{equation}
		\boldsymbol{u}_{\boldsymbol{\theta }}(x,y)=(u_{\mu _{1}}(x,y),u_{\mu
			_{2}}(x,y),u_{\sigma _{1}}(x,y),u_{_{\sigma _{2}}}(x,y),u_{\rho }(x,y))^{T},
		\label{u}
	\end{equation}%
	where%
	\begin{align*}
		u_{\mu _{1}}(x,y)
		&=\frac{1}{\sigma _{1}\left( 1-\rho ^{2}\right) }\left[ \frac{x-\mu _{1}%
		}{\sigma _{1}}-\rho \frac{y-\mu _{2}}{\sigma _{2}}\right] , \\
		u_{\mu _{2}}(x,y)
		&=\frac{1}{\sigma _{2}\left( 1-\rho ^{2}\right) }\left[ \frac{y-\mu _{2}%
		}{\sigma _{2}}-\rho \frac{x-\mu _{1}}{\sigma _{1}}\right] ,
	\end{align*}%
	\begin{align*}
		u_{\sigma _{1}}(x,y)
		& =-\frac{1}{\sigma _{1}}-\frac{1}{\sigma _{1}(1-\rho ^{2})}\left[ \rho 
		\frac{x-\mu _{1}}{\sigma _{1}}\frac{y-\mu _{2}}{\sigma _{2}}-\left( \frac{%
			x-\mu _{1}}{\sigma _{1}}\right) ^{2}\right] , \\
		u_{\sigma _{2}}(x,y)
		& =-\frac{1}{\sigma _{2}}-\frac{1}{\sigma _{2}(1-\rho ^{2})}\left[ \rho 
		\frac{x-\mu _{1}}{\sigma _{1}}\frac{y-\mu _{2}}{\sigma _{2}}-\left( \frac{%
			y-\mu _{2}}{\sigma _{2}}\right) ^{2}\right] ,\\
		u_{\rho }(x,y) 
		&=%
		\frac{1}{(1-\rho ^{2})}\left[ \rho +\frac{x-\mu _{1}}{\sigma _{1}}\frac{%
			y-\mu _{2}}{\sigma _{2}}\right] \\
		& \hspace{0.2cm} -\frac{\rho }{(1-\rho ^{2})^{2}}\left[ \left( \frac{x-\mu _{1}}{\sigma _{1}%
		}\right) ^{2}+\left( \frac{y-\mu _{2}}{\sigma _{2}}\right) ^{2}-2\rho \frac{%
			x-\mu _{1}}{\sigma _{1}}\frac{y-\mu _{2}}{\sigma _{2}}\right] .
	\end{align*}
\end{proposition}

\begin{proposition}
	\label{Prop4}For the bidimensional normal model, (\ref{model}), the
	expressions of (\ref{eps}) and (\ref{kappa}) are given by%
	\begin{equation}
		\boldsymbol{c}_{\alpha }\left( \boldsymbol{\theta }\right) =\left( c_{\alpha
		}(\mu _{1}),c_{\alpha }(\mu _{2}),c_{\alpha }(\sigma _{1}),c_{\alpha
		}(\sigma _{2}),c_{\alpha }(\rho )\right) ^{T}=%
		\begin{pmatrix}
			\boldsymbol{c}_{1,\alpha }\left( \boldsymbol{\theta }\right) \\ 
			\boldsymbol{c}_{2,\alpha }\left( \boldsymbol{\theta }\right)%
		\end{pmatrix}%
		,  \label{d}
	\end{equation}%
	where%
	\begin{equation*}
		\boldsymbol{c}_{1,\alpha }\left( \boldsymbol{\theta }\right)  =\boldsymbol{0%
		}_{2}, \hspace{0.3cm}
		\boldsymbol{c}_{2,\alpha }\left( \boldsymbol{\theta }\right)  =\frac{\alpha 
		}{\alpha +1}\boldsymbol{D}_{2,\sigma _{1},\sigma _{2}}^{-1}%
		\begin{pmatrix}
			-1 \\ 
			-1 \\ 
			\frac{\rho }{1-\rho ^{2}}%
		\end{pmatrix}%
		=\frac{\alpha }{\alpha +1}%
		\begin{pmatrix}
			-\frac{1}{\sigma _{1}} \\ 
			-\frac{1}{\sigma _{2}} \\ 
			\frac{\rho }{1-\rho ^{2}}%
		\end{pmatrix}%
		,
	\end{equation*}%
	with 
	\begin{equation}
		\boldsymbol{D}_{2,\sigma _{1},\sigma _{2}}=\mathrm{diag}\{\sigma _{1},\sigma
		_{2},1\},  \label{D}
	\end{equation}%
	and%
	\begin{equation}
		\kappa _{\alpha }(\boldsymbol{\theta })=\frac{1}{k^{\alpha }(\boldsymbol{%
				\theta })\left( \alpha +1\right) },  \label{eq:kappa}
	\end{equation}%
	with%
	\begin{equation}
		k(\boldsymbol{\theta })=2\pi \sigma _{1}\sigma _{2}\sqrt{1-\rho ^{2}}.
		\label{k}
	\end{equation}
\end{proposition}

In the following theorem we shall present the expressions of the matrices $%
\boldsymbol{K}_{\alpha }\left( \boldsymbol{\theta }\right) $ and $%
\boldsymbol{S}_{\alpha }\left( \boldsymbol{\theta }\right) $, defined in (%
\ref{M}) and (\ref{S-1}). But first, it is necessary to provide the
following result.

\begin{proposition}
	\label{Prop3}For the bidimensional normal model, (\ref{model}), we have the
	following results concerning with the integrals of the cross product for the
	score functions%
	\begin{equation*}
		\boldsymbol{J}_{\alpha }\left( \boldsymbol{\theta }\right) =\left( 
		\begin{array}{cc}
			\boldsymbol{J}_{1,\alpha }\left( \boldsymbol{\theta }\right) & \boldsymbol{0}%
			_{2\times 3} \\ 
			\boldsymbol{0}_{3\times 2} & \boldsymbol{J}_{2,\alpha }\left( \boldsymbol{%
				\theta }\right)%
		\end{array}%
		\right) ,
	\end{equation*}%
	where 
	\begin{align}
		\boldsymbol{J}_{1,\alpha }\left( \boldsymbol{\theta }\right) & =\frac{1}{%
			k^{\alpha }(\boldsymbol{\theta })(\alpha +1)^{2}}\boldsymbol{D}_{1,\sigma
			_{1},\sigma _{2}}^{-1}\boldsymbol{J}_{1}(\rho )\boldsymbol{D}_{1,\sigma
			_{1},\sigma _{2}}^{-1},  \label{J1} \\
		\boldsymbol{J}_{1}(\rho )& =\frac{1}{1-\rho ^{2}}\left( 
		\begin{array}{cc}
			1 & -\rho \\ 
			-\rho & 1%
		\end{array}%
		\right) ,  \notag \\
		\boldsymbol{J}_{2,\alpha }\left( \boldsymbol{\theta }\right) & =\frac{1}{%
			k^{\alpha }(\boldsymbol{\theta })(\alpha +1)^{3}}\boldsymbol{D}_{2,\sigma
			_{1},\sigma _{2}}^{-1}\boldsymbol{J}_{2,\alpha }(\rho )\boldsymbol{D}%
		_{2,\sigma _{1},\sigma _{2}}^{-1},  \notag \\
		\boldsymbol{J}_{2,\alpha }(\rho )& =\frac{1}{1-\rho ^{2}}%
		\begin{pmatrix}
			\alpha ^{2}-\rho ^{2}(\alpha ^{2}+1)+2 & \alpha ^{2}-\rho ^{2}(\alpha ^{2}+1)
			& -\rho (\alpha ^{2}+1) \\ 
			\alpha ^{2}-\rho ^{2}(\alpha ^{2}+1) & \alpha ^{2}-\rho ^{2}(\alpha ^{2}+1)+2
			& -\rho (\alpha ^{2}+1) \\ 
			-\rho (\alpha ^{2}+1) & -\rho (\alpha ^{2}+1) & \frac{\rho ^{2}(\alpha
				^{2}+1)+1}{1-\rho ^{2}}%
		\end{pmatrix}%
		,  \notag
	\end{align}%
	with%
	\begin{equation}
		\boldsymbol{D}_{1,\sigma _{1},\sigma _{2}}=\mathrm{diag}\{\sigma _{1},\sigma
		_{2}\},  \label{D1}
	\end{equation}%
	$\boldsymbol{D}_{2,\sigma _{1},\sigma _{2}}$ is given by (\ref{D}) and $k(%
	\boldsymbol{\theta })$ by (\ref{k}).
\end{proposition}

\begin{theorem}
	\label{Th4}For the bidimensional normal model, (\ref{model}), we have the
	following results concerning with the expectations of the estimating
	equations%
	\begin{equation*}
		\boldsymbol{S}_{\alpha }(\boldsymbol{\boldsymbol{\theta }})=\left( 
		\begin{array}{cc}
			\boldsymbol{S}_{1,\alpha }(\boldsymbol{\boldsymbol{\theta }}) & \boldsymbol{0%
			}_{2\times 3} \\ 
			\boldsymbol{0}_{3\times 2} & \boldsymbol{S}_{2,\alpha }(\boldsymbol{%
				\boldsymbol{\theta }})%
		\end{array}%
		\right) ,
	\end{equation*}%
	with%
	\begin{equation*}
		\boldsymbol{S}_{1,\alpha }(\boldsymbol{\boldsymbol{\theta }}) =\boldsymbol{J%
		}_{1,\alpha }\left( \boldsymbol{\theta }\right) , \hspace{0.3cm}
		\boldsymbol{S}_{2,\alpha }(\boldsymbol{\boldsymbol{\theta }}) =\frac{1}{%
			k^{\alpha }(\boldsymbol{\theta })(\alpha +1)^{3}}\boldsymbol{D}_{2,\sigma
			_{1},\sigma _{2}}^{-1}\boldsymbol{S}_{2,1}(\rho )\boldsymbol{D}_{2,\sigma
			_{1},\sigma _{2}}^{-1},  \notag
	\end{equation*}%
	and%
	\begin{equation}
		\boldsymbol{S}_{2,1}(\rho )=\frac{1}{1-\rho ^{2}}%
		\begin{pmatrix}
			2-\rho ^{2} & -\rho ^{2} & -\rho \\ 
			-\rho ^{2} & 2-\rho ^{2} & -\rho \\ 
			-\rho & -\rho & \frac{1+\rho ^{2}}{1-\rho ^{2}}%
		\end{pmatrix}%
		.  \label{s2b}
	\end{equation}%
	On the other hand 
	\begin{equation*}
		\boldsymbol{K}_{\alpha }\left( \boldsymbol{\theta }\right) =\left( 
		\begin{array}{cc}
			\boldsymbol{K}_{1,\alpha }\left( \boldsymbol{\theta }\right) & \boldsymbol{0}%
			_{2\times 3} \\ 
			\boldsymbol{0}_{3\times 2} & \boldsymbol{K}_{2,\alpha }\left( \boldsymbol{%
				\theta }\right)%
		\end{array}%
		\right)
	\end{equation*}%
	where 
	\begin{align*}
		\boldsymbol{K}_{1,\alpha }\left( \boldsymbol{\theta }\right) & =\boldsymbol{J%
		}_{1,2\alpha }\left( \boldsymbol{\theta }\right) , \\
		\boldsymbol{K}_{2,\alpha }\left( \boldsymbol{\theta }\right) & =\tfrac{1}{%
			k^{2\alpha }(\boldsymbol{\theta })(2\alpha +1)^{3}\left( \alpha +1\right)
			^{2}}\boldsymbol{D}_{2,\sigma _{1},\sigma _{2}}^{-1}\boldsymbol{K}_{2,\alpha
		}(\rho )\boldsymbol{D}_{2,\sigma _{1},\sigma _{2}}^{-1},
	\end{align*}%
	with%
	\begin{equation}
		\boldsymbol{K}_{2,\alpha }(\rho )=(\alpha +1)^{2}\boldsymbol{S}_{2,1}(\rho
		)+\alpha ^{2}\boldsymbol{S}_{2,2}(\rho ),  \label{k2-1}
	\end{equation}%
	and%
	\begin{equation}
		\boldsymbol{S}_{2,2}(\rho )=\tfrac{1}{1-\rho ^{2}}%
		\begin{pmatrix}
			1-\rho ^{2} & 1-\rho ^{2} & -\rho \\ 
			1-\rho ^{2} & 1-\rho ^{2} & -\rho \\ 
			-\rho & -\rho & \frac{\rho ^{2}}{1-\rho ^{2}}%
		\end{pmatrix}%
		.  \label{s2bb}
	\end{equation}%
	For both, $\boldsymbol{J}_{1,\alpha }\left( \boldsymbol{\theta }\right) $ is
	given by (\ref{J1}), $\boldsymbol{D}_{2,\sigma _{1},\sigma _{2}}$ by (\ref{D}%
	) and $k(\boldsymbol{\theta })$ by (\ref{k}).
\end{theorem}

\begin{proof}
	See Appendix \ref{A2}.\medskip
\end{proof}

Inverting the diagonal blocks of $\boldsymbol{S}_{\alpha }(\boldsymbol{%
	\boldsymbol{\theta }})$, we obtain 
\begin{equation}
	\boldsymbol{S}_{1,\alpha }^{-1}(\boldsymbol{\boldsymbol{\theta }})=k^{\alpha
	}(\boldsymbol{\theta })\left( \alpha +1\right) ^{2}\boldsymbol{D}_{1,\sigma
		_{1},\sigma _{2}}\boldsymbol{J}_{1}^{-1}(\rho )\boldsymbol{D}_{1,\sigma
		_{1},\sigma _{2}},  \label{s1}
\end{equation}%
where%
\begin{equation*}
	\boldsymbol{J}_{1}^{-1}(\rho )=\left( 
	\begin{array}{cc}
		1 & \rho \\ 
		\rho & 1%
	\end{array}%
	\right)
\end{equation*}%
and 
\begin{equation}
	\boldsymbol{S}_{2,\alpha }^{-1}(\boldsymbol{\boldsymbol{\theta }})=k^{\alpha
	}(\boldsymbol{\theta })\left( \alpha +1\right) ^{3}\boldsymbol{D}_{2,\sigma
		_{1},\sigma _{2}}\boldsymbol{S}_{2,1}^{-1}(\rho )\boldsymbol{D}_{2,\sigma
		_{1},\sigma _{2}},  \label{s2}
\end{equation}%
where%
\begin{equation}
	\boldsymbol{S}_{2,1}^{-1}(\rho )=\frac{1}{2}%
	\begin{pmatrix}
		1 & \rho ^{2} & \rho (1-\rho ^{2}) \\ 
		\rho ^{2} & 1 & \rho (1-\rho ^{2}) \\ 
		\rho (1-\rho ^{2}) & \rho (1-\rho ^{2}) & 2(1-\rho ^{2})^{2}%
	\end{pmatrix}%
	.  \label{s2rho}
\end{equation}%
Therefore, (\ref{V}) is given by 
\begin{equation}
	\boldsymbol{\boldsymbol{V}}_{\alpha }\boldsymbol{\left( {\boldsymbol{\theta }%
		}\right) =}\left( 
	\begin{array}{cc}
		\boldsymbol{V}_{1,\alpha }\boldsymbol{\left( {\boldsymbol{\theta }}\right) }
		& \boldsymbol{0}_{2\times 3} \\ 
		\boldsymbol{0}_{3\times 2} & \boldsymbol{V}_{2,\alpha }\boldsymbol{\left( {%
				\boldsymbol{\theta }}\right) }%
	\end{array}%
	\right) ,  \label{1.4}
\end{equation}%
where 
\begin{align}
	\boldsymbol{V}_{1,\alpha }({\boldsymbol{\theta }})& =\boldsymbol{S}%
	_{1,\alpha }^{-1}\left( \boldsymbol{\theta }\right) \boldsymbol{K}_{1,\alpha
	}\left( \boldsymbol{\theta }\right) \boldsymbol{S}_{1,\alpha }^{-1}\left( 
	\boldsymbol{\theta }\right) =\frac{\left( \alpha +1\right) ^{4}}{\left(
		2\alpha +1\right) ^{2}}\boldsymbol{D}_{1,\sigma _{1},\sigma _{2}}\boldsymbol{%
		J}_{1}^{-1}(\rho )\boldsymbol{D}_{1,\sigma _{1},\sigma _{2}},  \label{v1} \\
	\boldsymbol{V}_{2,\alpha }({\boldsymbol{\theta }})& =\boldsymbol{S}%
	_{2,\alpha }^{-1}\left( \boldsymbol{\theta }\right) \boldsymbol{K}_{2,\alpha
	}\left( \boldsymbol{\theta }\right) \boldsymbol{S}_{2,\alpha }^{-1}\left( 
	\boldsymbol{\theta }\right) =\frac{\left( \alpha +1\right) ^{4}}{\left(
		2\alpha +1\right) ^{3}}\boldsymbol{D}_{2,\sigma _{1},\sigma _{2}}\boldsymbol{%
		V}_{2,\alpha }\left( \rho \right) \boldsymbol{D}_{2,\sigma _{1},\sigma _{2}},
	\label{v2}
\end{align}%
with%
\begin{equation}
	\boldsymbol{V}_{2,\alpha }\left( \rho \right) =\boldsymbol{S}%
	_{2,1}^{-1}(\rho )\boldsymbol{K}_{2,\alpha }(\rho )\boldsymbol{S}%
	_{2,1}^{-1}(\rho )=(\alpha +1)^{2}\boldsymbol{S}_{2,1}^{-1}(\rho )+\alpha
	^{2}\boldsymbol{S}_{2,1}^{-1}(\rho )\boldsymbol{S}_{2,2}(\rho )\boldsymbol{S}%
	_{2,1}^{-1}(\rho ).  \label{v2rho}
\end{equation}

Based on the previous results we have the following Theorem.

\begin{theorem}
	\label{Th5}For the bidimensional normal model, (\ref{model}), the MRPDE for $%
	\boldsymbol{\theta }$, 
	\begin{equation}
		\widehat{\boldsymbol{\theta }}_{R,\alpha }=(\widehat{\mu }_{1,R,\alpha },%
		\widehat{\mu }_{2,R,\alpha },\widehat{\sigma }_{1,R,\alpha },\widehat{\sigma 
		}_{2,R,\alpha },\widehat{\rho }_{R,\alpha })^{T},  \label{thetaR}
	\end{equation}%
	is obtained as a solution of%
	\begin{equation*}
		\sum_{i=1}^{n}w_{i,\boldsymbol{\theta }}^{-\alpha }\left( \boldsymbol{u}_{%
			\boldsymbol{\theta }}(X_{i},Y_{i})-\boldsymbol{c}_{\alpha }(\boldsymbol{%
			\theta })\right) =\boldsymbol{0}_{5},
	\end{equation*}%
	with 
	\begin{equation}
		w_{i,\boldsymbol{\theta }}=\exp \left\{ \tfrac{1}{2(1-\rho ^{2})}\left[ (%
		\tfrac{X_{i}-\mu _{1}}{\sigma _{1}})^{2}+(\tfrac{Y_{i}-\mu _{2}}{\sigma _{2}}%
		)^{2}-2\rho \tfrac{X_{i}-\mu _{1}}{\sigma _{1}}\tfrac{Y_{i}-\mu _{2}}{\sigma
			_{2}}\right] \right\} ,  \label{eq:f_alpha}
	\end{equation}%
	$\boldsymbol{u}_{\boldsymbol{\theta }}(X_{i},Y_{i})$ is given in Proposition %
	\ref{Prop2} and $\boldsymbol{c}_{\alpha }(\boldsymbol{\theta })$\ in
	Proposition \ref{Prop4}. The corresponding asymptotic distribution is%
	\begin{equation}
		\sqrt{n}(\widehat{\boldsymbol{\theta }}_{R,\alpha }-\boldsymbol{\theta }_{0})%
		\underset{n\rightarrow \infty }{\overset{\mathcal{L}}{\rightarrow }}\mathcal{%
			N}\left( \boldsymbol{0}_{p},\boldsymbol{V}_{\alpha }\left( \boldsymbol{%
			\theta }_{0}\right) \right) ,  \label{1.5}
	\end{equation}%
	where $\boldsymbol{\theta }_{0}$ is the true unknown value of (\ref{theta})
	and $\boldsymbol{V}_{\alpha }\left( {\boldsymbol{\theta }}\right) $ was
	given in (\ref{1.4}).
\end{theorem}

The following algorithm is useful for computing the MRPDE of $\boldsymbol{%
	\theta }$ given in Theorem \ref{Th5}. It works iteratively for a sequence of
increasing values of the tuning parameter, $\alpha \in \{\alpha
_{k}\}_{k=0}^{K}$ with $\alpha _{k}=\frac{k}{K}$,
having a very simple iterative
scheme and converging rapidly to the final optimal value. As the MLEs have
an explicit expression, the tuning parameter $\alpha _{0}=0$ initializes the
iterations. Herein the following parameter transformation is considered 
\begin{equation*}
	\boldsymbol{\vartheta }=\left( \mu _{1},\mu _{2},\zeta _{1}^{2},\zeta
	_{2}^{2},\rho \right) ^{T}
	\hspace{0.3cm} \text{where} \hspace{0.3cm} 
	\sigma _{j}^{2}=(\alpha +1)\zeta _{j}^{2},\quad j=1,2.
\end{equation*}%
The strength of the algorithm is its simplicity for estimating in a chained
way and with the semi-explicit expressions given in the inner iterations,
with expression which mimics the MLEs as weighted version (see the
corresponding proof given in Appendix \ref{A3}). The updating recursive
elements comprise only of the weights, as the name Iteratively Reweighted
Moments Algorithm suggests. 
The semi-explicit expressions given in the inner iterations of Algorithm \ref%
{Algo} are particular cases of (6) and (7) of \cite{toma2015} 
with $N$ equals $2$. However, our proposed algorithm is different from the
one proposed in the Monte-Carlo simulations of Toma and Leoni-Aubin (2015)
in two features. First, their algorithm does not consider outer iterations
as the estimation is initialized always with the MLE and second, they do not
consider the reparameterization of the variance-covariance matrix. Both are
crucial features, first one to save iterations when a grid of the tuning
parameter is considered, taking into account the smoothness of the objective
function. The second one justifies our proposed Iteratively Reweighted
Moments Algorithm to be an EM algorithm. It is well-known that some
M-estimators are obtained from estimating equations associated to scale
mixture of normal distributions, and their corresponding iterative
reweighted algorithm is an EM algorithm. Under weak conditions, each step of
IRLS increases the objective function and the solution of the iterative
reweighted algorithm converges to a local maximum of the likelihood
function. The minimum pseudodistance distance estimators for bivariate
normal populations falls within this class of M-estimators. From Algorithm %
\ref{Algo} and taking into account the Bayesian interpretation of the EM
algorithm (see for example Section 12.2 in \cite{little2019}, we
obtain the following way of understanding the weights update or the E step.
Let us consider the sample $(X_{i}^{(k)},Y_{i}^{(k)},Z_{i}^{(k)})$, $%
i=1,\ldots ,n$, with $(X_{i}^{(k)},Y_{i}^{(k)})$ being a bivariate normal
distribution with parameters $\boldsymbol{\vartheta }_{k}=( \mu
_{1},\mu _{2},\zeta _{1,k}^{2},\zeta _{2,k}^{2},\rho ) ^{T}$, where $%
\zeta _{j,k}^{2}=\frac{1}{\alpha _{k}+1}\sigma _{j}^{2}$, $j=1,2$, and $%
Z_{i}^{(k)}$ being an auxiliary random variable (a latent random variable
for the usual EM algorithm with missing data). Assuming that $Z_{i}^{(k)}$
is a random variable degenerated at $\frac{\alpha _{k}}{\alpha _{k}+1}$
(prior distribution), it holds $\left( (X_{i}^{(k)},Y_{i}^{(k)})\left\vert Z_{i}^{(k)}=z\right. \right) $ behaves as a bivariate
normal distribution with parameters $(\mu _{1},\mu _{2},\frac{\zeta
	_{1,k}^{2}}{z},\frac{\zeta _{2,k}^{2}}{z},\rho )^{T}$, while the weight is
the expectation of the posterior distribution (line 13 of Algorithm \ref%
{Algo}),%
\begin{equation*}
	\widehat{\varpi }_{i,\alpha _{k}}=E\left[ Z_{i}^{(k)}\left\vert
	(X_{i}^{(k)},Y_{i}^{(k)})\right. \right] ,\quad i=1,\ldots ,n.
\end{equation*}
The estimator of $\boldsymbol{\vartheta }_{k}$ is updated in the M step
(line 14 of Algorithm \ref{Algo}). While parameters $\zeta _{j,k}^{2}$, $%
j=1,2$, produce a shrinkage effect on the original variance parameters $%
\sigma _{j}^{2}$, $j=1,2$, the estimates associated with small values the
weights, $\widehat{\varpi }_{i,\alpha _{k}}^{(k)}$, $i=1,\ldots ,n$, produce
a down-weighting effect on estimates of $\boldsymbol{\vartheta }_{k}$\ to
prevent from high value of the Mahalanobis distance (line 11) for the $i$-th
observation and $k$-th tuning parameter $\alpha _{k}$, $k=1,\ldots ,n$.

\begin{algorithm}
	\scriptsize
	\caption{Iteratively Reweighted Moments Algorithm for MRPDE of $\boldsymbol{\vartheta }$}
	\label{Algo}
	\begin{algorithmic}[1] 
		\Require  $K$, $\{\alpha _{k}=\frac{k}{K}\}_{k=0}^{K}$, $\{(X_{i},Y_{i})\}_{i=1}^{n}$, $\xi $;
		\State $k\leftarrow 0$, $r_{k=0}\leftarrow 0$
		\Comment {Inizialization}
		\State $\mathrm{vec}\{\widehat{\varpi}_{i,\alpha
			_{k=0}=0,(r_{k=0}=0)}\}_{i=1}^{n}=\boldsymbol{1}_{n} $
		\State Compute $\widehat{\boldsymbol{\vartheta }}_{\alpha
			_{k=0}=0,(r_{k=0}=0)}$:
		\begin{align*}
			\widehat{\mu }_{1,R,\alpha _{k=0}=0,(r_{k=0}=0)}& =\frac{\sum_{i=1}^{n}X_{i}%
			}{n},\quad \widehat{\zeta }_{1,R,\alpha _{k=0}}^{2}=\frac{%
				\sum_{i=1}^{n}(X_{i}-\widehat{\mu }_{1,R,\alpha _{0}=0,(r_{k=0}=0)})^{2}}{n},
			\\
			\widehat{\mu }_{2,R,\alpha _{k=0}=0,(r_{k=0}=0)}& =\frac{\sum_{i=1}^{n}Y_{i}%
			}{n},\quad \widehat{\zeta }_{2,R,\alpha _{k=0}}^{2}=\frac{%
				\sum_{i=1}^{n}(Y_{i}-\widehat{\mu }_{2,R,\alpha _{0}=0,(r_{k=0}=0)})^{2}}{n},
			\\
			\widehat{\rho }_{R,\alpha _{k=0}=0,(r_{k=0}=0)}& =\frac{\sum_{i=1}^{n}\frac{%
					X_{i}-\widehat{\mu }_{1,R,\alpha _{k=0}=0,(r_{k=0}=0)}}{\widehat{\zeta }%
					_{1,R,\alpha _{k=0}=0,(r_{k=0}=0)}}\frac{Y_{i}-\widehat{\mu }_{2,R,\alpha
						_{0}=0,(r_{k=0}=0)}}{\widehat{\zeta }_{2,R,\alpha _{k=0}=0,(r_{k=0}=0)}}}{n};
		\end{align*}
		
		\While{$k<K$}
		\Comment {Outer loop starts}
		\State $k\leftarrow k+1$, $r_{k}\leftarrow 0$
		\State $\mathrm{vec}\{\widehat{\varpi}_{i,\alpha
			_{k},(r_{k}=0)}\}_{i=1}^{n}\leftarrow \mathrm{vec}\{\widehat{\varpi}_{i,\alpha _{k-1,r_{k-1}}}\}_{i=1}^{n}$
		\State $\widehat{\boldsymbol{\vartheta}}_{\alpha _{k},(r_{k}=0)}=\widehat{\boldsymbol{\vartheta}}_{\alpha
			_{k-1,r_{k-1}}}$
		\Comment {Inner loop starts}
		\Repeat 
		\State  $\mathrm{vec}\{\widehat{X}_{i,\alpha _{k},(r_{k})}\}_{i=1}^{n}\leftarrow 
		\mathrm{vec}\left\{ \tfrac{X_{i}-\widehat{\mu }_{1,R,\alpha _{k},(r_{k})}}{%
			\widehat{\zeta }_{1,R,\alpha _{k},(r_{k})}}\right\} _{i=1}^{n}$
		\State  $\mathrm{vec}\{\widehat{Y}_{i,\alpha _{k},(r_{k})}\}_{i=1}^{n}\leftarrow 
		\mathrm{vec}\left\{ \tfrac{Y_{i}-\widehat{\mu }_{2,R,\alpha _{k},(r_{k})}}{%
			\widehat{\zeta }_{2,R,\alpha _{k},(r_{k})}}\right\} _{i=1}^{n}$
		\State  $\mathrm{vec}\{d_{i,\alpha _{k},(r_{k})}^{2}\}_{i=1}^{n}\leftarrow \mathrm{%
			vec}\left\{ \tfrac{\widehat{X}_{i,\alpha _{k},(r_{k})}^{2}+\widehat{Y}%
			_{i,\alpha _{k},(r_{k})}^{2}-2\widehat{\rho }_{1,R,\alpha _{k},(r_{k})}%
			\widehat{X}_{i,\alpha _{k},(r_{k})}\widehat{Y}_{i,\alpha _{k},(r_{k})}}{%
			1- \widehat{\rho }_{1,R,\alpha _{k},(r_{k})} ^{2}}\right\}
		_{i=1}^{n}$
		\State $r_{k} \leftarrow r_{k}+1, $
		\State $\mathrm{vec}\{\widehat{\varpi }_{i,\alpha
			_{k},(r_{k})}\}_{i=1}^{n}\leftarrow \mathrm{vec}\{\exp ^{-\frac{\alpha _{h}}{%
				\alpha _{h}+1}}(\frac{1}{2}d_{i,\alpha _{k},(r_{k}-1)}^{2})\}_{i=1}^{n}$
		\State Compute $\widehat{\boldsymbol{\vartheta }}_{\alpha
			_{k},(r_{k})}$:%
		\begin{align*}
			t_{\alpha _{k},(r_{k})}& =\sum_{i=1}^{n}\widehat{\varpi }_{i,\alpha
				_{k},(r_{k})}, \\
			\widehat{\mu }_{1,R,\alpha _{k},(r_{k})}& =\frac{\sum_{i=1}^{n}\widehat{%
					\varpi }_{i,\alpha _{k},(r_{k})}X_{i}}{t_{\alpha _{k},(r_{k})}},\quad 
			\widehat{\zeta }_{1,R,\alpha _{k},(r_{k})}^{2}=\frac{\sum_{i=1}^{n}\widehat{%
					\varpi }_{i,\alpha _{k},(r_{k})}(X_{i}-\widehat{\mu }_{1,R,\alpha
					_{k},(r_{k})})^{2}}{t_{\alpha _{k},(r_{k})}}, \\
			\widehat{\mu }_{2,R,\alpha _{k},(r_{k})}& =\frac{\sum_{i=1}^{n}\widehat{%
					\varpi }_{i,\alpha _{k},(r_{k})}Y_{i}}{t_{\alpha _{k},(r_{k})}},\quad 
			\widehat{\zeta }_{2,R,\alpha _{k},(r_{k})}^{2}=\frac{\sum_{i=1}^{n}\widehat{%
					\varpi }_{i,\alpha _{k},(r_{k})}(Y_{i}-\widehat{\mu }_{2,R,\alpha
					_{k},(r_{k})})^{2}}{t_{\alpha _{k},(r_{k})}}, \\
			\widehat{\rho }_{R,\alpha _{k},(r_{k})}& =\frac{\sum_{i=1}^{n}\widehat{%
					\varpi }_{i,\alpha _{k},(r_{k})}\tfrac{X_{i}-\widehat{\mu }_{1,R,\alpha
						_{k},(r_{k})}}{\widehat{\zeta }_{1,R,\alpha _{k},(r_{k})}}\tfrac{Y_{i}-%
					\widehat{\mu }_{2,R,\alpha _{k},(r_{k})}}{\widehat{\zeta }_{2,R,\alpha
						_{k},(r_{k})}}}{t_{\alpha _{k},(r_{k})}}
		\end{align*}
		\Until{ $\left\Vert \widehat{\boldsymbol{\vartheta}}_{\alpha _{k},(r_{k})}-\widehat{\boldsymbol{\vartheta}}_{\alpha _{k},(r_{k}-1)}\right\Vert _{2}<\xi $}
		\Comment {Inner loop ends}
		\EndWhile
		\Comment {Outer loop ends}
		\Ensure $\{\widehat{\boldsymbol{\vartheta}}_{\alpha _{k}}=\widehat{\boldsymbol{\vartheta}}_{\alpha _{k},(r_{k})}\}_{k=0}^{K}$.
	\end{algorithmic}
\end{algorithm}

It is clear that due to the invariance property of the MRPDEs, it holds%
\begin{equation*}
	\widehat{\sigma }_{j,R,\alpha }=\sqrt{\alpha +1}\widehat{\zeta }_{j,R,\alpha
	},\quad j=1,2.
\end{equation*}

\section{Wald-type tests based on R\'{e}nyi's pseudodistance estimators \label%
	{secWald}}

Based on the asymptotic distribution of the MRPDE for $\boldsymbol{\theta }$%
, $\widehat{\boldsymbol{\theta }}_{R,\alpha }$, given in Theorem \ref{Th5},
we present Wald-type tests for testing composite null hypothesis regarding
bidimensional normal model parameters.

The restricted parameter space $\Theta _{0}\subset \Theta =\mathbb{R}%
^{2}\times \mathbb{R}_{+}^{2}\times (-1,1)$, is often defined by a set of $r$
restrictions of the form 
\begin{equation}
	\boldsymbol{m}(\boldsymbol{\theta })=\boldsymbol{0}_{r},  \label{2.5}
\end{equation}%
where $\boldsymbol{\theta }$ is (\ref{theta}) and $\boldsymbol{m}:\;\Theta
\rightarrow \mathbb{R}^{r}$ (see Serfling 1980). Assume that the $5\times r$
matrix 
\begin{equation}
	\boldsymbol{M}({\boldsymbol{\theta }})=\frac{\partial \boldsymbol{m}^{T}(%
		\boldsymbol{\theta })}{\partial {\boldsymbol{\theta }}}  \label{2.6}
\end{equation}%
exists and is continuous in ${\boldsymbol{\theta }}$, and rank$\left( 
\boldsymbol{M}({\boldsymbol{\theta }})\right) =r$, where $r\leq 5$. Let $%
(X_{1},Y_{1}),\ldots ,\allowbreak (X_{n},Y_{n})$ be a random sample of size $%
n$ from a distribution modelled by the bidimensional normal model
probability density function $f_{\boldsymbol{\theta }}(x,y)$, where ${%
	\boldsymbol{\theta }}\in \Theta $. Our interest is in testing the hypothesis 
\begin{equation}
	H_{0}:{\boldsymbol{\theta }}\in \Theta _{0}\quad \text{against}\quad H_{1}:%
	\boldsymbol{\theta }\notin \Theta _{0},  \label{4.1}
\end{equation}%
where $\Theta _{0} = \{\boldsymbol{\theta} \in \Theta : \boldsymbol{m}(\boldsymbol{\theta}) = \boldsymbol{0} \}$. 

\begin{definition}
	\label{defW}Let $\widehat{\boldsymbol{\theta }}_{R,\alpha }$ be the MDPDE of 
	${\boldsymbol{\theta }}$. The family of proposed Wald-type test statistics
	for testing the null hypothesis in (\ref{4.1}) is given by 
	\begin{equation}
		W_{n,\alpha }(\widehat{\boldsymbol{\theta }}_{R,\alpha })=n\boldsymbol{m}%
		^{T}(\widehat{\boldsymbol{\theta }}_{R,\alpha })\left( \boldsymbol{M}^{T}(%
		\widehat{\boldsymbol{\theta }}_{R,\alpha })\boldsymbol{V}_{\alpha }(\widehat{%
			\boldsymbol{\theta }}_{R,\alpha })\boldsymbol{M}(\widehat{\boldsymbol{\theta 
		}}_{R,\alpha })\right) ^{-1}\boldsymbol{m}(\widehat{\boldsymbol{\theta }}%
		_{R,\alpha }),  \label{4.2}
	\end{equation}%
	where the matrix $\boldsymbol{V}_{\alpha }$ is as in (\ref{1.4}) and the
	functions $\boldsymbol{m}(\cdot )$ and $\boldsymbol{M}(\cdot )$\ are defined
	in (\ref{2.5}) and (\ref{2.6}).
\end{definition}

\begin{theorem} \label{Thm7}
	\label{th:Wald} The asymptotic null distribution of the proposed Wald-type
	test statistics given in (\ref{4.2}) is chi-square with $r$ degrees of
	freedom, $\chi^2_r$.
\end{theorem}

\begin{proof}
	See Appendix \ref{app:Thm7}.
\end{proof}

We will reject the null hypothesis in (\ref{2.5}) if $W_{n,\alpha }(\widehat{%
	\boldsymbol{\theta }}_{R,\alpha })>\chi _{r,\varsigma }^{2}$, where $\chi
_{r,\varsigma }^{2}$ is the upper percentage point of order $\varsigma $ of
the $\chi _{r}^{2}$ distribution. Based on Definition \ref{defW}\ and
Theorem \ref{th:Wald}, the following subsections are devoted to derive a
variety hypothesis tests for the bidimensional normal model. The proofs of the stated formulas are given in Appendix \ref{A5}.

\begin{case}
	\label{subSec1}\textbf{(Comparing means of two dependent populations with
		normal distribution)}.\textbf{\medskip }\texttt{\newline
	}If we are interested in testing%
	\begin{equation}
		H_{0}:\mu _{1}=\mu _{2},  \label{TestMeans}
	\end{equation}%
	the corresponding Wald-type test statistics based on MRPDEs is%
	\begin{equation}
		W_{n,\alpha }(\widehat{\boldsymbol{\theta }}_{R,\alpha })=n\frac{\left(
			2\alpha +1\right) ^{2}}{\left( \alpha +1\right) ^{4}}\frac{(\widehat{\mu }%
			_{1,R,\alpha }-\widehat{\mu }_{2,R,\alpha })^{2}}{(\widehat{\sigma }%
			_{1,R,\alpha }-\widehat{\sigma }_{2,R,\alpha })^{2}+2(1-\widehat{\rho }%
			_{R,\alpha })\widehat{\sigma }_{1,R,\alpha }\widehat{\sigma }_{2,R,\alpha }},
		\label{waldCase1}
	\end{equation}%
	and its asymptotic distribution is a chi-squared distribution with one
	degree of freedom under (\ref{TestMeans}).
\end{case}

In Case \ref{subSec1}\ a non-standard Behrens--Fisher problem is covered,
i.e., a comparison of the means of two populations which may possess not
only different variances, but also a non-null correlation. It is of great
interest to be aware that formulating the same problem as a paired test
constructed taking the difference of both populations, $V=X-Y$, as a single
population problem for testing $H_{0}:\mu _{V}=0$, with an unknown variance $%
\sigma _{V}^{2}$, the same value of the Wald-type test statistics%
\begin{equation*}
	W_{n,\alpha }(\widehat{\mu }_{V,R,\alpha },\widehat{\sigma }_{V,R,\alpha })=n%
	\frac{\left( 2\alpha +1\right) ^{2}}{\left( \alpha +1\right) ^{4}}\frac{%
		\widehat{\mu }_{V,R,\alpha }^{2}}{\widehat{\sigma }_{V,R,\alpha }^{2}}
\end{equation*}%
is obtained ($W_{n,\alpha }(\widehat{\boldsymbol{\theta }}_{R,\alpha
})=W_{n,\alpha }(\widehat{\mu }_{V,R,\alpha },\widehat{\sigma }_{V,R,\alpha
})$) from the invariance property of the R\'{e}nyi's pseudodistance
estimators, since $\mu _{V}=\mu _{1}-\mu _{2}$, $\sigma _{V}^{2}=\sigma
_{1}^{2}+\sigma _{2}^{2}-2\rho \sigma _{1}\sigma _{2}=(\sigma _{1}-\sigma
_{2})^{2}+2(1-\rho )\sigma _{1}\sigma _{2}$.

The most efficient classic procedure to address this problem is the paired $%
t $-test, i.e 
\begin{equation*}
	T_{V}=\sqrt{n}\frac{\overline{V}_{n}}{S_{V,n-1}},
\end{equation*}%
where%
\begin{equation*}
	\overline{V}_{n}=\overline{X}_{n}-\overline{Y}_{n},
\end{equation*}%
\begin{align*}
	S_{V,n-1}^{2}& =\frac{1}{n-1}\sum_{i=1}^{n}(V_{i}-\overline{V}%
	_{n})^{2}=S_{X,n-1}^{2}+S_{Y,n-1}^{2}-2S_{XY,n-1}, \\
	S_{X,n-1}^{2}& =\frac{1}{n-1}\sum_{i=1}^{n}(X_{i}-\overline{X}_{n})^{2}=%
	\frac{n}{n-1}\widehat{\sigma }_{1,R,\alpha =0}^{2}, \\
	S_{Y,n-1}^{2}& =\frac{1}{n-1}\sum_{i=1}^{n}(Y_{i}-\overline{Y}_{n})^{2}=%
	\frac{n}{n-1}\widehat{\sigma }_{2,R,\alpha =0}^{2}, \\
	S_{XY,n-1}& =\frac{1}{n-1}\sum_{i=1}^{n}(X_{i}-\overline{X}_{n})(Y_{i}-%
	\overline{Y}_{n})=\frac{n}{n-1}\widehat{\rho }_{R,\alpha =0}\widehat{\sigma }%
	_{2,R,\alpha =0}\widehat{\sigma }_{1,R,\alpha =0}.
\end{align*}%
Its exact distribution is a Student-$t$ with $n-1$ degrees of freedom, $%
t_{n-1}$.

\begin{case}
	\label{subSec2}\textbf{(Comparing variances of two dependent populations
		with normal distribution)}.\textbf{\medskip }\texttt{\newline
	}If we are interested in testing%
	\begin{equation}
		H_{0}:\sigma _{1}=\sigma _{2}.  \label{H0}
	\end{equation}%
	the corresponding Wald-type test statistics based on MRPDEs is%
	\begin{equation}
		W_{n,\alpha }(\widehat{\boldsymbol{\theta }}_{R,\alpha })=n\frac{\left(
			2\alpha +1\right) ^{3}}{\left( \alpha +1\right) ^{6}}\frac{(\widehat{\sigma }%
			_{1,R,\alpha }-\widehat{\sigma }_{2,R,\alpha })^{2}}{\beta _{\alpha }(%
			\widehat{\boldsymbol{\theta }}_{R,\alpha })},  \label{TS}
	\end{equation}%
	where%
	\begin{equation}
		\beta _{\alpha }(\widehat{\boldsymbol{\theta }}_{R,\alpha })=\left[ \frac{1}{%
			4}\left( \tfrac{\alpha }{\alpha +1}\right) ^{2}+\frac{1}{2}\right] (\widehat{%
			\sigma }_{1,R,\alpha }-\widehat{\sigma }_{2,R,\alpha })^{2}+(1-\widehat{\rho 
		}_{R,\alpha }^{2})\widehat{\sigma }_{1,R,\alpha }\widehat{\sigma }%
		_{2,R,\alpha }.  \label{b2}
	\end{equation}%
	The asymptotic distribution of (\ref{TS}) is a chi-squared distribution with 
	$1$ degree of freedom under (\ref{H0}).
\end{case}

\begin{case}
	\label{subSec3}\textbf{(Fixing a value of the for correlation coefficient of
		two normal populations)}.\textbf{\medskip }\texttt{\newline
	}If we are interested in testing%
	\begin{equation}
		H_{0}:\rho =\rho _{0},  \label{H0b}
	\end{equation}%
	the corresponding Wald-type test statistics based on MRPDEs is%
	\begin{equation}
		W_{n,\alpha }(\widehat{\boldsymbol{\theta }}_{R,\alpha })=n\frac{\left(
			2\alpha +1\right) ^{3}}{\left( \alpha +1\right) ^{6}}\frac{(\widehat{\rho }%
			_{R,\alpha }-\rho _{0})^{2}}{(1-\widehat{\rho }_{R,\alpha }^{2})^{2}},
		\label{TS2}
	\end{equation}%
	and its asymptotic distribution is a chi-squared distribution with $1$
	degree of freedom under (\ref{H0b}).
\end{case}

The classic Wald and Rao test statistics are given by%
\begin{align*}
	W_{n,\alpha =0}(\widehat{\boldsymbol{\theta }}_{R,\alpha =0})& =n\frac{(%
		\widehat{\rho }_{R,\alpha =0}-\rho _{0})^{2}}{(1-\widehat{\rho }_{R,\alpha
			=0}^{2})^{2}}, \\
	R_{n,\alpha =0}(\widehat{\boldsymbol{\theta }}_{R,\alpha =0})& =n\frac{(%
		\widehat{\rho }_{R,\alpha =0}-\rho _{0})^{2}}{(1-\rho _{0}\widehat{\rho }%
		_{R,\alpha =0})^{2}},
\end{align*}%
where%
\begin{equation*}
	\widehat{\rho }_{R,\alpha =0}=\frac{ \sum_{i=1}^{n}(X_{i}-\bar{X}_{n})(Y_{i}-%
		\bar{Y}_{n})}{\sqrt{ \sum_{i=1}^{n}(X_{i}-\bar{X}_{n})^{2}}\sqrt{%
			\sum_{i=1}^{n}(Y_{i}-\bar{Y}_{n})^{2}}}=R_{XY},
\end{equation*}%
but%
\begin{equation*}
	W_{n,\alpha =0}^{\prime }(\widehat{\boldsymbol{\theta }}_{R,\alpha =0})=n%
	\frac{(\widehat{\rho }_{R,\alpha =0}-\rho _{0})^{2}}{(1-\rho _{0}^{2})^{2}}
\end{equation*}%
convergences more rapidly to the chi-square distribution with $1$ degree of
freedom (see \cite{anderson2003}). The extension of $W_{n,\alpha =0}^{\prime }(%
\widehat{\boldsymbol{\theta }}_{R,\alpha =0})$ to%
\begin{equation}
	W_{n,\alpha }^{\prime }(\widehat{\boldsymbol{\theta }}_{R,\alpha })=n\frac{%
		\left( 2\alpha +1\right) ^{3}}{\left( \alpha +1\right) ^{6}}\frac{(\widehat{%
			\rho }_{R,\alpha }-\rho _{0})^{2}}{(1-\rho _{0}^{2})^{2}},  \label{TS2b}
\end{equation}%
is directly obtained from the same proof of $W_{n,\alpha =0}(\widehat{%
	\boldsymbol{\theta }}_{R,\alpha =0})$, since in (\ref{key}) $\widehat{\rho }%
_{R,\alpha =0}$\ can be replaced by $\rho _{0}$. In the particular case of
fixing $\rho _{0}=0$ under the null, (\ref{H0b}), it holds $W_{n,\alpha
	=0}^{\prime }(\widehat{\boldsymbol{\theta }}_{R,\alpha =0})=R_{n,\alpha =0}(%
\widehat{\boldsymbol{\theta }}_{R,\alpha =0})$

\begin{case}
	\label{subSec4}\textbf{(Comparing means and variances of two dependent
		populations with normal distribution)}.\textbf{\medskip }\texttt{\newline
	}If we are interested in testing%
	\begin{equation}
		H_{0}:\mu _{1}=\mu _{2}\text{ and }\sigma _{1}=\sigma _{2},
		\label{testCase2}
	\end{equation}%
	the corresponding Wald-type test statistics based on MRPDEs is%
	\begin{equation}
		\begin{aligned}
			W_{n,\alpha }(\widehat{\boldsymbol{\theta }}_{R,\alpha })=n\frac{\left(
				2\alpha +1\right) ^{2}}{\left( \alpha +1\right) ^{4}}
			\bigg( &\frac{(\widehat{%
					\mu }_{1,R,\alpha }-\widehat{\mu }_{2,R,\alpha })^{2}}{(\widehat{\sigma }%
				_{1,R,\alpha }-\widehat{\sigma }_{1,R,\alpha })^{2}+2(1-\widehat{\rho }%
				_{R,\alpha }^{2})\widehat{\sigma }_{1,R,\alpha },\widehat{\sigma }%
				_{2,R,\alpha }}\\
			&+\frac{\left( 2\alpha +1\right) (\widehat{\sigma }%
				_{1,R}^{\alpha }-\widehat{\sigma }_{2,R}^{\alpha })^{2}}{\left( \alpha
				+1\right) ^{2}\beta _{\alpha }(\widehat{\boldsymbol{\theta }}_{R,\alpha })}%
			\bigg) ,
		\end{aligned} \label{waldCase2}
	\end{equation}%
	and its asymptotic distribution is a chi-squared distribution with $2$
	degrees of freedom under (\ref{testCase2}).
\end{case}

\begin{case}
	\label{subSec5}\textbf{(Fixing a value for covariance of two normal
		populations)}.\textbf{\medskip }\texttt{\newline
	}If we are interested in testing%
	\begin{equation}
		H_{0}:\sigma _{1}\sigma _{2}\rho =\sigma _{12,0},  \label{testChi}
	\end{equation}%
	where $\sigma _{12,0}\in 
	\mathbb{R}
	$, the corresponding Wald-type test statistics based on MRPDEs is 
	\begin{equation}
		W_{n,\alpha }(\widehat{\boldsymbol{\theta }}_{R,\alpha })=n\frac{\left(
			2\alpha +1\right) ^{3}}{\left( \alpha +1\right) ^{4}}\frac{\left( \widehat{%
				\sigma }_{1,R,\alpha }\widehat{\sigma }_{2,R,\alpha }\widehat{\rho }%
			_{R,\alpha }-\sigma _{12,0}\right) ^{2}}{\widehat{\sigma }_{1,R,\alpha }^{2}%
			\widehat{\sigma }_{2,R,\alpha }^{2}\left[ \left( \alpha +1\right) ^{2}(%
			\widehat{\rho }_{R,\alpha }^{2}+1)+\frac{\alpha ^{2}}{2}\widehat{\rho }%
			_{R,\alpha }^{2}\right] }.  \label{chi}
	\end{equation}%
	The asymptotic distribution of (\ref{chi}) is a chi-squared distribution
	with $1$ degree of freedom under (\ref{testChi}).
\end{case}

\begin{case}
	\label{subSec6}\textbf{(Fixing values for means of two dependent populations
		with normal distribution)}.\textbf{\medskip }\texttt{\newline
	}If we are interested in testing%
	\begin{equation}
		H_{0}:\mu _{1}=\mu _{1,0}\text{ and }\mu _{2}=\mu _{2,0}\text{,}
		\label{testw}
	\end{equation}%
	the corresponding Wald-type test statistics based on MRPDEs is%
	\begin{equation}
		\begin{aligned}
			W_{n,\alpha }(\widehat{\boldsymbol{\theta }}_{R,\alpha })=n\dfrac{\left(
				2\alpha +1\right) ^{2}}{\left( \alpha +1\right) ^{4}}\bigg[
			&\frac{\left( \dfrac{%
					\widehat{\mu }_{1,R,\alpha }-\mu _{1,0}}{\widehat{\sigma }_{1,R,\alpha }}-%
				\dfrac{\widehat{\mu }_{2,R,\alpha }-\mu _{2,0}}{\widehat{\sigma }%
					_{2,R,\alpha }}\right) ^{2}}{1-%
				\widehat{\rho }_{R,\alpha }^{2}}\\
			&\frac{+2(1-\widehat{\rho }_{R,\alpha })\left(\frac{\widehat{%
						\mu}_{1,R,\alpha }-\mu _{1,0}}{\widehat{\sigma }_{1,R,\alpha }}\right) \left(\frac{%
					\widehat{\mu }_{2,R,\alpha }-\mu _{2,0}}{\widehat{\sigma }_{2,R,\alpha }}\right)}{1-%
				\widehat{\rho }_{R,\alpha }^{2}} \bigg]
		\end{aligned}\label{w}
	\end{equation}%
	The asymptotic distribution of (\ref{w}) is a chi squared distribution with $%
	2$ degrees of freedom under (\ref{testw}).
\end{case}

\begin{case}
	\label{SecVarCov}\textbf{(Fixing values for variances and covariance of two
		dependent populations with normal distribution)}.\textbf{\medskip }\texttt{%
		\newline
	}If we are interested in testing%
	\begin{equation}
		H_{0}:\sigma _{1}=\sigma _{1,0}\text{, }\sigma _{2}=\sigma _{2,0},\text{ }%
		\sigma _{1}\sigma _{2}\rho =\sigma _{12,0},  \label{testchi2}
	\end{equation}%
	the corresponding Wald-type test statistics based on MRPDEs is%
	\begin{equation}
		W_{n,\alpha }(\widehat{\boldsymbol{\theta }}_{R,\alpha })=n\frac{(2\alpha
			+1)^{3}}{(\alpha +1)^{4}}\boldsymbol{w}_{\alpha }^{T}(\widehat{\boldsymbol{%
				\theta }}_{R,\alpha })\boldsymbol{V}_{2,\alpha }^{-1}(\widehat{\rho }%
		_{R,\alpha })\boldsymbol{w}_{\alpha }(\widehat{\boldsymbol{\theta }}%
		_{R,\alpha }),  \label{chi2}
	\end{equation}%
	where%
	\begin{align*}
		\boldsymbol{w}_{\alpha }(\widehat{\boldsymbol{\theta }}_{R,\alpha }) &=%
		\boldsymbol{D}_{2,\widehat{\sigma }_{1,R,\alpha },\widehat{\sigma }%
			_{2,R,\alpha }}^{-1}\left( \boldsymbol{M}_{22}^{T}(\boldsymbol{\theta }%
		)\right) ^{-1}\boldsymbol{m}(\widehat{\boldsymbol{\theta }}_{R,\alpha })\\
		&=%
		\begin{pmatrix}
			1-\frac{\sigma _{1,0}}{\widehat{\sigma }_{1,R,\alpha }} \\ 
			1-\frac{\sigma _{2,0}}{\widehat{\sigma }_{2,R,\alpha }} \\ 
			\widehat{\rho }_{R,\alpha }-\frac{\sigma _{12,0}}{\widehat{\sigma }%
				_{1,R,\alpha }\widehat{\sigma }_{2,R,\alpha }}-\widehat{\rho }_{R,\alpha
			}\left( 2-\frac{\sigma _{1,0}}{\widehat{\sigma }_{1,R,\alpha }}-\frac{\sigma
				_{2,0}}{\widehat{\sigma }_{2,R,\alpha }}\right)%
		\end{pmatrix}%
		,
	\end{align*}%
	with 
	\begin{equation*}
		\boldsymbol{V}_{2,\alpha }^{-1}(\widehat{\rho }_{R,\alpha })=(\alpha +1)^{2}%
		\boldsymbol{S}_{2,1}(\widehat{\rho }_{R,\alpha })\left[ \boldsymbol{S}_{2,1}(%
		\widehat{\rho }_{R,\alpha })+\left( \frac{\alpha }{\alpha +1}\right) ^{2}%
		\boldsymbol{S}_{2,2}(\widehat{\rho }_{R,\alpha })\right] ^{-1}\boldsymbol{S}%
		_{2,1}(\widehat{\rho }_{R,\alpha }),
	\end{equation*}%
	$\boldsymbol{S}_{2,1}(\widehat{\rho }_{R,\alpha })$ given by (\ref{s2b}) and 
	$\boldsymbol{S}_{2,2}(\widehat{\rho }_{R,\alpha })$\ by (\ref{s2bb}). The
	asymptotic distribution of (\ref{chi2}) is a chi squared with $3$ degrees of
	freedom under (\ref{testchi2}).
\end{case}

\section{Study of the Influence Function}

In the precedent sections we have developed the MRPDE for $\boldsymbol{%
	\theta }$ in the bidimensional normal model, as well as Wald-type tests
based on MRPDE, as a robust alternative to the MLE and classic Wald-type
tests. In this section, we will theoretically justify the robustness of the
proposed estimators through the study of its Influence Function (IF). The IF
(\cite{hampel1986}) for any estimator defined in terms of an statistical
functional $\boldsymbol{T}(F)$ from the true distribution $F$, is defined as

\begin{equation}
	\mathcal{IF}(t,\boldsymbol{T},F)=\lim_{\varepsilon \downarrow 0}\frac{%
		\boldsymbol{T}((1-\varepsilon )F+\varepsilon \Delta _{t})-\boldsymbol{T}(F)}{%
		\varepsilon }=\left. \frac{\partial \boldsymbol{T}(F_{\varepsilon })}{%
		\partial \varepsilon }\right\vert _{\varepsilon =0^{+}},  \label{eq:robust_1}
\end{equation}%
with $\varepsilon $ being the contamination proportion and $\Delta _{t}$
being the degenerate distribution at the contamination point $t$. Thus, the
IF, as a function of $t$, measures the standardized asymptotic bias caused
by the infinitesimal contamination at the point $t$. The maximum of this IF
over $t$ indicates the extent of bias due to contamination and so smaller
its value, the more robust the estimator is. Note that, in this context, the
statistical functional $\boldsymbol{T}_{\alpha }$ corresponding to the MRPDE
is defined as the minimizer of $R_{\alpha }(f_{\boldsymbol{\theta }},g)$ in (%
\ref{1.1}).


IF for MLE in the bidimensional normal model has been widely studied in
literature. For example, \cite{Devlin1975} presented
the IF for the Pearson's correlation coefficient $\rho$. A proof was given
years later by \cite{Chernick1982}. IFs for the mean and variance can be found,
in \cite{Radhakrishnan} and \cite{isogai1989}, among
others.

In \cite{Broniatowski2012} (Theorem 5), the IF of Renyi's
pseudodistances-based estimators was provided in a general form and
particularized to some particular models. \cite{castilla2021} generalized
this result to the case of independent not identically distributed
observations. Based on these results, in Theorem \ref{th:IF} we present the
IF associated to the MRPDE of $\boldsymbol{\theta }$ the bidimensional
normal model. A detailed proof of the following result is provided in
Appendix \ref{A4}.

\begin{theorem}
	\label{th:IF} Let us consider the bidimensional normal model (\ref{model}).
	The IF associated to the MRPDE of $\boldsymbol{\theta }$ is given by 
	\begin{equation*}
		\mathcal{IF}((x,y)^{T},\boldsymbol{T}_{\alpha },F_{\boldsymbol{\theta }})=(%
		\mathcal{IF}_{\alpha }(\mu _{1}),\mathcal{IF}_{\alpha }(\mu _{2}),\mathcal{IF%
		}_{\alpha }(\sigma _{1}),\mathcal{IF}_{\alpha }(\sigma _{2}),\mathcal{IF}%
		_{\alpha }(\rho ))^{T},
	\end{equation*}%
	where%
	\begin{align}
		\mathcal{IF}_{\alpha }(\mu _{1})& =(\alpha +1)^{2}w_{\boldsymbol{\theta }%
		}^{-\alpha /(1-\rho ^{2})}(x,y)(x-\mu _{1}),  \label{IF1} \\
		\mathcal{IF}_{\alpha }(\mu _{2})& =(\alpha +1)^{2}w_{\boldsymbol{\theta }%
		}^{-\alpha /(1-\rho ^{2})}(x,y)(y-\mu _{2}),  \label{IF2}
	\end{align}%
	\begin{align}
		\mathcal{IF}_{\alpha }(\sigma _{1})& =\frac{\left( \alpha +1\right) ^{3}}{2}%
		w_{\boldsymbol{\theta }}^{-\alpha }(x,y)\tfrac{\sigma _{1}}{1-\rho ^{2}}%
		\left[ \left( \tfrac{x-\mu _{1}}{\sigma _{1}}\right) ^{2}-\rho ^{2}\left( 
		\tfrac{y-\mu _{2}}{\sigma _{2}}\right) ^{2}-(1-\rho ^{2})(1+2\rho ^{2})%
		\tfrac{1}{\alpha +1}\right] ,  \label{IF3} \\
		\mathcal{IF}_{\alpha }(\sigma _{2})& =\frac{\left( \alpha +1\right) ^{3}}{2}%
		w_{\boldsymbol{\theta }}^{-\alpha }(x,y)\tfrac{\sigma _{2}}{1-\rho ^{2}}%
		\left[ \left( \tfrac{y-\mu _{2}}{\sigma _{2}}\right) ^{2}-\rho ^{2}\left( 
		\tfrac{x-\mu _{1}}{\sigma _{1}}\right) ^{2}-(1-\rho ^{2})(1+2\rho ^{2})%
		\tfrac{1}{\alpha +1}\right] ,  \label{IF4} \\
		\mathcal{IF}_{\alpha }(\rho )& =\left( \alpha +1\right) ^{3}w_{\boldsymbol{%
				\theta }}^{-\alpha }(x,y)\left\{ -\tfrac{\rho }{2}\left[ \left( \tfrac{x-\mu
			_{1}}{\sigma _{1}}\right) ^{2}+\left( \tfrac{x-\mu _{1}}{\sigma _{1}}\right)
		^{2}\right] +\tfrac{x-\mu _{1}}{\sigma _{1}}\tfrac{y-\mu _{2}}{\sigma _{2}}%
		\right\} ,  \label{IF5}
	\end{align}%
	with%
	\begin{equation*}
		w_{\boldsymbol{\theta }}(x,y)=\exp \left\{ \tfrac{1}{2(1-\rho ^{2})}\left[ (%
		\tfrac{x-\mu _{1}}{\sigma _{1}})^{2}+(\tfrac{y-\mu _{2}}{\sigma _{2}}%
		)^{2}-2\rho \tfrac{x-\mu _{1}}{\sigma _{1}}\tfrac{y-\mu _{2}}{\sigma _{2}}%
		\right] \right\} .
	\end{equation*}
\end{theorem}

\begin{remark}
	In particular, for $\alpha =0$ (MLE),%
	\begin{align*}
		\mathcal{IF}_{\alpha =0}(\mu _{1})& =x-\mu _{1}, \\
		\mathcal{IF}_{\alpha =0}(\mu _{2})& =y-\mu _{2},
	\end{align*}%
	\begin{align*}
		\mathcal{IF}_{\alpha =0}(\sigma _{1})& =\tfrac{\sigma _{1}}{2(1-\rho ^{2})}%
		\left[ \left( \tfrac{x-\mu _{1}}{\sigma _{1}}\right) ^{2}-\rho ^{2}\left( 
		\tfrac{y-\mu _{2}}{\sigma _{2}}\right) ^{2}\right] -\tfrac{\sigma _{1}}{2}%
		(1+2\rho ^{2}), \\
		\mathcal{IF}_{\alpha =0}(\sigma _{2})& =\tfrac{\sigma _{2}}{2(1-\rho ^{2})}%
		\left[ \left( \tfrac{y-\mu _{2}}{\sigma _{2}}\right) ^{2}-\rho ^{2}\left( 
		\tfrac{x-\mu _{1}}{\sigma _{1}}\right) ^{2}\right] -\tfrac{\sigma _{1}}{2}%
		(1+2\rho ^{2}),
	\end{align*}%
	\begin{equation*}
		\mathcal{IF}_{\alpha =0}(\rho )=-\tfrac{\rho }{2}\left[ \left( \tfrac{x-\mu
			_{1}}{\sigma _{1}}\right) ^{2}+\left( \tfrac{y-\mu _{2}}{\sigma _{2}}\right)
		^{2}\right] +\tfrac{x-\mu _{1}}{\sigma _{1}}\tfrac{y-\mu _{2}}{\sigma _{2}}.
	\end{equation*}
	
	The IF (\ref{eq:IF}) is bounded for positive values of the parameter $\alpha 
	$, $\alpha >0$, and is unbounded at the MLE, $\alpha =0$.
\end{remark}

\begin{figure}[]
	\centering
	\begin{tabular}{cc}
		\includegraphics[scale=0.3]{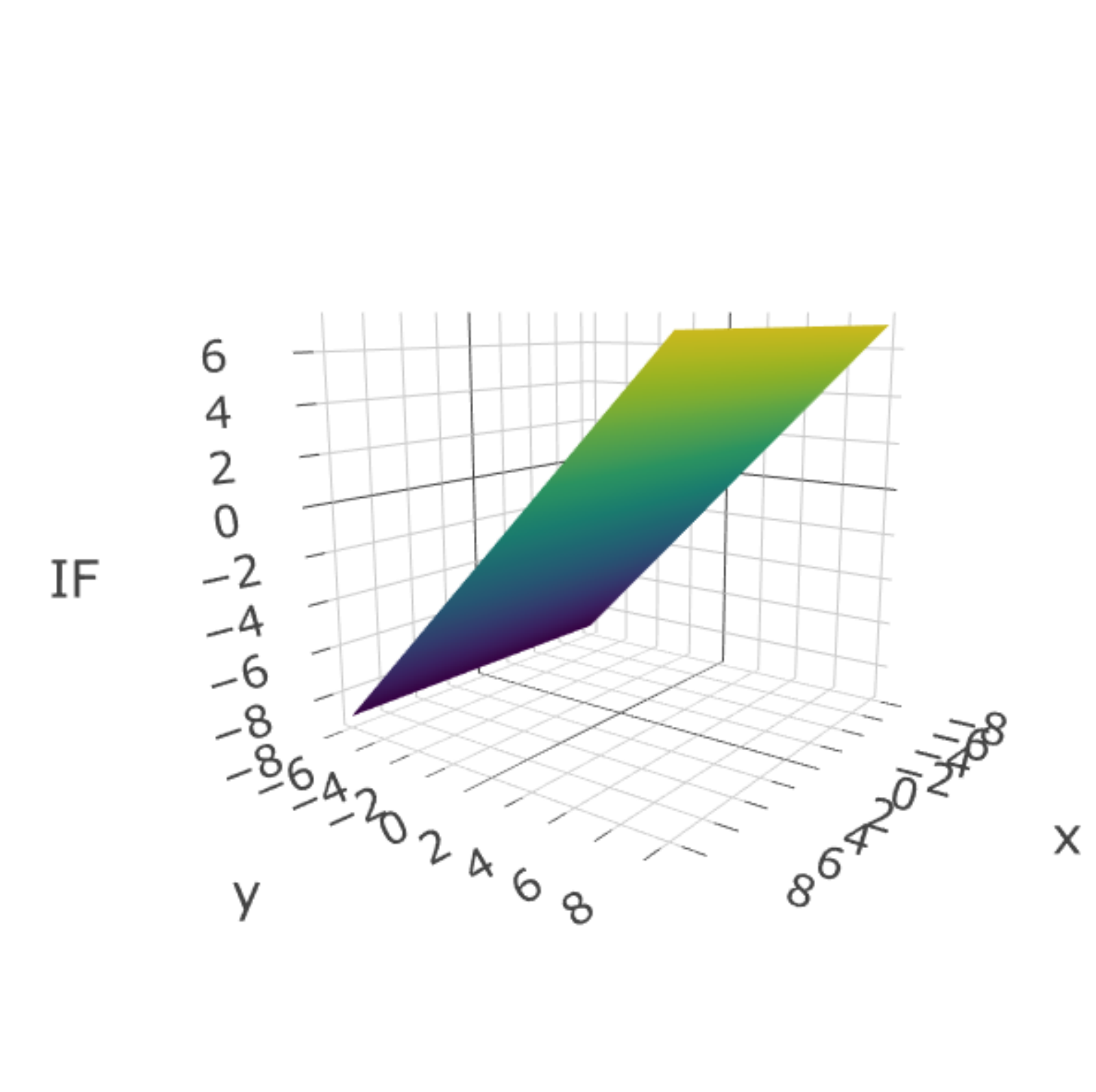} & %
		\includegraphics[scale=0.3]{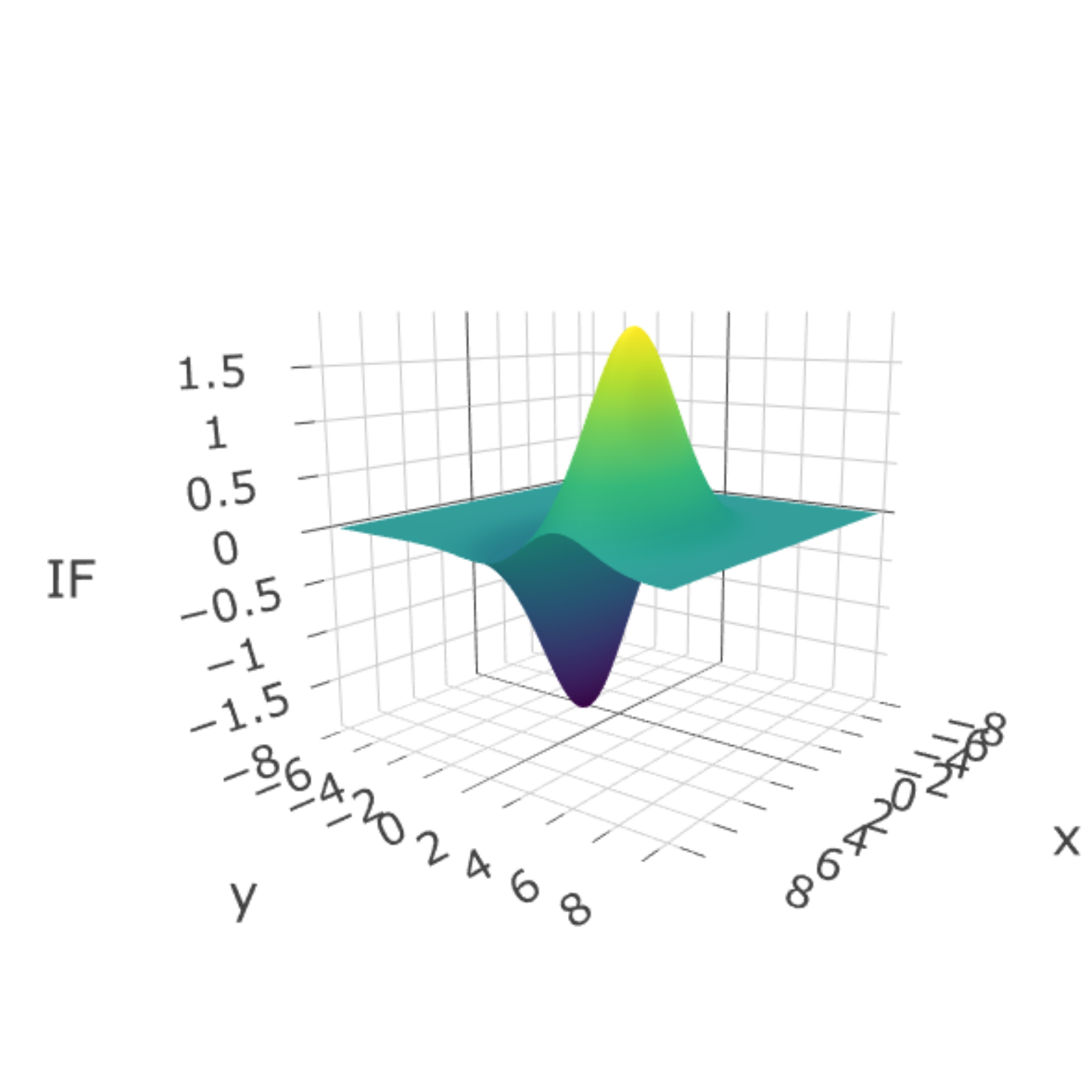} \\ 
		\includegraphics[scale=0.3]{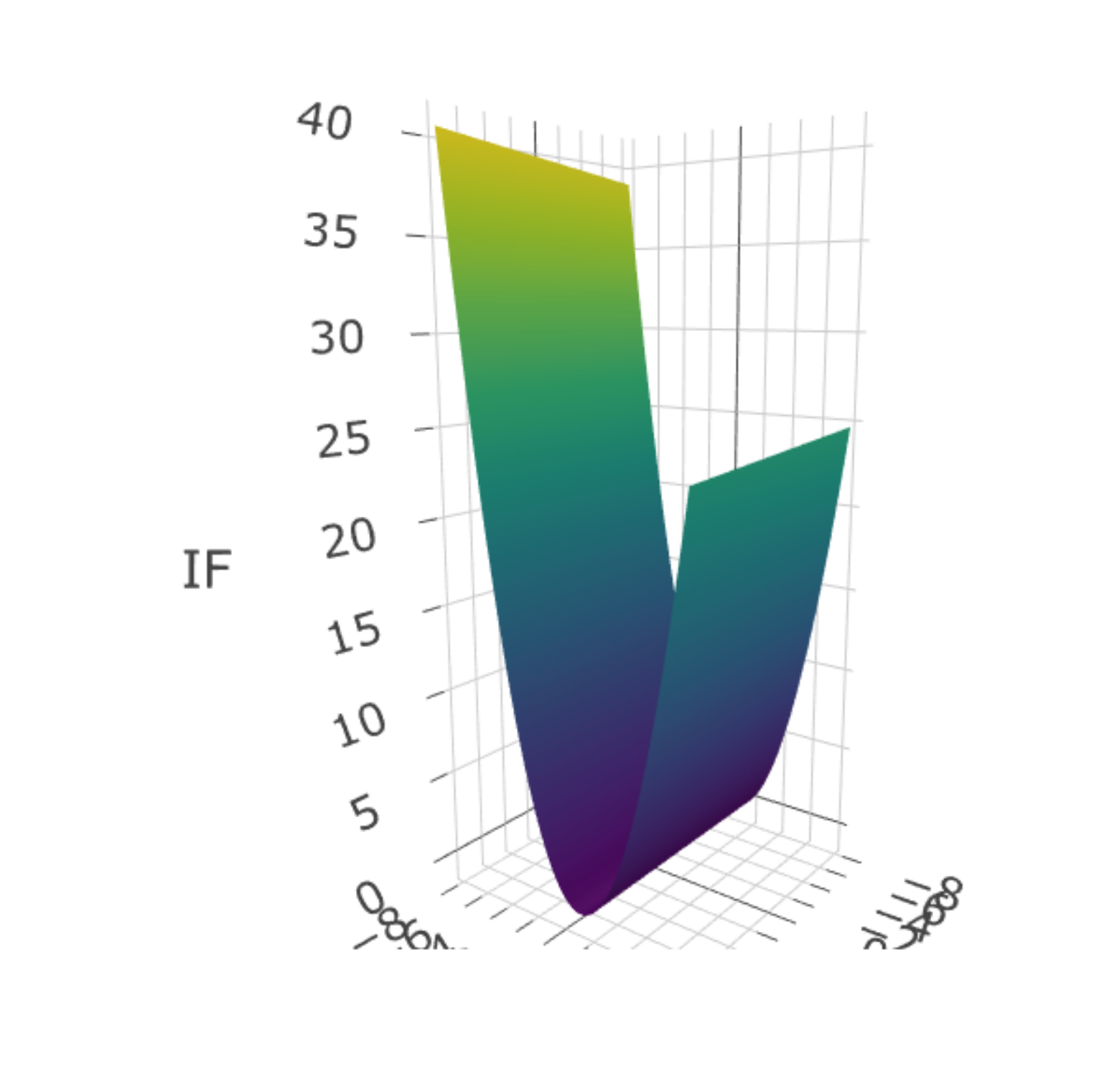} & %
		\includegraphics[scale=0.3]{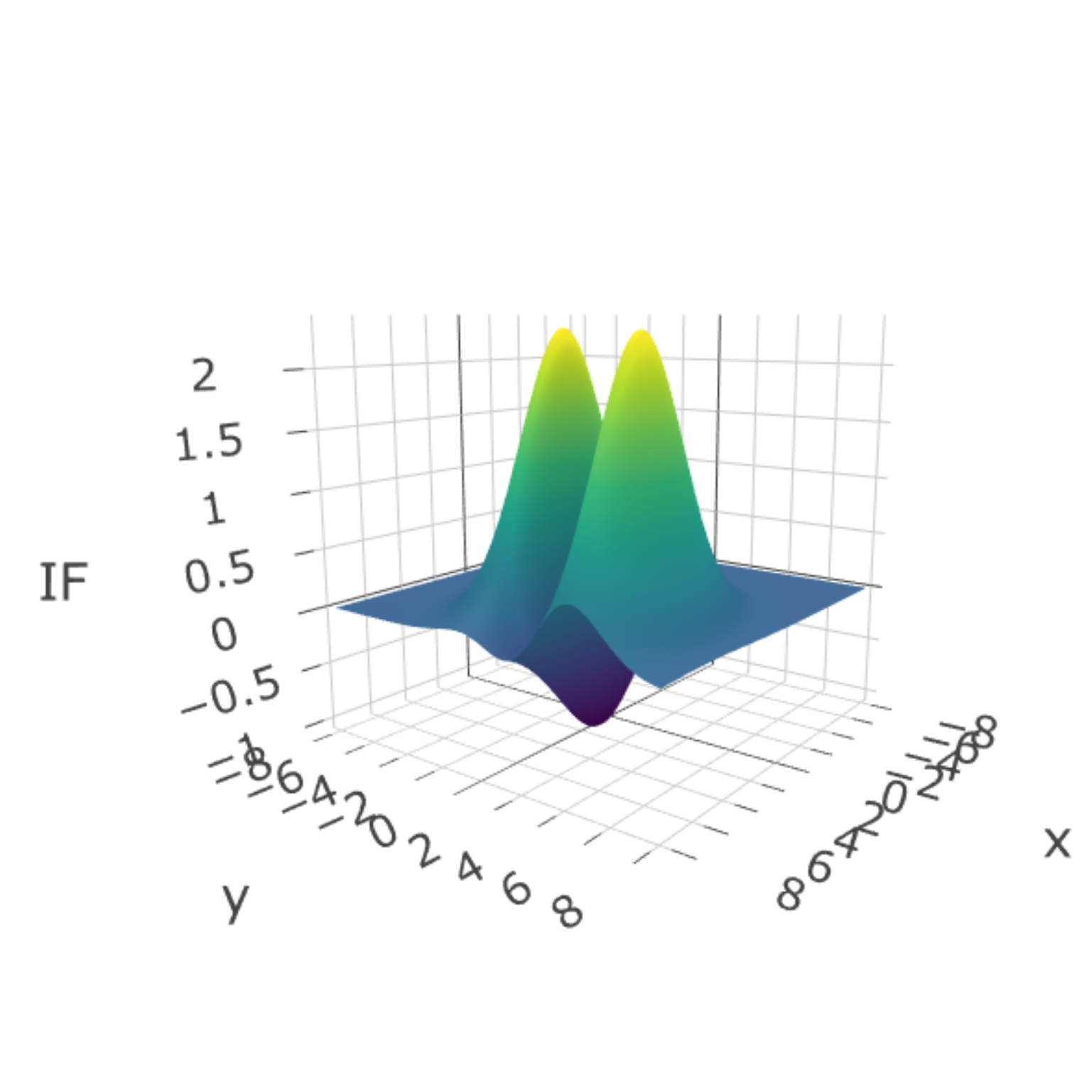} \\ 
		\includegraphics[scale=0.3]{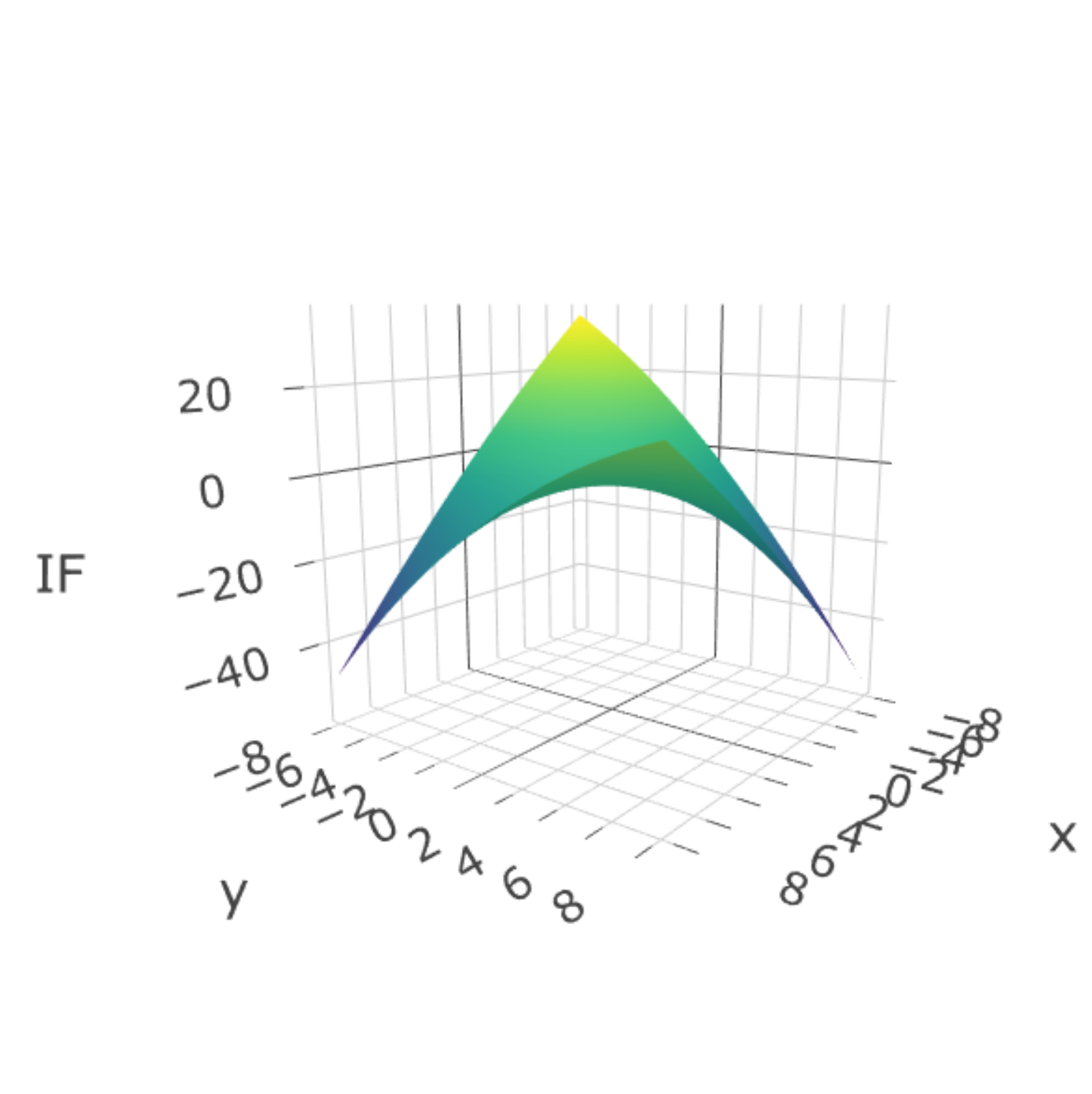} & %
		\includegraphics[scale=0.3]{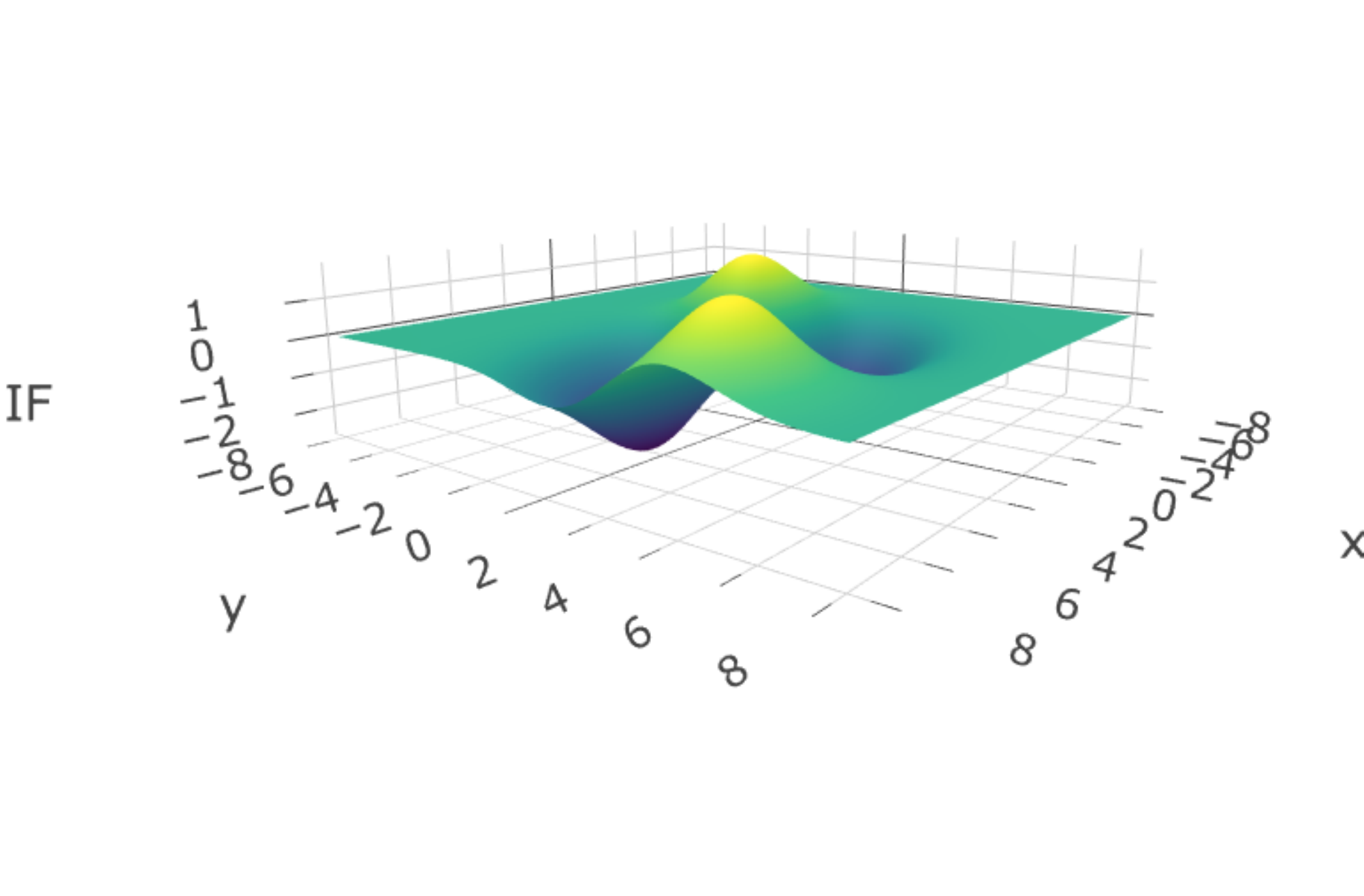}%
	\end{tabular}%
	\caption{$\mathcal{IF}_{\protect\alpha }(\protect\mu _{1})$ (above), $%
		\mathcal{IF}_{\protect\alpha }(\protect\sigma _{1})$ (middle) and $\mathcal{%
			IF}_{\protect\alpha }(\protect\rho )$ (below) for $\protect\alpha =0$ (left)
		and $\protect\alpha =0.3$ (right), with $\boldsymbol{\protect\theta }%
		=(1,2,1,1.5,0.3)^{T}$. }
	\label{fig:IF}
\end{figure}

Once we have computed the IF for the minimum RP estimators, we can define
and study the IF for the Wald-type test statistics defined in (\ref{4.2}).
As noted by \cite{castilla2021}, when the corresponding IF is identically
zero and is therefore necessary to consider the second order IF of the
proposed Wald-type tests functional $W_{\alpha}$.

\begin{theorem}
	Let us consider the bidimensional normal model (\ref{model}). The second
	order IF of the proposed Wald-type test functionals for testing simple null
	hypothesis in (\ref{4.1}) is given by 
	\begin{align}
		& \mathcal{IF}_{2}((x,y)^{T},W_{\alpha },F_{\boldsymbol{\theta }})  \notag \\
		& =2(\mathcal{IF}((x,y)^{T},\boldsymbol{T}_{\alpha },F_{\boldsymbol{\theta }%
		}))^{T}\boldsymbol{M}(\boldsymbol{\theta })\left( \boldsymbol{M}^{T}(%
		\boldsymbol{\theta })\boldsymbol{V}_{\alpha }(\boldsymbol{\theta })%
		\boldsymbol{M}(\boldsymbol{\theta })\right) ^{-1}\boldsymbol{M}^{T}(%
		\boldsymbol{\theta })\mathcal{IF}((x,y)^{T},\boldsymbol{T}_{\alpha },F_{%
			\boldsymbol{\theta }}),  \label{eq:IF}
	\end{align}%
	where $\mathcal{IF}((x,y)^{T},\boldsymbol{T}_{\alpha },F_{\boldsymbol{\theta 
	}})$ is given in Theorem \ref{th:IF}.
\end{theorem}

Note that the second-order IF of the proposed Wald-type tests is a quadratic
function of the corresponding IF of the MRPDE. Therefore, the boundedness of
the IF of MRPDE at $\alpha >0$ also indicates the boundedness of the IF of
the Wald-type test functionals, implying its robustness against
contamination.

\begin{figure}[]
	\centering
	\begin{tabular}{cc}
		\includegraphics[scale=0.33]{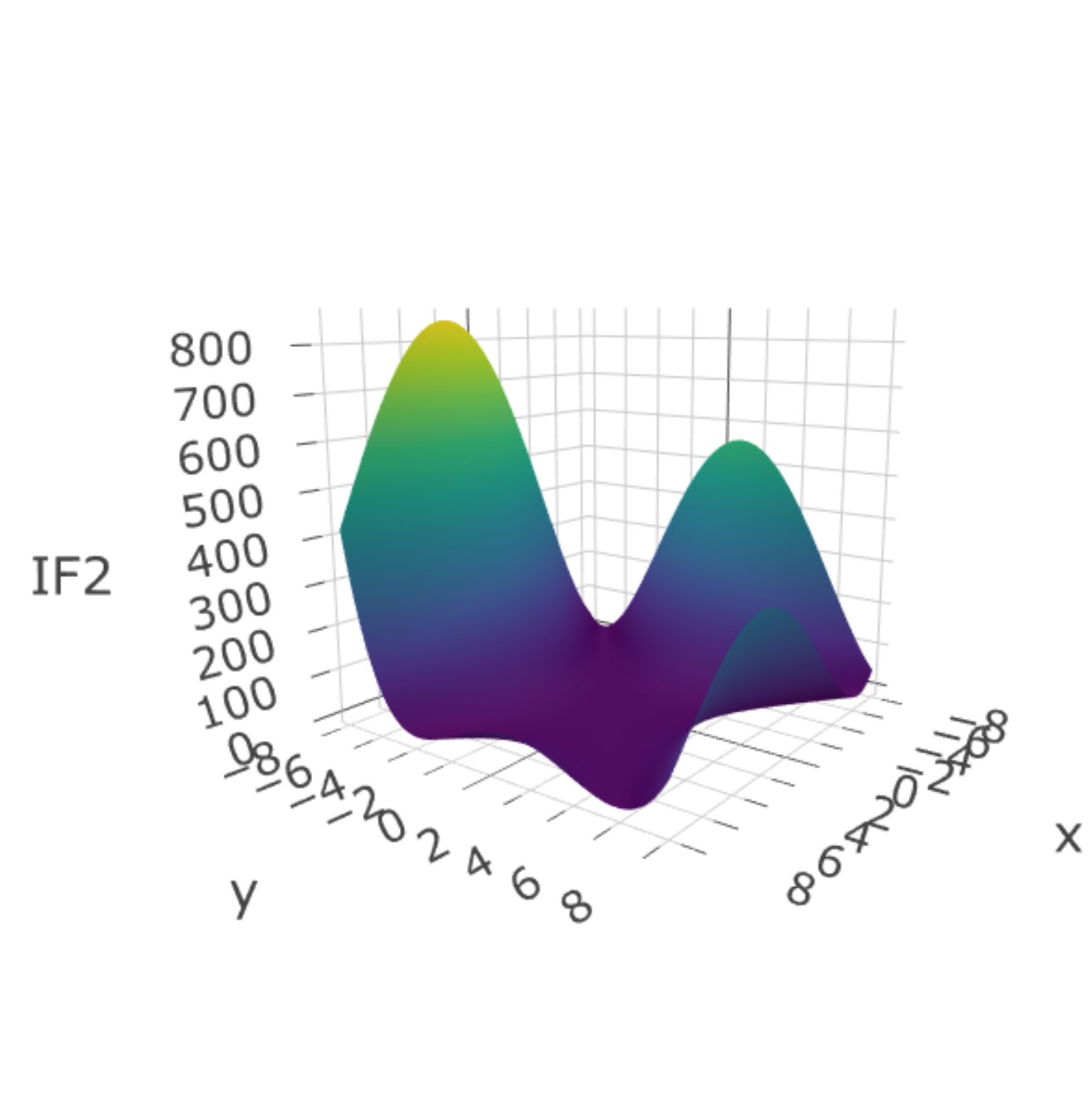} & %
		\includegraphics[scale=0.33]{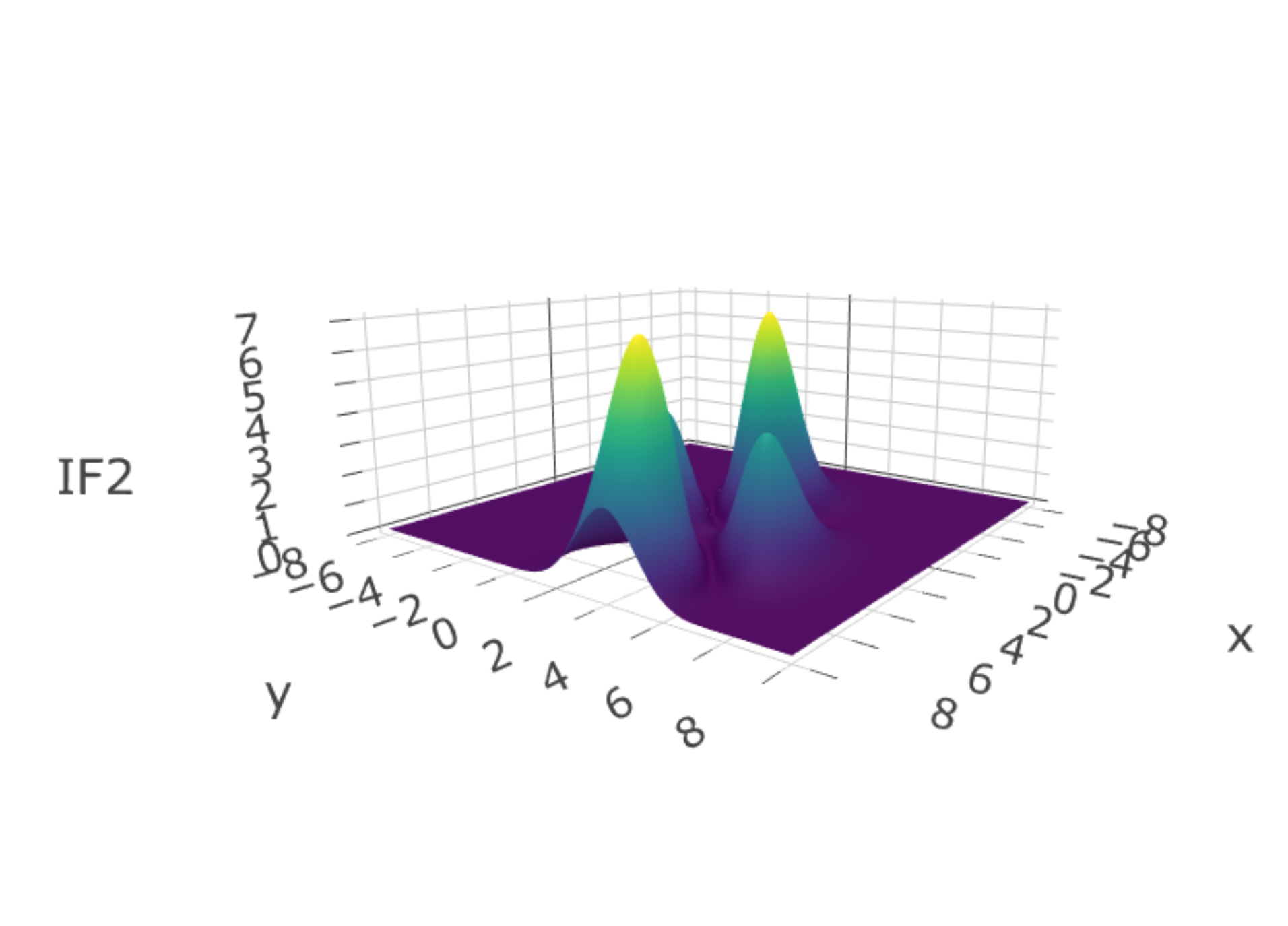}%
	\end{tabular}%
	\caption{$\mathcal{IF}_{2}((x,y)^{T},W_{\protect\alpha },F_{\boldsymbol{%
				\protect\theta }})$ for testing $H_{0}$: $\protect\sigma _{1}=\protect\sigma %
		_{2}$, with $\protect\alpha =0$ (left) and $\protect\alpha =0.3$ (right),
		and $\boldsymbol{\protect\theta }=(1,2,1,1.5,0.3)^{T}$. }
	\label{fig:IF2}
\end{figure}

\newpage

\section{Simulation study\label{sim}}

It is well-known that the Morgan-Pitman test is best unbiased and best
invariant test statistic for testing equality of variances (see \cite{Morgan1939}, \cite{Pitman1939} and \cite{hsul1940}). The idea behind the Morgan-Pitman
test allows us to include for testing equality of variances not only the
family of test statistics given in Case \ref{subSec2} but also the family of
test statistics given in Case \ref{subSec3} for the data $%
(u_{1},v_{1}),\ldots ,(u_{n},v_{n})$, where $U=X+Y$ and $V=X-Y$ are
transformed variables. As $\mathrm{Cov}[U,V]=\sigma _{1}^{2}-\sigma _{2}^{2}$%
, testing $H_{0}$: $\sigma _{1}=\sigma _{2}$ given in (\ref{H0}) (Case \ref%
{subSec2}), matches $H_{0}$: $\rho _{UV}=0$, from (\ref{H0b}) with $\rho
_{0}=0$ (Case \ref{subSec3}). We have compared these two families of
test statistics, first (\ref{TS}) in Tables \ref{table3}, \ref{table1}, \ref%
{table5}, \ref{table7}, rewritten as%
\begin{equation}
	W_{n,\alpha }(\widehat{\gamma }_{R,\alpha },\widehat{\rho }_{R,\alpha })=n%
	\frac{\left( 2\alpha +1\right) ^{3}}{\left( \alpha +1\right) ^{6}}\frac{(%
		\widehat{\gamma }_{R,\alpha }-1)^{2}}{\beta _{\alpha }(\widehat{\gamma }%
		_{R,\alpha },\widehat{\rho }_{R,\alpha })},  \label{simW1}
\end{equation}%
where%
\begin{align}
	\widehat{\gamma }_{R,\alpha }& =\frac{\widehat{\sigma }_{1,R,\alpha }}{%
		\widehat{\sigma }_{2,R,\alpha }},  \notag \\
	\beta _{\alpha }(\widehat{\gamma }_{R,\alpha },\widehat{\rho }_{R,\alpha })&
	=\frac{1}{4}\left[ \left( \tfrac{\alpha }{\alpha +1}\right) ^{2}+2\right] (%
	\widehat{\gamma }_{R,\alpha }-1)^{2}+(1-\widehat{\rho }_{R,\alpha }^{2})%
	\widehat{\gamma }_{R,\alpha }.  \label{beta}
\end{align}%
and second (\ref{TS2b}) in Tables \ref{table3b}, \ref{table1b}, \ref{table5b}%
, \ref{table7b}, rewritten as%
\begin{equation}
	W_{n,\alpha }^{\prime }(\widehat{\rho }_{UV,R,\alpha })=n\frac{\left(
		2\alpha +1\right) ^{3}}{\left( \alpha +1\right) ^{6}}\widehat{\rho }%
	_{UV,R,\alpha }^{2}.  \label{simW2}
\end{equation}%
A third one, (\ref{TS2}), was also considered but do not present the results
here as\ the corresponding results were very bad in comparison with (\ref%
{TS2b}). In addition, the exact Morgan-Pitman test,%
\begin{equation}
	T_{MP}=\widehat{\rho }_{UV,R,\alpha =0}\sqrt{\frac{n-2}{1-\widehat{\rho }%
			_{UV,R,\alpha =0}^{2}}},  \label{MorganPitman}
\end{equation}%
is taken into account, whose exact distribution is a Student $t$ with $n-2$
degrees of freedom ($t_{n-2}$) under $H_{0}$, with 
\begin{equation*}
	\widehat{\rho }_{UV,R,\alpha =0}=\frac{\sum_{i=1}^{n}(U_{i}-\bar{U}%
		_{n})(V_{i}-\bar{V}_{n})}{\sqrt{\sum_{i=1}^{n}(U_{i}-\bar{U}_{n})^{2}}\sqrt{%
			\sum_{i=1}^{n}(V_{i}-\bar{V}_{n})^{2}}}=R_{UV},
\end{equation*}%
being the Pearson correlation coefficient, i.e. the MLE of $\rho _{UV}$.
Furthermore, we included the simulated significance level of the Morgan
Pitman test described in (\ref{MorganPitman}) in the aforemetioned tables
and calculated by simulation $\mathrm{MSE}(\widehat{\gamma }_{R,\alpha
})=\left\vert \widehat{\gamma }_{R,\alpha }-1\right\vert $ in Tables \ref%
{table4}, \ref{table2}, \ref{table6}, \ref{table8} as well as $\mathrm{MSE}(%
\widehat{\gamma }_{R,\alpha })=\left\vert \widehat{\rho }_{R,\alpha
}\right\vert $ in Tables \ref{table4b}, \ref{table2b}, \ref{table6b}, \ref%
{table8b}.

So as to evaluate the performance of the proposed Wald-type tests, we
considered the bidimensional normal model (\ref{model}) for the true
parameters values $\mu _{1}=\mu _{2}=0$, $\sigma _{1}=\sigma _{2}=1$ and the
different correlations between the normal variables $\rho \in
\{0,0.\allowbreak 3,0.\allowbreak 6,0.\allowbreak 9\}.$ Additionally, in
order to evaluate the robustness of the Wald-type tests, we analysed ten
different scenarios of contamination:

\begin{itemize}
	\item Pure data
	
	\item Slightly contaminated data : We replace a $5\%,10\%$ and $20\%$ of the
	samples by a bidimensional normal distribution, substituting the true
	parameter values $\sigma _{1}^{\prime }=\sigma _{2}^{\prime }=1$ by $\sigma
	_{1}^{\prime }=\sigma _{2}^{\prime }=\sqrt{3}.$
	
	\item Contaminated data : We replace a $5\%,10\%$ and $20\%$ of the samples
	by a bidimensional Student $t$ distribution with $d=5$ degrees of freedom.
	
	\item Heavily contaminated data : We replace a $5\%,10\%$ and $20\%$ of the
	samples by a bidimensional normal distribution, substituting the true
	parameter value $\sigma _{2}^{\prime }=1$ by $\sigma _{2}^{\prime }=5.$
\end{itemize}

\begin{center}
	\begin{table}[htbp]  \tabcolsep2.8pt  \centering%
		\begin{tabular}{llllllllllll}
			\hline
			&  &  & \multicolumn{3}{c}{slightly} & \multicolumn{3}{c}{regular} & 
			\multicolumn{3}{c}{heavily} \\ 
			$\rho $ & $\alpha $ & pure & $0.05$ & $0.10$ & $0.20$ & $0.05$ & $0.10$ & $%
			0.20$ & $0.05$ & $0.10$ & $0.20$ \\ \hline
			$0$ & $0$ & 0.169 & 0.175 & 0.182 & 0.190 & 0.173 & 0.179 & 0.186 & 0.279 & 
			0.380 & 0.522 \\ 
			& $0.1$ & 0.170 & 0.173 & 0.178 & 0.184 & 0.172 & 0.175 & 0.178 & 0.209 & 
			0.281 & 0.432 \\ 
			& $0.2$ & 0.176 & 0.177 & 0.182 & 0.187 & 0.178 & 0.179 & 0.181 & 0.187 & 
			0.220 & 0.320 \\ 
			& $0.3$ & 0.187 & 0.187 & 0.192 & 0.198 & 0.188 & 0.189 & 0.192 & 0.191 & 
			0.209 & 0.267 \\ 
			& $0.5$ & 0.223 & 0.224 & 0.229 & 0.236 & 0.228 & 0.227 & 0.230 & 0.226 & 
			0.234 & 0.264 \\ 
			& $0.7$ & 0.307 & 0.327 & 0.313 & 0.320 & 0.304 & 0.384 & 0.313 & 0.330 & 
			0.317 & 0.351 \\ \hline
			$0.3$ & $0$ & 0.161 & 0.169 & 0.175 & 0.182 & 0.164 & 0.168 & 0.176 & 0.278
			& 0.376 & 0.523 \\ 
			& $0.1$ & 0.162 & 0.167 & 0.170 & 0.176 & 0.162 & 0.164 & 0.169 & 0.204 & 
			0.273 & 0.430 \\ 
			& $0.2$ & 0.168 & 0.171 & 0.174 & 0.180 & 0.168 & 0.168 & 0.173 & 0.181 & 
			0.210 & 0.313 \\ 
			& $0.3$ & 0.178 & 0.181 & 0.184 & 0.189 & 0.177 & 0.177 & 0.183 & 0.184 & 
			0.199 & 0.256 \\ 
			& $0.5$ & 0.213 & 0.216 & 0.219 & 0.225 & 0.214 & 0.213 & 0.218 & 0.218 & 
			0.226 & 0.252 \\ 
			& $0.7$ & 0.290 & 0.290 & 0.304 & 0.329 & 0.298 & 0.289 & 0.295 & 0.290 & 
			0.312 & 0.353 \\ \hline
			$0.6$ & $0$ & 0.133 & 0.142 & 0.146 & 0.150 & 0.139 & 0.142 & 0.147 & 0.268
			& 0.371 & 0.522 \\ 
			& $0.1$ & 0.134 & 0.140 & 0.142 & 0.146 & 0.137 & 0.139 & 0.140 & 0.178 & 
			0.252 & 0.417 \\ 
			& $0.2$ & 0.139 & 0.143 & 0.145 & 0.149 & 0.141 & 0.143 & 0.143 & 0.154 & 
			0.181 & 0.288 \\ 
			& $0.3$ & 0.147 & 0.151 & 0.154 & 0.157 & 0.149 & 0.152 & 0.152 & 0.156 & 
			0.169 & 0.227 \\ 
			& $0.5$ & 0.176 & 0.179 & 0.184 & 0.189 & 0.176 & 0.181 & 0.183 & 0.183 & 
			0.191 & 0.220 \\ 
			& $0.7$ & 0.241 & 0.248 & 0.253 & 0.254 & 0.237 & 0.245 & 0.244 & 0.250 & 
			0.268 & 0.289 \\ \hline
			$0.9$ & $0$ & 0.074 & 0.077 & 0.079 & 0.081 & 0.075 & 0.077 & 0.081 & 0.240
			& 0.367 & 0.525 \\ 
			& $0.1$ & 0.074 & 0.076 & 0.077 & 0.078 & 0.075 & 0.075 & 0.077 & 0.099 & 
			0.162 & 0.348 \\ 
			& $0.2$ & 0.077 & 0.078 & 0.079 & 0.080 & 0.077 & 0.077 & 0.078 & 0.082 & 
			0.098 & 0.170 \\ 
			& $0.3$ & 0.081 & 0.082 & 0.084 & 0.084 & 0.081 & 0.081 & 0.083 & 0.085 & 
			0.092 & 0.120 \\ 
			& $0.5$ & 0.097 & 0.098 & 0.101 & 0.102 & 0.099 & 0.097 & 0.099 & 0.101 & 
			0.107 & 0.121 \\ 
			& $0.7$ & 0.131 & 0.131 & 0.136 & 0.143 & 0.137 & 0.130 & 0.135 & 0.136 & 
			0.148 & 0.204 \\ \hline
		\end{tabular}
		\caption{Simulated mean square error of the MRPDE for ratio of variances,
			$\protect\widehat{\gamma}_{R,\alpha}$, when $n=25$\label{table4}} 
	\end{table}%
	
	\begin{table}[htbp]  \tabcolsep2.8pt  \centering%
		\begin{tabular}{llllllllllll}
			\hline
			&  &  & \multicolumn{3}{c}{slightly} & \multicolumn{3}{c}{regular} & 
			\multicolumn{3}{c}{heavily} \\ 
			$\rho $ & $\alpha $ & pure & $0.05$ & $0.10$ & $0.20$ & $0.05$ & $0.10$ & $%
			0.20$ & $0.05$ & $0.10$ & $0.20$ \\ \hline
			$0$ & $0$ & 0.059 & 0.070 & 0.081 & 0.093 & 0.067 & 0.074 & 0.089 & 0.352 & 
			0.585 & 0.852 \\ 
			& $0.1$ & 0.058 & 0.063 & 0.072 & 0.079 & 0.062 & 0.064 & 0.069 & 0.172 & 
			0.362 & 0.699 \\ 
			& $0.2$ & 0.059 & 0.064 & 0.069 & 0.075 & 0.063 & 0.064 & 0.066 & 0.094 & 
			0.180 & 0.438 \\ 
			& $0.3$ & 0.064 & 0.068 & 0.071 & 0.079 & 0.067 & 0.067 & 0.072 & 0.081 & 
			0.124 & 0.276 \\ 
			& $0.5$ & 0.085 & 0.091 & 0.094 & 0.103 & 0.093 & 0.089 & 0.097 & 0.100 & 
			0.115 & 0.182 \\ 
			& $0.7$ & 0.144 & 0.148 & 0.155 & 0.167 & 0.154 & 0.146 & 0.159 & 0.157 & 
			0.173 & 0.217 \\ \hline
			& \multicolumn{1}{c}{MP} & 0.051 & 0.062 & 0.071 & 0.082 & 0.059 & 0.065 & 
			0.080 & 0.341 & 0.572 & 0.844 \\ \hline
			$0.3$ & $0$ & 0.061 & 0.075 & 0.082 & 0.092 & 0.062 & 0.072 & 0.086 & 0.365
			& 0.596 & 0.860 \\ 
			& $0.1$ & 0.060 & 0.066 & 0.071 & 0.078 & 0.056 & 0.060 & 0.069 & 0.179 & 
			0.361 & 0.708 \\ 
			& $0.2$ & 0.061 & 0.064 & 0.067 & 0.075 & 0.058 & 0.058 & 0.068 & 0.097 & 
			0.181 & 0.433 \\ 
			& $0.3$ & 0.065 & 0.066 & 0.070 & 0.079 & 0.062 & 0.063 & 0.073 & 0.081 & 
			0.123 & 0.271 \\ 
			& $0.5$ & 0.089 & 0.093 & 0.097 & 0.104 & 0.088 & 0.088 & 0.098 & 0.101 & 
			0.122 & 0.181 \\ 
			& $0.7$ & 0.146 & 0.152 & 0.156 & 0.172 & 0.148 & 0.143 & 0.154 & 0.158 & 
			0.181 & 0.223 \\ \hline
			& \multicolumn{1}{c}{MP} & 0.052 & 0.063 & 0.071 & 0.082 & 0.054 & 0.062 & 
			0.075 & 0.352 & 0.584 & 0.852 \\ \hline
			$0.6$ & $0$ & 0.057 & 0.072 & 0.081 & 0.090 & 0.067 & 0.075 & 0.090 & 0.394
			& 0.631 & 0.883 \\ 
			& $0.1$ & 0.056 & 0.064 & 0.071 & 0.075 & 0.060 & 0.064 & 0.070 & 0.175 & 
			0.368 & 0.713 \\ 
			& $0.2$ & 0.058 & 0.064 & 0.068 & 0.072 & 0.060 & 0.062 & 0.070 & 0.095 & 
			0.167 & 0.430 \\ 
			& $0.3$ & 0.063 & 0.069 & 0.074 & 0.077 & 0.065 & 0.067 & 0.074 & 0.085 & 
			0.116 & 0.265 \\ 
			& $0.5$ & 0.088 & 0.093 & 0.102 & 0.109 & 0.088 & 0.091 & 0.099 & 0.101 & 
			0.120 & 0.183 \\ 
			& $0.7$ & 0.148 & 0.148 & 0.160 & 0.171 & 0.143 & 0.151 & 0.155 & 0.157 & 
			0.179 & 0.226 \\ \hline
			& \multicolumn{1}{c}{MP} & 0.047 & 0.062 & 0.069 & 0.076 & 0.057 & 0.064 & 
			0.079 & 0.380 & 0.616 & 0.875 \\ \hline
			$0.9$ & $0$ & 0.064 & 0.078 & 0.084 & 0.091 & 0.068 & 0.077 & 0.092 & 0.465
			& 0.719 & 0.932 \\ 
			& $0.1$ & 0.061 & 0.068 & 0.073 & 0.075 & 0.062 & 0.065 & 0.073 & 0.125 & 
			0.282 & 0.647 \\ 
			& $0.2$ & 0.060 & 0.067 & 0.073 & 0.072 & 0.063 & 0.063 & 0.071 & 0.075 & 
			0.111 & 0.280 \\ 
			& $0.3$ & 0.066 & 0.069 & 0.079 & 0.077 & 0.068 & 0.068 & 0.075 & 0.073 & 
			0.094 & 0.162 \\ 
			& $0.5$ & 0.089 & 0.093 & 0.105 & 0.110 & 0.097 & 0.092 & 0.102 & 0.100 & 
			0.118 & 0.146 \\ 
			& $0.7$ & 0.147 & 0.151 & 0.165 & 0.171 & 0.157 & 0.153 & 0.163 & 0.158 & 
			0.185 & 0.217 \\ \hline
			& \multicolumn{1}{c}{MP} & 0.051 & 0.064 & 0.069 & 0.077 & 0.057 & 0.063 & 
			0.078 & 0.448 & 0.706 & 0.927 \\ \hline
		\end{tabular}
		\caption{Simulated significance level for testing equal variances through
			$W_{n,\alpha }(\widehat{\gamma }_{R,\alpha },\widehat{\rho }_{R,\alpha })$ given by
			(\ref{simW1}) and the Morgan-Pitman test, when $n=25$\label{table3}} 
	\end{table}%
	
	\begin{table}[htbp]  \tabcolsep2.8pt  \centering%
		\begin{tabular}{llllllllllll}
			\hline
			&  &  & \multicolumn{3}{c}{slightly} & \multicolumn{3}{c}{regular} & 
			\multicolumn{3}{c}{heavily} \\ 
			$\rho $ & $\alpha $ & pure & $0.05$ & $0.10$ & $0.20$ & $0.05$ & $0.10$ & $%
			0.20$ & $0.05$ & $0.10$ & $0.20$ \\ \hline
			$0$ & $0$ & 0.165 & 0.171 & 0.177 & 0.185 & 0.169 & 0.174 & 0.181 & 0.317 & 
			0.448 & 0.627 \\ 
			& $0.1$ & 0.167 & 0.169 & 0.174 & 0.180 & 0.168 & 0.170 & 0.174 & 0.225 & 
			0.322 & 0.515 \\ 
			& $0.2$ & 0.172 & 0.173 & 0.177 & 0.183 & 0.173 & 0.174 & 0.177 & 0.193 & 
			0.238 & 0.371 \\ 
			& $0.3$ & 0.181 & 0.182 & 0.187 & 0.192 & 0.183 & 0.183 & 0.187 & 0.192 & 
			0.218 & 0.297 \\ 
			& $0.5$ & 0.213 & 0.214 & 0.218 & 0.224 & 0.217 & 0.215 & 0.219 & 0.220 & 
			0.234 & 0.272 \\ 
			& $0.7$ & 0.270 & 0.272 & 0.277 & 0.284 & 0.275 & 0.273 & 0.278 & 0.277 & 
			0.291 & 0.320 \\ \hline
			$0.3$ & $0$ & 0.165 & 0.173 & 0.178 & 0.185 & 0.168 & 0.173 & 0.180 & 0.325
			& 0.453 & 0.636 \\ 
			& $0.1$ & 0.166 & 0.171 & 0.174 & 0.180 & 0.166 & 0.169 & 0.173 & 0.227 & 
			0.321 & 0.520 \\ 
			& $0.2$ & 0.171 & 0.175 & 0.178 & 0.184 & 0.171 & 0.173 & 0.177 & 0.194 & 
			0.236 & 0.370 \\ 
			& $0.3$ & 0.181 & 0.184 & 0.187 & 0.192 & 0.181 & 0.182 & 0.187 & 0.194 & 
			0.216 & 0.294 \\ 
			& $0.5$ & 0.213 & 0.217 & 0.220 & 0.225 & 0.213 & 0.213 & 0.220 & 0.223 & 
			0.234 & 0.271 \\ 
			& $0.7$ & 0.270 & 0.275 & 0.278 & 0.287 & 0.271 & 0.270 & 0.277 & 0.280 & 
			0.293 & 0.322 \\ \hline
			$0.6$ & $0$ & 0.163 & 0.172 & 0.177 & 0.183 & 0.169 & 0.172 & 0.179 & 0.342
			& 0.477 & 0.664 \\ 
			& $0.1$ & 0.165 & 0.170 & 0.173 & 0.178 & 0.168 & 0.169 & 0.172 & 0.225 & 
			0.324 & 0.534 \\ 
			& $0.2$ & 0.171 & 0.175 & 0.177 & 0.181 & 0.172 & 0.174 & 0.175 & 0.192 & 
			0.230 & 0.370 \\ 
			& $0.3$ & 0.180 & 0.183 & 0.186 & 0.191 & 0.181 & 0.184 & 0.184 & 0.193 & 
			0.212 & 0.290 \\ 
			& $0.5$ & 0.212 & 0.215 & 0.218 & 0.225 & 0.212 & 0.215 & 0.218 & 0.221 & 
			0.232 & 0.271 \\ 
			& $0.7$ & 0.270 & 0.272 & 0.280 & 0.287 & 0.268 & 0.275 & 0.276 & 0.279 & 
			0.294 & 0.323 \\ \hline
			$0.9$ & $0$ & 0.166 & 0.173 & 0.178 & 0.183 & 0.168 & 0.172 & 0.180 & 0.382
			& 0.542 & 0.729 \\ 
			& $0.1$ & 0.167 & 0.171 & 0.174 & 0.176 & 0.167 & 0.167 & 0.172 & 0.199 & 
			0.283 & 0.514 \\ 
			& $0.2$ & 0.172 & 0.175 & 0.178 & 0.179 & 0.172 & 0.171 & 0.175 & 0.179 & 
			0.201 & 0.294 \\ 
			& $0.3$ & 0.181 & 0.183 & 0.187 & 0.189 & 0.182 & 0.180 & 0.184 & 0.186 & 
			0.198 & 0.235 \\ 
			& $0.5$ & 0.215 & 0.216 & 0.223 & 0.228 & 0.217 & 0.214 & 0.219 & 0.220 & 
			0.232 & 0.254 \\ 
			& $0.7$ & 0.285 & 0.286 & 0.296 & 0.307 & 0.289 & 0.284 & 0.294 & 0.294 & 
			0.311 & 0.344 \\ \hline
		\end{tabular}%
		\caption{Simulated mean square error of the MRPDE for correlation coefficient,
			$\protect\widehat{\rho}_{R,\alpha}$, when $n=25$\label{table4b}} 
	\end{table}%
	
	\begin{table}[htbp]  \tabcolsep2.8pt  \centering%
		\begin{tabular}{llllllllllll}
			\hline
			&  &  & \multicolumn{3}{c}{slightly} & \multicolumn{3}{c}{regular} & 
			\multicolumn{3}{c}{heavily} \\ 
			$\rho $ & $\alpha $ & pure & $0.05$ & $0.10$ & $0.20$ & $0.05$ & $0.10$ & $%
			0.20$ & $0.05$ & $0.10$ & $0.20$ \\ \hline
			$0$ & $0$ & 0.053 & 0.065 & 0.074 & 0.085 & 0.062 & 0.069 & 0.082 & 0.345 & 
			0.576 & 0.847 \\ 
			& $0.1$ & 0.053 & 0.058 & 0.065 & 0.072 & 0.057 & 0.059 & 0.063 & 0.162 & 
			0.351 & 0.689 \\ 
			& $0.2$ & 0.053 & 0.058 & 0.063 & 0.068 & 0.057 & 0.058 & 0.061 & 0.087 & 
			0.170 & 0.425 \\ 
			& $0.3$ & 0.058 & 0.060 & 0.064 & 0.071 & 0.060 & 0.061 & 0.065 & 0.073 & 
			0.115 & 0.264 \\ 
			& $0.5$ & 0.076 & 0.082 & 0.085 & 0.095 & 0.085 & 0.080 & 0.090 & 0.092 & 
			0.106 & 0.168 \\ 
			& $0.7$ & 0.132 & 0.137 & 0.143 & 0.154 & 0.141 & 0.135 & 0.146 & 0.144 & 
			0.160 & 0.203 \\ \hline
			& \multicolumn{1}{c}{MP} & 0.051 & 0.062 & 0.071 & 0.082 & 0.059 & 0.065 & 
			0.080 & 0.341 & 0.572 & 0.844 \\ \hline
			$0.3$ & $0$ & 0.055 & 0.067 & 0.074 & 0.084 & 0.056 & 0.065 & 0.079 & 0.357
			& 0.588 & 0.855 \\ 
			& $0.1$ & 0.055 & 0.059 & 0.063 & 0.072 & 0.051 & 0.054 & 0.062 & 0.170 & 
			0.350 & 0.699 \\ 
			& $0.2$ & 0.055 & 0.058 & 0.060 & 0.067 & 0.051 & 0.051 & 0.062 & 0.088 & 
			0.170 & 0.422 \\ 
			& $0.3$ & 0.058 & 0.059 & 0.062 & 0.072 & 0.055 & 0.056 & 0.066 & 0.073 & 
			0.113 & 0.258 \\ 
			& $0.5$ & 0.080 & 0.084 & 0.088 & 0.094 & 0.079 & 0.080 & 0.088 & 0.091 & 
			0.111 & 0.168 \\ 
			& $0.7$ & 0.134 & 0.138 & 0.143 & 0.157 & 0.137 & 0.132 & 0.141 & 0.144 & 
			0.167 & 0.208 \\ \hline
			& \multicolumn{1}{c}{MP} & 0.052 & 0.063 & 0.071 & 0.082 & 0.054 & 0.062 & 
			0.075 & 0.352 & 0.584 & 0.852 \\ \hline
			$0.6$ & $0$ & 0.049 & 0.064 & 0.072 & 0.080 & 0.060 & 0.067 & 0.082 & 0.386
			& 0.621 & 0.878 \\ 
			& $0.1$ & 0.048 & 0.057 & 0.063 & 0.066 & 0.053 & 0.056 & 0.064 & 0.164 & 
			0.353 & 0.705 \\ 
			& $0.2$ & 0.049 & 0.056 & 0.060 & 0.064 & 0.052 & 0.055 & 0.060 & 0.087 & 
			0.155 & 0.417 \\ 
			& $0.3$ & 0.055 & 0.060 & 0.065 & 0.068 & 0.057 & 0.058 & 0.065 & 0.075 & 
			0.104 & 0.251 \\ 
			& $0.5$ & 0.077 & 0.082 & 0.089 & 0.098 & 0.077 & 0.082 & 0.087 & 0.089 & 
			0.106 & 0.168 \\ 
			& $0.7$ & 0.135 & 0.134 & 0.148 & 0.159 & 0.130 & 0.138 & 0.142 & 0.146 & 
			0.166 & 0.211 \\ \hline
			& \multicolumn{1}{c}{MP} & 0.047 & 0.062 & 0.069 & 0.076 & 0.057 & 0.064 & 
			0.079 & 0.380 & 0.616 & 0.875 \\ \hline
			$0.9$ & $0$ & 0.055 & 0.067 & 0.072 & 0.081 & 0.060 & 0.067 & 0.081 & 0.453
			& 0.710 & 0.929 \\ 
			& $0.1$ & 0.052 & 0.059 & 0.063 & 0.064 & 0.053 & 0.055 & 0.063 & 0.112 & 
			0.268 & 0.635 \\ 
			& $0.2$ & 0.051 & 0.057 & 0.061 & 0.063 & 0.052 & 0.053 & 0.060 & 0.064 & 
			0.098 & 0.265 \\ 
			& $0.3$ & 0.055 & 0.059 & 0.065 & 0.066 & 0.057 & 0.057 & 0.064 & 0.064 & 
			0.082 & 0.148 \\ 
			& $0.5$ & 0.080 & 0.084 & 0.094 & 0.100 & 0.088 & 0.082 & 0.091 & 0.089 & 
			0.106 & 0.140 \\ 
			& $0.7$ & 0.152 & 0.158 & 0.171 & 0.182 & 0.161 & 0.157 & 0.168 & 0.164 & 
			0.191 & 0.236 \\ \hline
			& \multicolumn{1}{c}{MP} & 0.051 & 0.064 & 0.069 & 0.077 & 0.057 & 0.063 & 
			0.078 & 0.448 & 0.706 & 0.927 \\ \hline
		\end{tabular}%
		\caption{Simulated significance level for testing null correlation coefficient through
			$W_{n,\alpha }^{\prime }(\widehat{\rho }_{UV,R,\alpha })$ given by
			(\ref{simW2}) and the Morgan-Pitman test, when $n=25$\label{table3b}} 
	\end{table}%
\end{center}

We repeated the same schema for a nominal type I error, $\varsigma
=0.\allowbreak 05$, for\ different sample sizes $n\in \{15,25,50,100\}$, but
in the main document only the case of $n=25$ is presented (the remaining
sizes are included in the Appendix \ref{A6}). We report, for the different
values of the tuning parameter $\alpha \in \{0,0.\allowbreak 1,0.\allowbreak
2,0.\allowbreak 3,0.\allowbreak 5,0.\allowbreak 7\}$, the
simulated mean square error (MSE) committed in the estimation of $\gamma
=\sigma _{1}/\sigma _{2}$\ and $\rho $ as well as the simulated significance
level of the tests, computed as the number of times the null hypothesis is
rejected out of the total simulated samples $R=15,000$.

With pure data, as expected, the MSEs and closeness of the simulated
significance level of both asymptotic tests, (\ref{simW1}) and (\ref{simW2}%
), to the nominal significance level, $\varsigma =0.\allowbreak 05$, is
improved as the sample size, $n$, increases.
For MSEs under contamination $%
\alpha \in \{0.\allowbreak 1,0.\allowbreak 2\}$ tuning parameters outperform
the MSEs with $\alpha = 0$ but the greatest improvement under contamination
is for the simulated significance levels of $W_{n,\alpha }^{\prime }(%
\widehat{\rho }_{UV,R,\alpha })$, given in (\ref{simW2}), when $\alpha
=0.\allowbreak 2$, since it is always better than any other, included the
well-known Morgan-Pitman test, for all the considered scenarios.

For MSEs under contamination, MRPDEs with $\alpha \geq 0.\allowbreak 1$ outperform the MSEs with $\alpha = 0$ (the MLE) for all the considered scenarios. Furthermore, the effect of contamination is accented for low sample sizes. The greater contamination, the greater the optimal choice of the tuning parameter is, although moderate values of $\alpha$ over $0.2$ generally offer a suitable trade-off between efficiency and robustness.
Conversely, contamination is quite difficult to measure in real-life datasets, and therefore an optimality criterion for the choice of the best tuning parameter is of great practical interest. Several criteria have been proposed in the robustness literature for choosing optimal values of the DPD tuning parameter, that can be straightforward extended for the RP divergence. 
\cite{warwick2005} proposed a useful data-based procedure 
based on the minimization of an estimate of the asymptotic MSE, given by
\begin{equation} \label{eq:Choice}
	\widehat{\text{MSE}}(\alpha) = (\widehat{\boldsymbol{\theta}}_{R,\alpha}-\boldsymbol{\theta}_P)^T(\widehat{\boldsymbol{\theta}}_{R,\alpha}-\boldsymbol{\theta}_P)  + \frac{1}{n}\operatorname{tr}(\boldsymbol{V}_\alpha(\widehat{\boldsymbol{\theta}}_{R,\alpha})),
\end{equation}
where the $\operatorname{tr}$ denotes the trace of the matrix $\boldsymbol{V}_\alpha(\widehat{\boldsymbol{\theta}}_{R,\alpha})$ given in (\ref{V}) and $\boldsymbol{\theta}_P$ is a pilot estimator used for assessing the bias. Naturally, the previous criterion was considerably pilot-dependent, as the pilot invariably draws the optimal estimator towards itself. \cite{Basak2021} improved the method by iteratively updating the pilot with the optimal estimate obtained, and so the process was repeated until there was no further change in the optimal estimate. The iterative algorithm empirically shown  to alleviate the pilot dependency of the optimal choice.
Then, Basak et al. algorithm could be adopted in real life problems for choosing the best MRPDE of the bivariate normal distribution, iteratively minimizing formula (\ref{eq:Choice}). The MLE or MRPDE with moderate values of the tuning parameter $\alpha=0.2$ can be used as initial pilot estimators.
Further, applying transformed Wald-type test statistics, $W_{n,\alpha }^{\prime }(\widehat{\rho }_{UV,R,\alpha })$ given in (\ref{simW2}) entails a clear improvement in terms of significance levels under contamination. In this case, we would empirically recommend the choice of $\alpha =0.\allowbreak 2$  since it is better than any other, included the well-known Morgan-Pitman test, in most of the considered scenarios.

\section{Illustrative examples}

\subsection{Cork data set: comparing means or variances}

Originally studied in \cite{Rao}, there is a well-known and publicity
available real data set, the cork data set. It is included in several R
packages (\cite{rcore2020}), in particular in \texttt{agridat} as a \texttt{%
	box.cork} data. The data report the weights of cork boring of the trunk of $%
28$ trees in the north, east, west and south sides. Rao pointed out that
there exist positive correlation between the reported pairs of $4$
variables, and sometimes it is assumed that they follow a normal
distribution. Four-dimensional normality is an arguable issue since using
the R package \texttt{MVN}, in four out of five tests could multivariate
normality be rejected with significance level $0.05$, as shown in Table \ref%
{tableN} (left hand side). We focussed on the two variables devoted to east
and south sides respectively, as in \cite{Wilcox2015, Wilcox2016}, performed
two-dimensional normal tests and this time all the tests rejected according
to Table \ref{tableN} (central columns). We did not study the robustness of
an estimator and test statistic as in \cite{Wilcox2015}, in the sense of being
resistant to data coming from distribution different from the bivariate
normal as required for the original data. Our proposed estimator and
test statistic are robust in the sense of being resistant to outliers once
normality is being assumed. Having this in mind, the data were transformed
using the base $e$ logarithm and as shown in Table \ref{tableN} (right hand
side) and in four out five tests could not be multivariate normality
rejected. In addition, outliers were studied for transformed data through
the scatter-plot with confidence ellipses shown in Figure \ref{fig:scatter1}%
, concluding that observations $18$ and $16$ were suspicious to be outliers.

Taking the root of the Wald-type test statistics given in Cases \ref{subSec1}%
, \ref{subSec2} and \ref{subSec3}, as well as the Wald-type test given in
Section \ref{sim} based on transformed variables, here we are going to
provide alternative but equivalent expressions in practice, having the same $%
p$-value. From the original data, $(x_{1},y_{1}),\ldots ,(x_{n},y_{n})$, the 
$z$-type test statistic for equal variances and based on MRPDEs, is given by%
\begin{equation}
	\begin{aligned}
		Z_{n,\alpha }(\widehat{\gamma }_{R,\alpha },\widehat{\rho }_{R,\alpha })&=%
		\mathrm{sign}(\widehat{\gamma }_{R,\alpha }-1)\sqrt{W_{n,\alpha }(\widehat{%
				\gamma }_{R,\alpha },\widehat{\rho }_{R,\alpha })}\\
		&=\sqrt{n}\left( \frac{%
			\sqrt{2\alpha +1}}{\alpha +1}\right) ^{3}\frac{\widehat{\gamma }_{R,\alpha
			}-1}{\sqrt{\beta _{\alpha }(\widehat{\gamma }_{R,\alpha },\widehat{\rho }%
				_{R,\alpha })}},
	\end{aligned}  \label{z3}
\end{equation}%
with $\beta _{\alpha }(\widehat{\gamma }_{R,\alpha },\widehat{\rho }%
_{R,\alpha })$\ given by (\ref{beta}), has a standard normal asymptotic
distribution. From the transformed data $(u_{1},v_{1}),\ldots ,(u_{n},v_{n})$%
, where $U=X+Y$ and $V=X-Y$, the $z$-type test statistic for equal variances
and based on MRPDEs, is given by%
\begin{equation}
	Z_{n,\alpha }^{\prime }(\widehat{\rho }_{UV,R,\alpha })=\mathrm{sign}(%
	\widehat{\rho }_{UV,R,\alpha })\sqrt{W_{n,\alpha }^{\prime }(\widehat{\rho }%
		_{UV,R,\alpha })}=\sqrt{n}\left( \frac{\sqrt{2\alpha +1}}{\alpha +1}\right)
	^{3}\widehat{\rho }_{UV,R,\alpha }.  \label{z2}
\end{equation}%
Notice that the paired $t$-test for the same test (null correlation), with
exact distribution $t_{n-2}$, is the Morgan-Pitman test $T_{MP}$\ given by (%
\ref{MorganPitman}). The $z$-type test statistic for equal means and based
on MRPDEs and the transformed data $%
v_{1},\ldots ,v_{n}$, where $V=X-Y$, is given by%
\begin{equation}
	Z_{n,\alpha }(\widehat{\mu }_{V,R,\alpha },\widehat{\sigma }_{V,R,\alpha })=%
	\mathrm{sign}(\widehat{\mu }_{V,R,\alpha })\sqrt{W_{n,\alpha }(\widehat{\mu }%
		_{V,R,\alpha },\widehat{\sigma }_{V,R,\alpha })}=\sqrt{n}\frac{2\alpha +1}{%
		\left( \alpha +1\right) ^{2}}\frac{\widehat{\mu }_{V,R,\alpha }}{\widehat{%
			\sigma }_{V,R,\alpha }}.  \label{z1}
\end{equation}%
Notice that the paired $t$-test for the same test (equal means), with exact
distribution $t_{n-1}$, is given by%
\begin{equation*}
	T_{V}=\sqrt{\frac{n-1}{n}}Z_{n,\alpha =0}(\widehat{\mu }_{V,R,\alpha },%
	\widehat{\sigma }_{V,R,\alpha }).
\end{equation*}

The results of the classic exact tests, $T_{V}$ for testing (\ref{TestMeans}%
) or $T_{MP}$ for (\ref{H0}), with two-sided alternative, are summarized in\
Table \ref{TableC}. The decision, with $0.05$ significance level, is
opposite for both versions of the data, since the null hypothesis cannot be
rejected for the full data set, while it is rejected for the outliers
deleted data.

\begin{table}[]
	\caption{Normality tests for three versions of the corn data set}
	\label{tableN}
	\begin{center}
		\begin{tabular}[H]{ccccccc}
			& \multicolumn{2}{c}{4-dimensional} & \multicolumn{2}{c}{2-dimensional} & 
			\multicolumn{2}{c}{$\log $ 2-dimensional} \\ \hline
			normality test & value & $p$-value & value & $p$-value & value & $p$-value
			\\ \hline
			\multicolumn{1}{l}{1.- \texttt{Doornik-Hansen}} & \multicolumn{1}{r}{\texttt{%
					16.123}} & \multicolumn{1}{r}{\texttt{0.041}} & \multicolumn{1}{r}{\texttt{%
					9.833}} & \multicolumn{1}{r}{\texttt{0.043}} & \multicolumn{1}{r}{\texttt{%
					2.490}} & \multicolumn{1}{r}{\texttt{0.646}} \\ 
			\multicolumn{1}{l}{2.- \texttt{Henze-Zirkler}} & \multicolumn{1}{r}{\texttt{%
					0.999}} & \multicolumn{1}{r}{\texttt{0.011}} & \multicolumn{1}{r}{\texttt{%
					1.236}} & \multicolumn{1}{r}{\texttt{0.003}} & \multicolumn{1}{r}{\texttt{%
					0.784}} & \multicolumn{1}{r}{\texttt{0.053}} \\ 
			\multicolumn{1}{l}{3.- \texttt{Royston}} & \multicolumn{1}{r}{\texttt{12.161}
			} & \multicolumn{1}{r}{\texttt{0.003}} & \multicolumn{1}{r}{\texttt{11.784}}
			& \multicolumn{1}{r}{\texttt{0.002}} & \multicolumn{1}{r}{\texttt{5.564}} & 
			\multicolumn{1}{r}{\texttt{0.047}} \\ 
			\multicolumn{1}{l}{4.- \texttt{E-statistic}} & \multicolumn{1}{r}{\texttt{%
					1.276}} & \multicolumn{1}{r}{\texttt{0.007}} & \multicolumn{1}{r}{\texttt{%
					1.473}} & \multicolumn{1}{r}{\texttt{0.001}} & \multicolumn{1}{r}{\texttt{%
					0.931}} & \multicolumn{1}{r}{\texttt{0.053}} \\ 
			\multicolumn{1}{l}{5a.- \texttt{Mardia: Skewness}} & \multicolumn{1}{r}{%
				\texttt{20.890}} & \multicolumn{1}{r}{\texttt{0.404}} & \multicolumn{1}{r}{%
				\texttt{10.231}} & \multicolumn{1}{r}{\texttt{0.037}} & \multicolumn{1}{r}{%
				\texttt{2.622}} & \multicolumn{1}{r}{\texttt{0.623}} \\ 
			\multicolumn{1}{l}{5b.- \texttt{Mardia: Kurtosis}} & \multicolumn{1}{r}{%
				\texttt{-0.398}} & \multicolumn{1}{r}{\texttt{0.690}} & \multicolumn{1}{r}{%
				\texttt{0.899}} & \multicolumn{1}{r}{\texttt{0.369}} & \multicolumn{1}{r}{%
				\texttt{-0.795}} & \multicolumn{1}{r}{\texttt{0.427}} \\ \hline
		\end{tabular}%
	\end{center}
\end{table}

\begin{figure}[]
	\centering
	\includegraphics[width=0.65\textwidth, keepaspectratio]{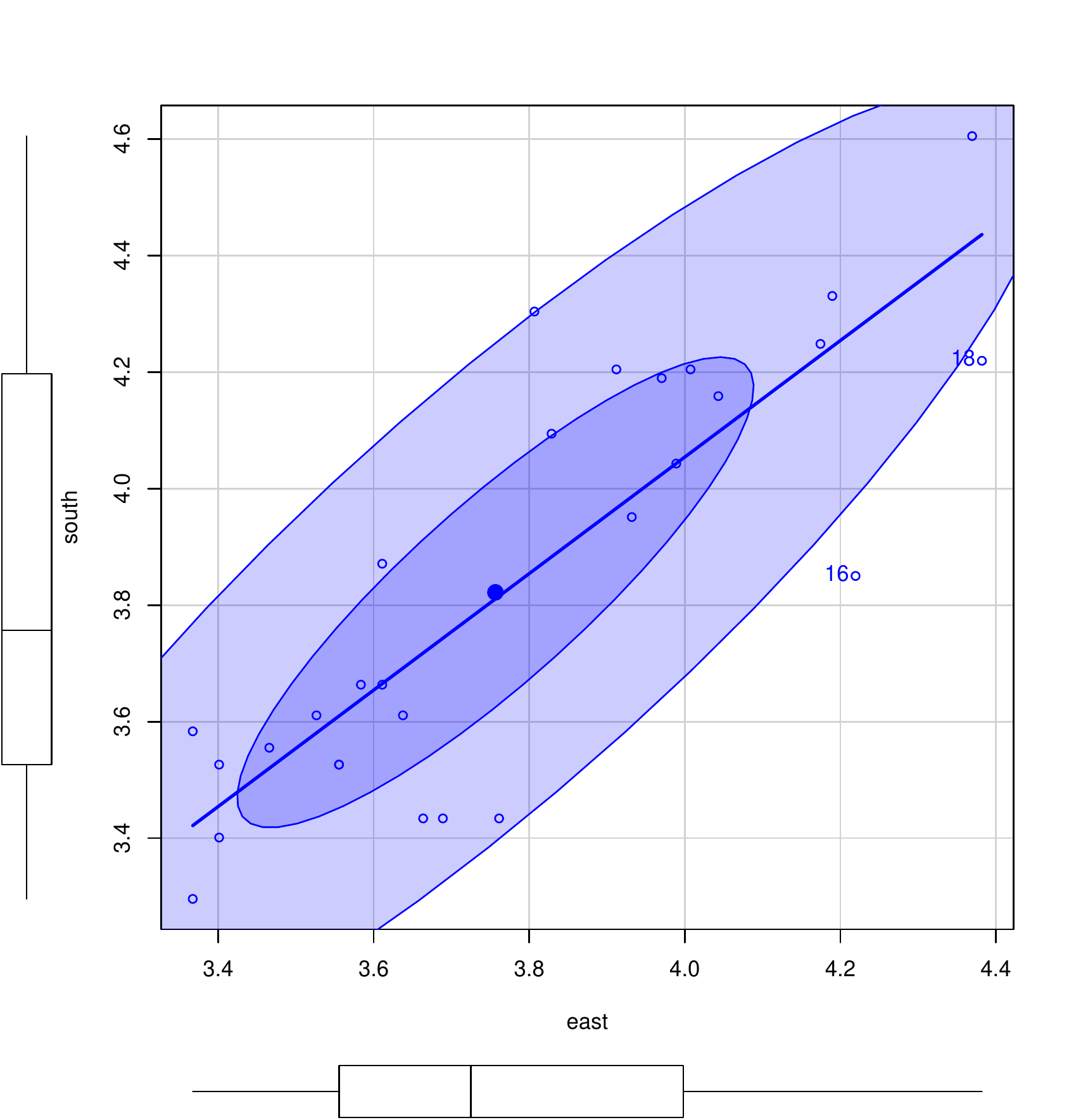}
	\caption{Scatter-plot of east and south variables, with confidence ellipses,
		for log-transformed cork data set.}
	\label{fig:scatter1}
\end{figure}

\begin{table}[tbp]
	\caption{Classic exact tests of equal means or equal variances for the
		log-transformed corn data set with respect to full data or ouliers deleted
		data.}
	\label{TableC}
	\begin{center}
		\begin{tabular}{ccccc}
			& \multicolumn{2}{c}{full data} & \multicolumn{2}{c}{outliers deleted data}
			\\ \hline
			classic exact t-test & value & $p$-value & value & $p$-value \\ \hline
			\multicolumn{1}{l}{\texttt{Paired t-test }(equal means), $T_{V}$} & 
			\multicolumn{1}{r}{\texttt{-1.454}} & \multicolumn{1}{r}{\texttt{0.157}} & 
			\multicolumn{1}{r}{\texttt{-2.233}} & \texttt{0.035} \\ \hline
			\multicolumn{1}{l}{\texttt{Morgan-Pitman test }(equal variances), $T_{MP}$}
			& \multicolumn{1}{r}{\texttt{-1.656}} & \multicolumn{1}{r}{\texttt{0.110}} & 
			\multicolumn{1}{r}{\texttt{-3.033}} & \texttt{0.005} \\ \hline
		\end{tabular}%
	\end{center}
\end{table}

The advantage of these new expressions, (\ref{z3})-(\ref{z1}), is that
one-sided tests can be considered, apart from the two sided ones (as in the
example given in Section \ref{Ex2}). Based on (\ref{z3})-(\ref{z1}), in
Figures \ref{fig:CorkAB}-\ref{fig:CorkC} the values of the test statistics
(left hand side) and the values of the estimates of parameters used to
construct the test statistics (right hand sides) are shown, in solid lines
the ones associated with the full log-transformed cork data set and in
dashed lines the ones associated with the outliers deleted log-transformed
cork data set. All the left hand side figures suggest rejecting the null,
equal means or variances, as an appropriate decision with 0.05 significance
level.

\begin{figure}[tbp]
	\centering
	\begin{tabular}{cc}
		\includegraphics[scale=0.35]{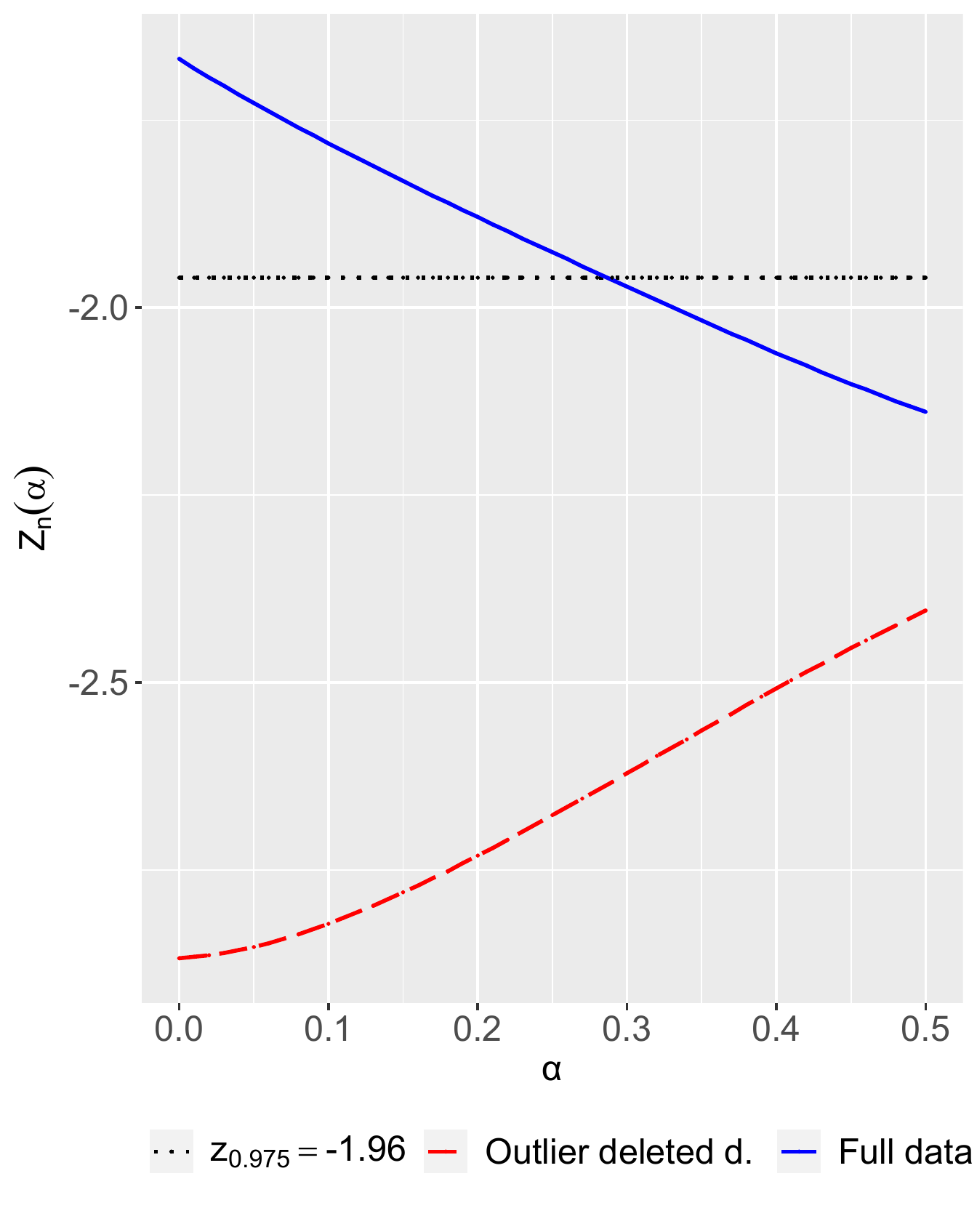} & %
		\includegraphics[scale=0.35]{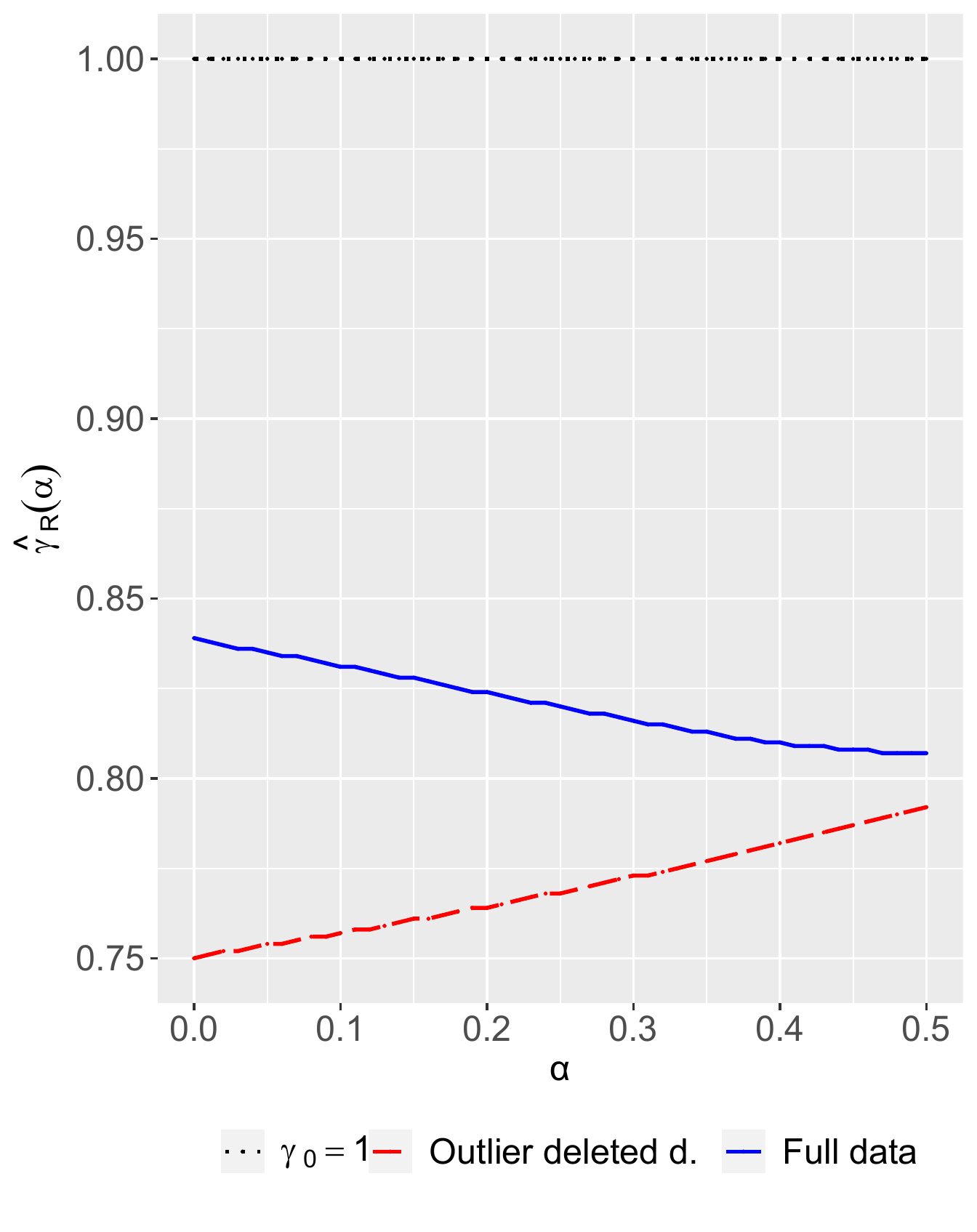} \\ 
		\includegraphics[scale=0.35]{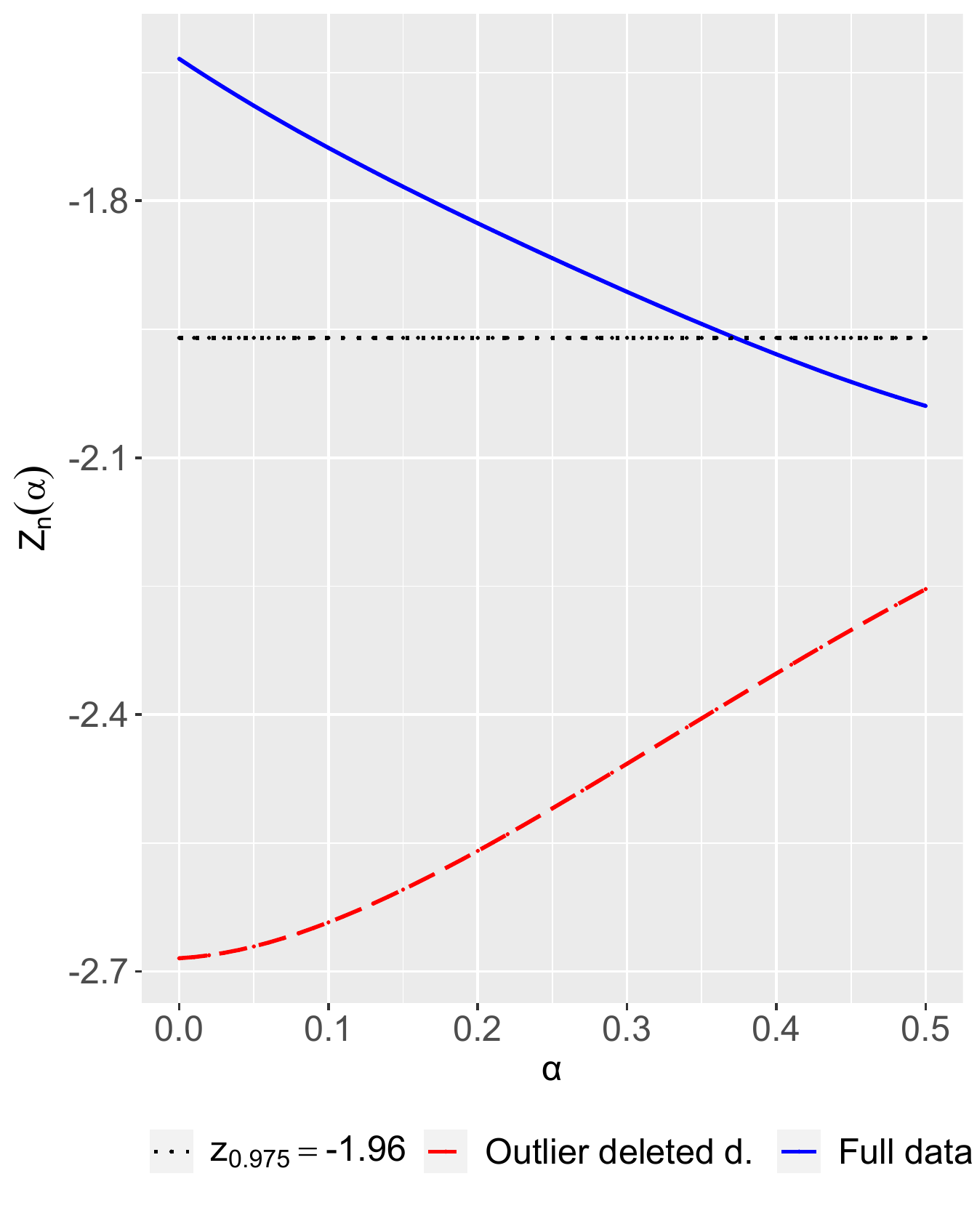} & %
		\includegraphics[scale=0.35]{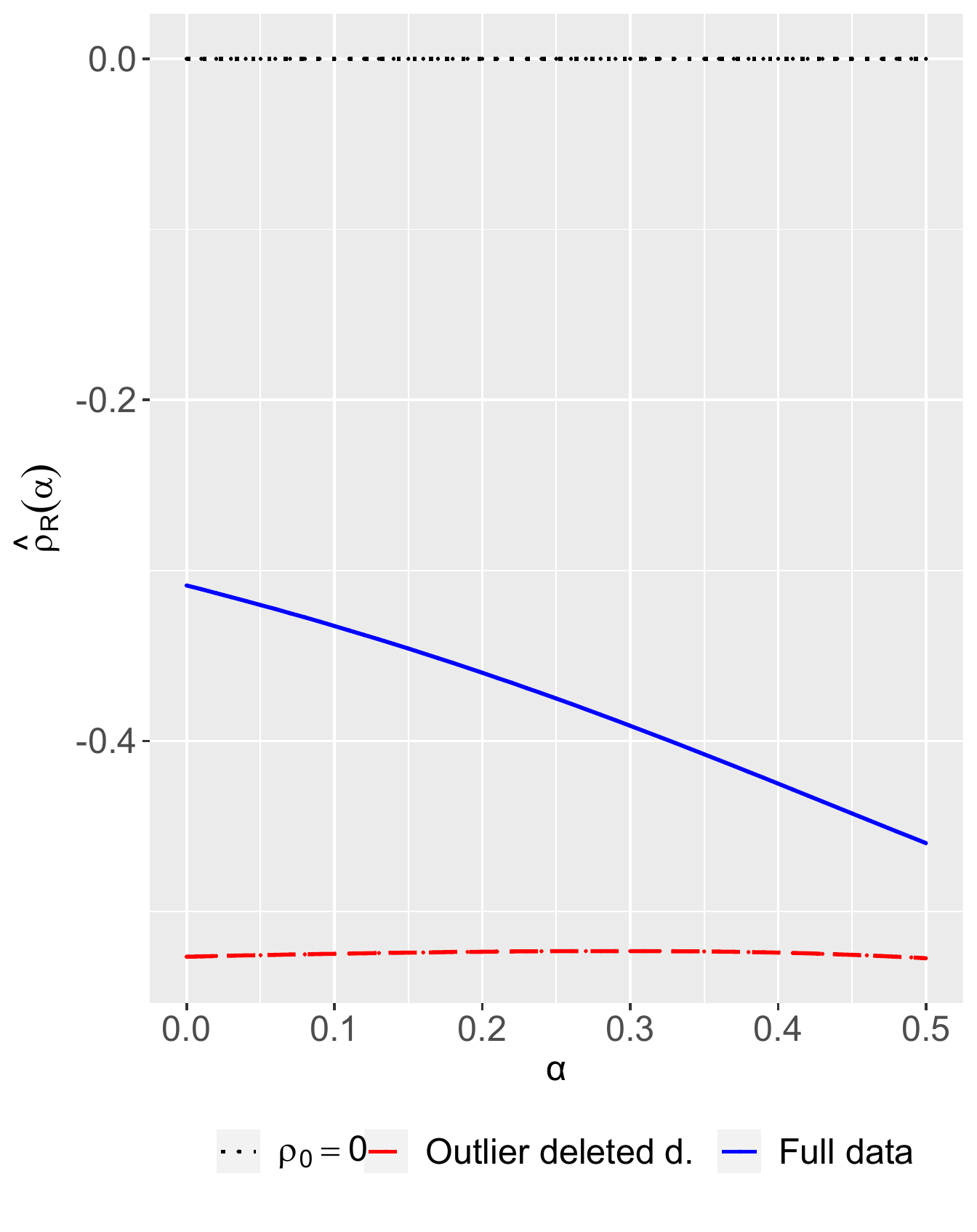} \\ 
		& 
	\end{tabular}%
	\caption{Wald-type tests (left) and estimates (right) for log-transformed
		cork data set: Case 2 (above) and Case 3 (below)}
	\label{fig:CorkAB}
\end{figure}
\begin{figure}[tbp]
	\centering
	\begin{tabular}{cc}
		\includegraphics[scale=0.35]{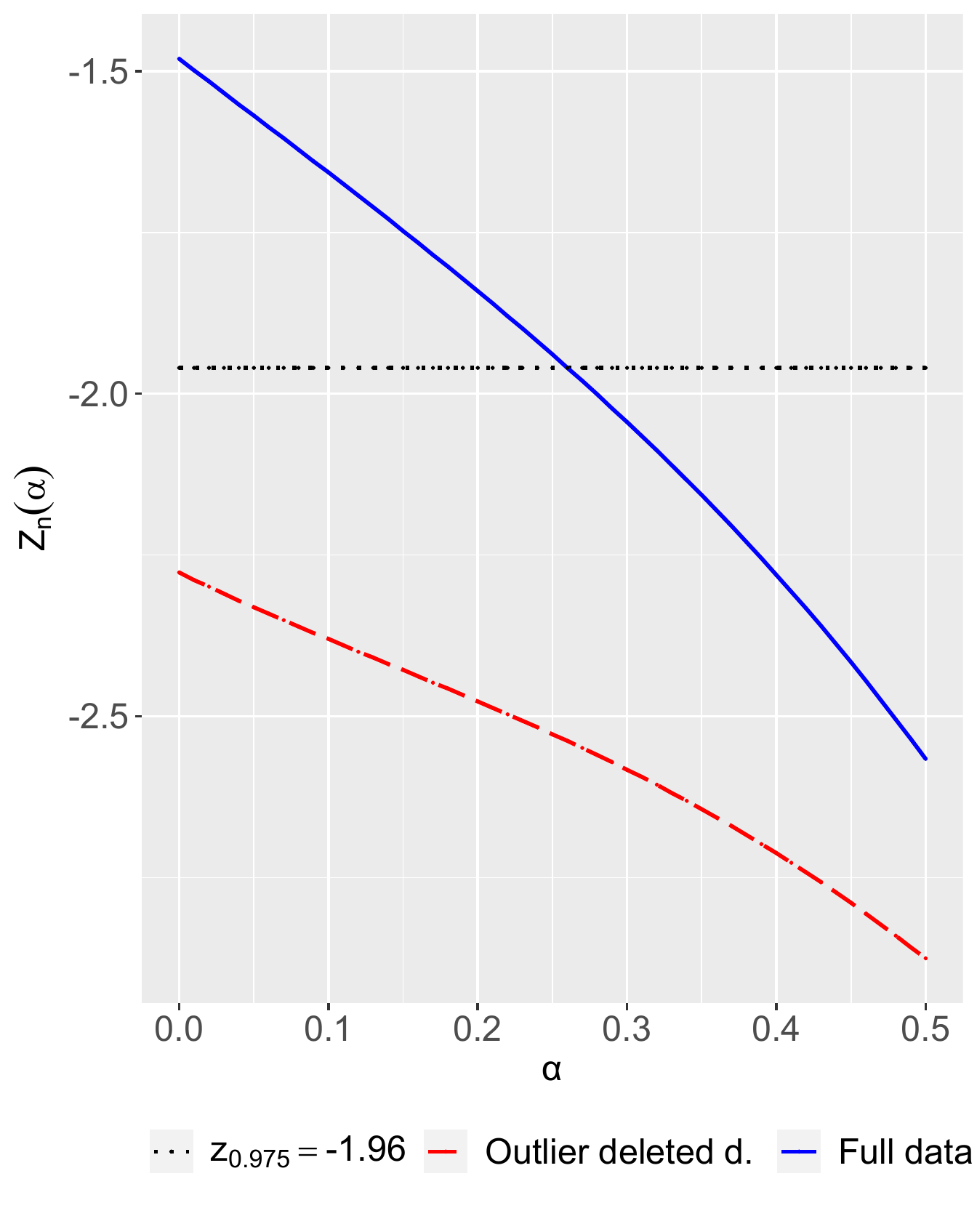} & %
		\includegraphics[scale=0.35]{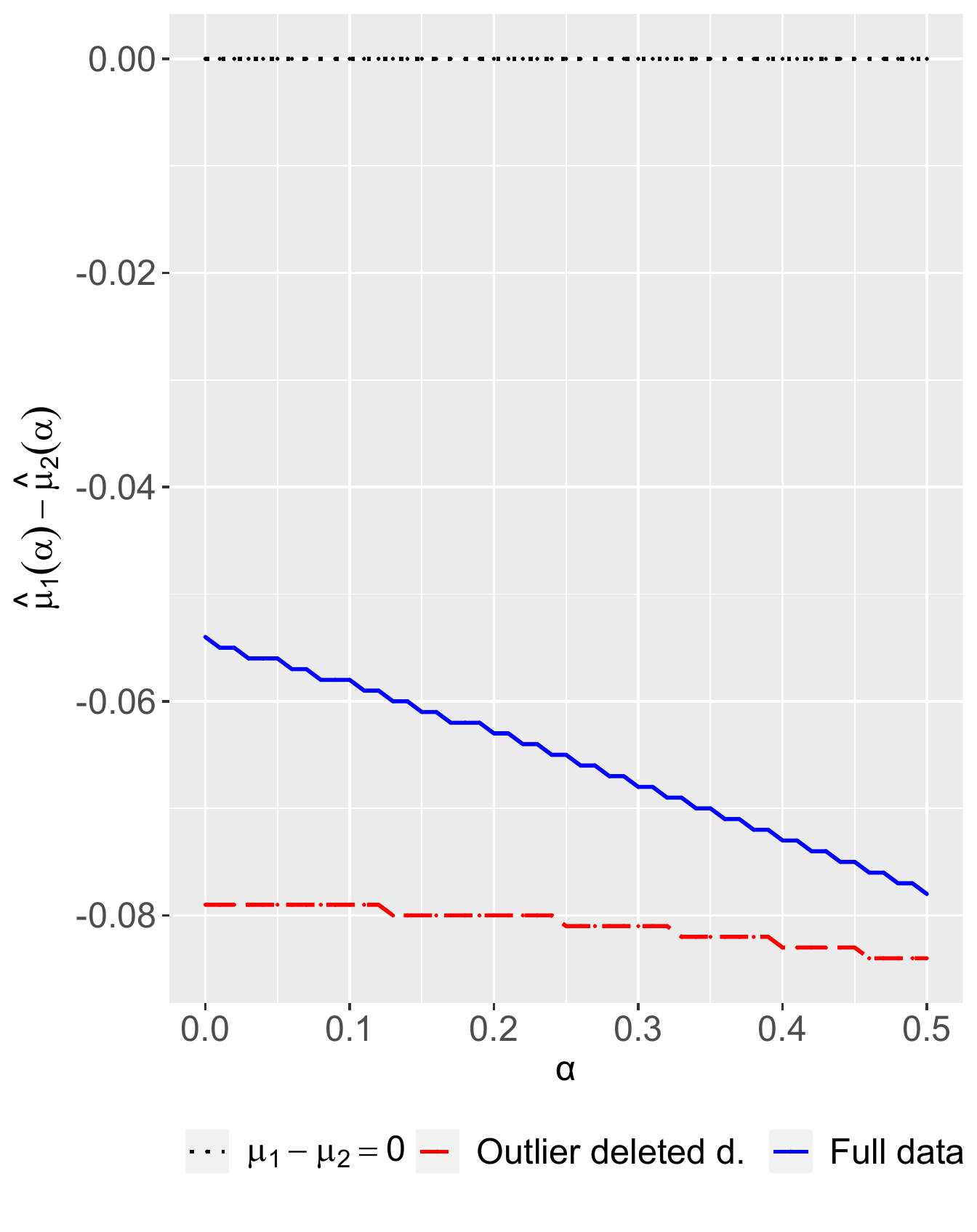}%
	\end{tabular}%
	\caption{Wald-type tests in Case 1 (left) and mean diference estimates
		(right) for log-transformed cork data set }
	\label{fig:CorkC}
\end{figure}

\subsection{Lactate levels data set: fixing a positive correlation
	coefficient\label{Ex2}}

\cite{hutson2019} studied the one sided test (\ref{H0b}), $H_{0}:\rho =\rho _{0}$
vs. $H_{0}:\rho >\rho _{0}$, where $\rho _{0}=0$, for lactate levels
measured in the blood and the cerebrospinal fluid on $13$ female subjects.
The study was done with a newly proposed robust test statistic, but in the
sense of being resistant to data coming from distribution different from the
bivariate normal. Using the R package \texttt{MVN}, in none of tests could
multivariate normality be rejected with significance level $0.05$, as shown
in Table \ref{tableN2}. Again, we highlight that our proposed estimators and
test statistics are robust in the sense of being resistant to outliers once
normality is being assumed. Deleting the two most influential observations,
i.e. taking observations $1$ and $7$ as influential (rather than outliers),
the sample Pearson correlation is modified from $0.572$ for the full data to 
$0.471$ for the influential observations deleted data (see Figure \ref%
{fig:scatter2}) and according to Table \ref{TableC2}, using the
Morgan-Pitman exact test statistic, the decision of being accepted a
positive correlation with $0.05$ significance level is modified to not being
possible to be accepted. With Figure \ref{fig:LactateAB} we try to test
whether with the $Z$-test statistic based on MRPDE of $\rho $, (\ref{z2}), a positive correlation could be accepted
for the lactate levels data and actually it suggest as desirable decision
nor being possible to reject it.

\begin{figure}[ht]
	\centering
	\includegraphics[width=0.85\textwidth, keepaspectratio]{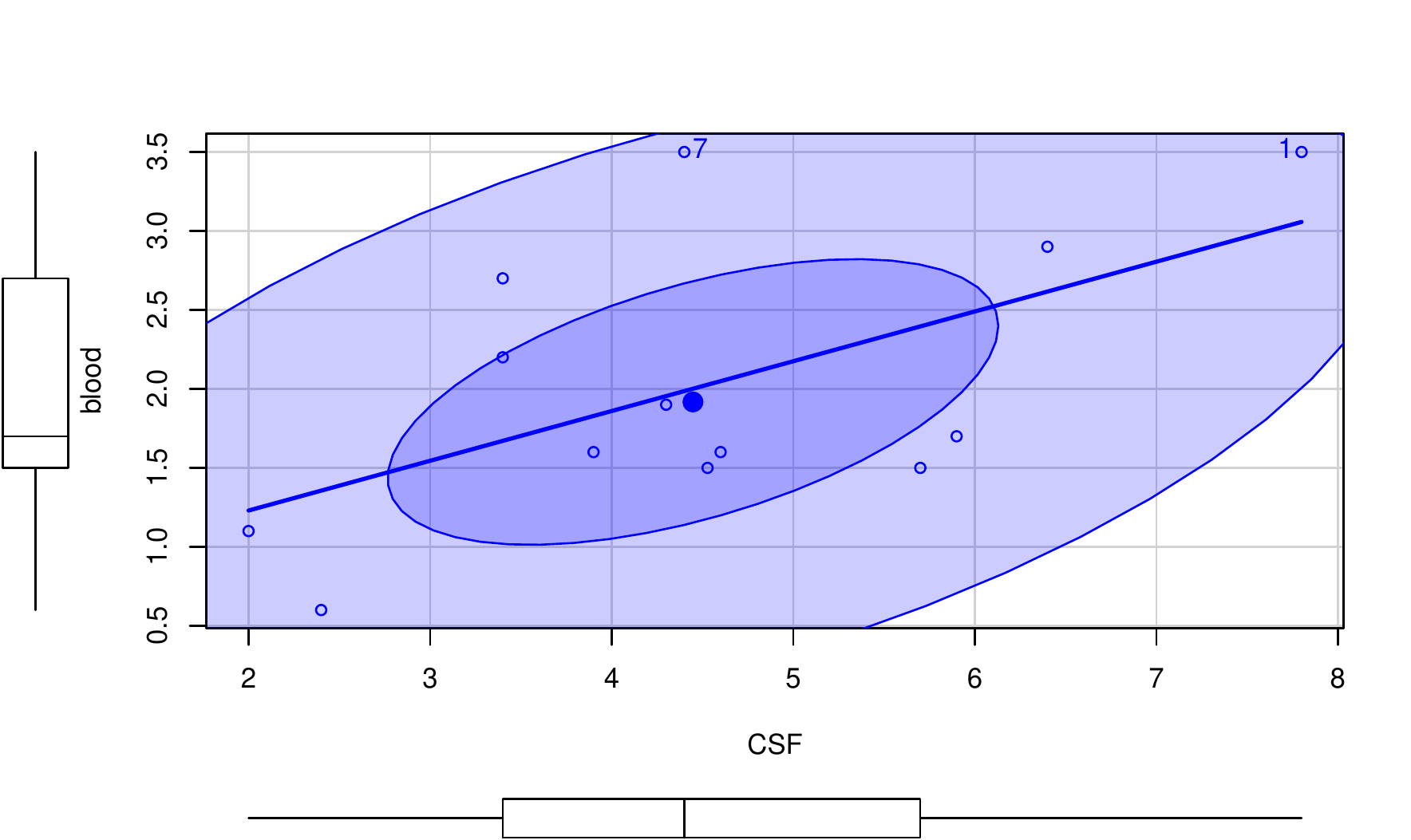}
	\caption{Scatter-plot of CSF and blood variables, with confidence ellipses,
		for lactate levels data set.}
	\label{fig:scatter2}
\end{figure}

\begin{table}[ht]
	\caption{Normality tests for lactate levels data set}
	\label{tableN2}
	\begin{center}
		\begin{tabular}{ccc}
			\hline
			normality test & value & $p$-value \\ \hline
			\multicolumn{1}{l}{1.- Doornik-Hansen} & \multicolumn{1}{r}{1.436} & 
			\multicolumn{1}{r}{0.838} \\ 
			\multicolumn{1}{l}{2.- Henze-Zirkler} & \multicolumn{1}{r}{0.352} & 
			\multicolumn{1}{r}{0.461} \\ 
			\multicolumn{1}{l}{3.- Royston} & \multicolumn{1}{r}{0.923} & 
			\multicolumn{1}{r}{0.642} \\ 
			\multicolumn{1}{l}{4.- E-statistic} & \multicolumn{1}{r}{0.620} & 
			\multicolumn{1}{r}{0.509} \\ 
			\multicolumn{1}{l}{5a.- Mardia: Skewness} & \multicolumn{1}{r}{1.814} & 
			\multicolumn{1}{r}{0.770} \\ 
			\multicolumn{1}{l}{5b.- Mardia: Kurtosis} & \multicolumn{1}{r}{-0.774} & 
			\multicolumn{1}{r}{0.439} \\ \hline
		\end{tabular}%
	\end{center}
\end{table}

\begin{figure}[ht]
	\centering
	\begin{tabular}{cc}
		\includegraphics[scale=0.35]{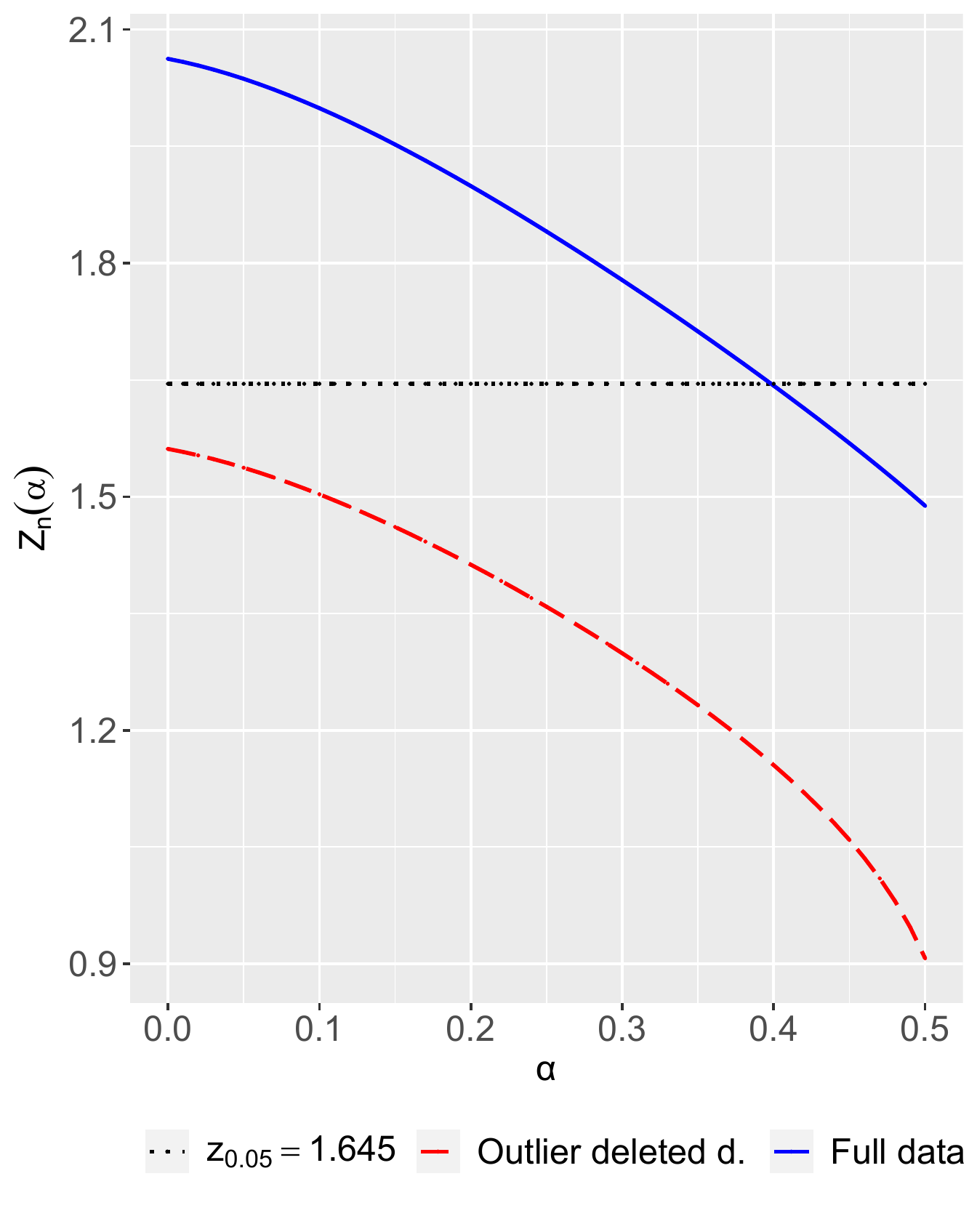} & %
		\includegraphics[scale=0.35]{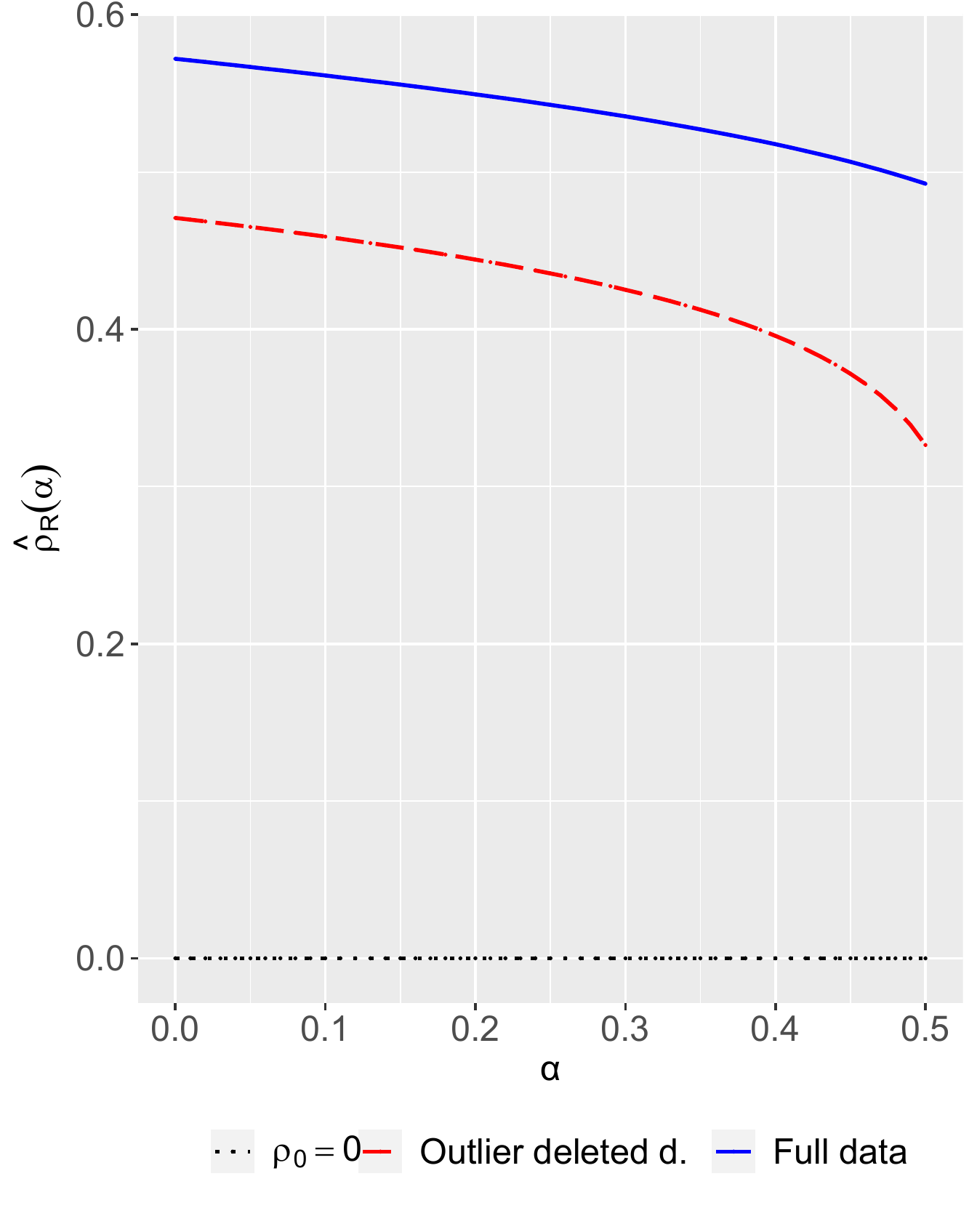} \\ 
		& 
	\end{tabular}%
	\caption{Wald-type tests (left) and estimates (right) for the lactate levels
		data set.}
	\label{fig:LactateAB}
\end{figure}

\begin{table}[ht]
	\caption{Classic exact test of uncorrelation for the lactate levels data set
		with respect to full data or influential observations deleted data.}
	\label{TableC2}
	\begin{center}
		\begin{tabular}{ccccc}
			& \multicolumn{2}{c}{full data} & \multicolumn{2}{c}{infl. obs. deleted data}
			\\ \hline
			classic exact t-test & value & $p$-value & value & $p$-value \\ \hline
			\multicolumn{1}{l}{one-sided positive correlation with $T_{MP}$} & 
			\multicolumn{1}{r}{2.313} & \multicolumn{1}{r}{0.020} & \multicolumn{1}{r}{
				1.601} & 0.072 \\ \hline
		\end{tabular}%
	\end{center}
\end{table}

\newpage

\section{Concluding Remarks}

In practice, it is very important finding out a robust estimator and
test statistic which does not loose much efficiency. Under the null
hypothesis $\rho =0$, the Morgan-Pitman exact test is the most efficient one
but among the classic asymptotic tests there are several versions we should
know. However, we would like to highlight that in comparative studies of
recent papers the most competitive one, the Rao test given in Case \ref%
{subSec3} (Section \ref{secWald}), $R_{n,\alpha =0}(\widehat{\boldsymbol{%
		\theta }}_{R,\alpha =0})$, is not often being recognized. Since the MRPDEs are regulated through a positive $\alpha $
tuning parameter, being the tuning parameter $\alpha =0$ the cornerstone as
being the most efficient one out of all possible values of $\alpha \geq 0$.
In case of having a poor efficiency for the asymptotic test statistic with
the null tuning parameter of the MRPDEs, the
test statistics constructed with the other values of the tuning parameters
will increase such lack of efficiency and the obtained robustness could not
compensate such drawback. This is just what happened with the Wald-type
test statistic $W_{n,\alpha =0}(\widehat{\boldsymbol{\theta }}_{R,\alpha =0})
$, given in Case \ref{subSec3} (Section \ref{secWald}) for testing $\rho =0$%
. Further, we used a modified version of the Wald-type test statistic $%
W_{n,\alpha =0}^{\prime }(\widehat{\boldsymbol{\theta }}_{R,\alpha =0})$,
given in Case \ref{subSec3} (Section \ref{secWald}), which matches the Rao
test statistic only when $\rho _{0}=0$, and has provided for $\alpha =0$
magnificent results in efficiency and also for $\alpha >0$ strong robust
properties. The improvements and properties are shown by simulation for the
specific null hypothesis $\rho =0$, but proven in the framework of the
developed general theory.


%
%
%

\section*{Acknowledgement}

This research is supported by the Spanish Grants PGC2018-095194-B-100 and FPU/018240 (M. Jaenada). E. Castilla, M.Jaenada, N. Mart\'in and L. Pardo
are members of the Instituto de Matematica Interdisciplinar, Complutense University of Madrid.

%
%
%
%

\begin{appendices}

	\section{Proof of Theorem \protect\ref{Th1}\label{A1}}
	
	\begin{align*}
		& \tfrac{\partial }{\partial \boldsymbol{\theta }}\boldsymbol{\Psi _{\alpha
			}^{T}}(X;\boldsymbol{\theta }) \\
		& =\tfrac{\partial }{\partial \boldsymbol{\theta }}f_{\boldsymbol{\theta }%
		}^{\alpha }(X)\left( \boldsymbol{u}_{\boldsymbol{\theta }}^{T}(X)-%
		\boldsymbol{c}_{\alpha }^{T}\left( \boldsymbol{\theta }\right) \right) +f_{%
			\boldsymbol{\theta }}^{\alpha }(X)\left( \tfrac{\partial }{\partial 
			\boldsymbol{\theta }}\boldsymbol{u}_{\boldsymbol{\theta }}^{T}(X)-\tfrac{%
			\partial }{\partial \boldsymbol{\theta }}\boldsymbol{c}_{\alpha }^{T}\left( 
		\boldsymbol{\theta }\right) \right) \\
		& =\alpha f_{\boldsymbol{\theta }}^{\alpha }(X)\boldsymbol{u}_{\boldsymbol{%
				\theta }}(X)\left( \boldsymbol{u}_{\boldsymbol{\theta }}^{T}(X)-\boldsymbol{c%
		}_{\alpha }^{T}\left( \boldsymbol{\theta }\right) \right) +f_{\boldsymbol{%
				\theta }}^{\alpha }(X)\left( \tfrac{\partial }{\partial \boldsymbol{\theta }}%
		\boldsymbol{u}_{\boldsymbol{\theta }}^{T}(X)-\tfrac{\partial }{\partial 
			\boldsymbol{\theta }}\boldsymbol{c}_{\alpha }^{T}\left( \boldsymbol{\theta }%
		\right) \right) \\
		& =\alpha f_{\boldsymbol{\theta }}^{\alpha }(X)\boldsymbol{u}_{\boldsymbol{%
				\theta }}(X)\boldsymbol{u}_{\boldsymbol{\theta }}^{T}(X)-\alpha f_{%
			\boldsymbol{\theta }}^{\alpha }(X)\boldsymbol{u}_{\boldsymbol{\theta }}(X)%
		\boldsymbol{c}_{\alpha }^{T}\left( \boldsymbol{\theta }\right) +f_{%
			\boldsymbol{\theta }}^{\alpha }(X)\tfrac{\partial }{\partial \boldsymbol{%
				\theta }}\boldsymbol{u}_{\boldsymbol{\theta }}^{T}(X)-f_{\boldsymbol{\theta }%
		}^{\alpha }(X)\tfrac{\partial }{\partial \boldsymbol{\theta }}\boldsymbol{c}%
		_{\alpha }^{T}\left( \boldsymbol{\theta }\right) ,
	\end{align*}%
	\begin{align*}
		& E\left[ \tfrac{\partial }{\partial \boldsymbol{\theta }}\boldsymbol{\Psi
			_{\alpha }^{T}}(X;\boldsymbol{\theta })\right] \\
		& =\alpha E\left[ f_{\boldsymbol{\theta }}^{\alpha }(X)\boldsymbol{u}_{%
			\boldsymbol{\theta }}(X)\boldsymbol{u}_{\boldsymbol{\theta }}^{T}(X)\right]
		-\alpha E\left[ f_{\boldsymbol{\theta }}^{\alpha }(X)\boldsymbol{u}_{%
			\boldsymbol{\theta }}(X)\right] \boldsymbol{c}_{\alpha }^{T}\left( 
		\boldsymbol{\theta }\right) \\
		& +E\left[ f_{\boldsymbol{\theta }}^{\alpha }(X)\tfrac{\partial }{\partial 
			\boldsymbol{\theta }}\boldsymbol{u}_{\boldsymbol{\theta }}^{T}(X)\right] -E%
		\left[ f_{\boldsymbol{\theta }}^{\alpha }(X)\right] \tfrac{\partial }{%
			\partial \boldsymbol{\theta }}\boldsymbol{c}_{\alpha }^{T}\left( \boldsymbol{%
			\theta }\right) \\
		& =\alpha E\left[ f_{\boldsymbol{\theta }}^{\alpha }(X)\boldsymbol{u}_{%
			\boldsymbol{\theta }}(X)\boldsymbol{u}_{\boldsymbol{\theta }}^{T}(X)\right]
		-\alpha E\left[ f_{\boldsymbol{\theta }}^{\alpha }(X)\right] \boldsymbol{c}%
		_{\alpha }\left( \boldsymbol{\theta }\right) \boldsymbol{c}_{\alpha
		}^{T}\left( \boldsymbol{\theta }\right) \\
		& +E\left[ f_{\boldsymbol{\theta }}^{\alpha }(X)\tfrac{\partial }{\partial 
			\boldsymbol{\theta }}\boldsymbol{u}_{\boldsymbol{\theta }}^{T}(X)\right] -E%
		\left[ f_{\boldsymbol{\theta }}^{\alpha }(X)\right] \tfrac{\partial }{%
			\partial \boldsymbol{\theta }}\boldsymbol{c}_{\alpha }^{T}\left( \boldsymbol{%
			\theta }\right) \\
		& =\alpha E\left[ f_{\boldsymbol{\theta }}^{\alpha }(X)\boldsymbol{u}_{%
			\boldsymbol{\theta }}(X)\boldsymbol{u}_{\boldsymbol{\theta }}^{T}(X)\right]
		-\alpha E\left[ f_{\boldsymbol{\theta }}^{\alpha }(X)\right] \boldsymbol{c}%
		_{\alpha }\left( \boldsymbol{\theta }\right) \boldsymbol{c}_{\alpha
		}^{T}\left( \boldsymbol{\theta }\right) \\
		& +(\alpha +1)E\left[ f_{\boldsymbol{\theta }}^{\alpha }(X)\boldsymbol{u}_{%
			\boldsymbol{\theta }}(X)\right] \boldsymbol{c}_{\alpha }^{T}\left( 
		\boldsymbol{\theta }\right) -(\alpha +1)E\left[ f_{\boldsymbol{\theta }%
		}^{\alpha }(X)\boldsymbol{u}_{\boldsymbol{\theta }}(X)\boldsymbol{u}_{%
			\boldsymbol{\theta }}^{T}(X)\right] \\
		& =E\left[ f_{\boldsymbol{\theta }}^{\alpha }(X)\boldsymbol{u}_{\boldsymbol{%
				\theta }}(X)\right] \boldsymbol{c}_{\alpha }^{T}\left( \boldsymbol{\theta }%
		\right) -E\left[ f_{\boldsymbol{\theta }}^{\alpha }(X)\boldsymbol{u}_{%
			\boldsymbol{\theta }}(X)\boldsymbol{u}_{\boldsymbol{\theta }}^{T}(X)\right]
		\\
		& =E\left[ f_{\boldsymbol{\theta }}^{\alpha }(X)\right] \boldsymbol{c}%
		_{\alpha }\left( \boldsymbol{\theta }\right) \boldsymbol{c}_{\alpha
		}^{T}\left( \boldsymbol{\theta }\right) -E\left[ f_{\boldsymbol{\theta }%
		}^{\alpha }(X)\boldsymbol{u}_{\boldsymbol{\theta }}(X)\boldsymbol{u}_{%
			\boldsymbol{\theta }}^{T}(X)\right] ,
	\end{align*}%
	\begin{align*}
		\boldsymbol{S}_{\alpha }\left( \boldsymbol{\theta }\right) & =-E\left[ 
		\tfrac{\partial }{\partial \boldsymbol{\theta }}\boldsymbol{\Psi _{\alpha
			}^{T}}(x;\boldsymbol{\theta })\right] =E\left[ f_{\boldsymbol{\theta }%
		}^{\alpha }(X)\boldsymbol{u}_{\boldsymbol{\theta }}(X)\boldsymbol{u}_{%
			\boldsymbol{\theta }}^{T}(X)\right] -E\left[ f_{\boldsymbol{\theta }%
		}^{\alpha }(X)\right] \boldsymbol{c}_{\alpha }\left( \boldsymbol{\theta }%
		\right) \boldsymbol{c}_{\alpha }^{T}\left( \boldsymbol{\theta }\right) \\
		& =\boldsymbol{J}_{\alpha }\left( \boldsymbol{\theta }\right) -\kappa
		_{\alpha }(\boldsymbol{\theta })\boldsymbol{c}_{\alpha }\left( \boldsymbol{%
			\theta }\right) \boldsymbol{c}_{\alpha }^{T}\left( \boldsymbol{\theta }%
		\right) ,
	\end{align*}
	
	\begin{align*}
		\boldsymbol{K}_{\alpha }\left( \boldsymbol{\theta }\right) & =\mathrm{E}%
		\left[ \left( f_{\boldsymbol{\theta }}^{\alpha }(X)\left( \boldsymbol{u}_{%
			\boldsymbol{\theta }}(X)-\boldsymbol{c}_{\alpha }\left( \boldsymbol{\theta }%
		\right) \right) \right) \left( f_{\boldsymbol{\theta }}^{\alpha }(X)\left( 
		\boldsymbol{u}_{\boldsymbol{\theta }}^{T}(X)-\boldsymbol{c}_{\alpha
		}^{T}\left( \boldsymbol{\theta }\right) \right) \right) \right] \\
		& =\mathrm{E}\left[ f_{\boldsymbol{\theta }}^{2\alpha }(X)\boldsymbol{u}_{%
			\boldsymbol{\theta }}(X)\boldsymbol{u}_{\boldsymbol{\theta }}^{T}(X)\right] +%
		\mathrm{E}\left[ f_{\boldsymbol{\theta }}^{2\alpha }(X)\right] \boldsymbol{c}%
		_{\alpha }\left( \boldsymbol{\theta }\right) \boldsymbol{c}_{\alpha
		}^{T}\left( \boldsymbol{\theta }\right) \\
		& -\mathrm{E}\left[ f_{\boldsymbol{\theta }}^{2\alpha }(X)\boldsymbol{u}_{%
			\boldsymbol{\theta }}(X)\right] \boldsymbol{c}_{\alpha }^{T}\left( 
		\boldsymbol{\theta }\right) -\boldsymbol{c}_{\alpha }\left( \boldsymbol{%
			\theta }\right) \mathrm{E}\left[ f_{\boldsymbol{\theta }}^{2\alpha }(X)%
		\boldsymbol{u}_{\boldsymbol{\theta }}^{T}(X)\right] \\
		& =\boldsymbol{J}_{2\alpha }\left( \boldsymbol{\theta }\right) +\kappa
		_{2\alpha }(\boldsymbol{\theta })\boldsymbol{c}_{\alpha }\left( \boldsymbol{%
			\theta }\right) \boldsymbol{c}_{\alpha }^{T}\left( \boldsymbol{\theta }%
		\right) \\
		& -\kappa _{2\alpha }(\boldsymbol{\theta })\boldsymbol{c}_{2\alpha }\left( 
		\boldsymbol{\theta }\right) \boldsymbol{c}_{\alpha }^{T}\left( \boldsymbol{%
			\theta }\right) -\kappa _{2\alpha }(\boldsymbol{\theta })\boldsymbol{c}%
		_{\alpha }\left( \boldsymbol{\theta }\right) \boldsymbol{c}_{2\alpha
		}^{T}\left( \boldsymbol{\theta }\right) \\
		& =\boldsymbol{S}_{2\alpha }\left( \boldsymbol{\theta }\right) +\kappa
		_{2\alpha }(\boldsymbol{\theta })\left( \boldsymbol{c}_{2\alpha }(%
		\boldsymbol{\theta })-\boldsymbol{c}_{\alpha }(\boldsymbol{\theta })\right)
		\left( \boldsymbol{c}_{2\alpha }(\boldsymbol{\theta })-\boldsymbol{c}%
		_{\alpha }(\boldsymbol{\theta })\right) ^{T},
	\end{align*}%
	where%
	\begin{align*}
		\boldsymbol{S}_{2\alpha }\left( \boldsymbol{\theta }\right) & =\boldsymbol{J}%
		_{2\alpha }\left( \boldsymbol{\theta }\right) -\kappa _{2\alpha }(%
		\boldsymbol{\theta })\boldsymbol{c}_{2\alpha }\left( \boldsymbol{\theta }%
		\right) \boldsymbol{c}_{2\alpha }^{T}\left( \boldsymbol{\theta }\right) \\
		\left( \boldsymbol{c}_{2\alpha }(\boldsymbol{\theta })-\boldsymbol{c}%
		_{\alpha }(\boldsymbol{\theta })\right) \left( \boldsymbol{c}_{2\alpha }(%
		\boldsymbol{\theta })-\boldsymbol{c}_{\alpha }(\boldsymbol{\theta })\right)
		^{T}& =\boldsymbol{c}_{\alpha }\left( \boldsymbol{\theta }\right) 
		\boldsymbol{c}_{\alpha }^{T}\left( \boldsymbol{\theta }\right) -\boldsymbol{c%
		}_{2\alpha }\left( \boldsymbol{\theta }\right) \boldsymbol{c}_{\alpha
		}^{T}\left( \boldsymbol{\theta }\right) -\boldsymbol{c}_{\alpha }\left( 
		\boldsymbol{\theta }\right) \boldsymbol{c}_{2\alpha }^{T}\left( \boldsymbol{%
			\theta }\right) +\boldsymbol{c}_{2\alpha }\left( \boldsymbol{\theta }\right) 
		\boldsymbol{c}_{2\alpha }^{T}\left( \boldsymbol{\theta }\right) .
	\end{align*}
	
	\section{Proof of Theorem \protect\ref{Th4}\label{A2}}
	
	We shall follow \ref{Th1} as well as Propositions \ref{Prop2}, \ref{Prop4}, %
	\ref{Prop3}. About the first partition, since $\boldsymbol{c}_{1,\alpha
	}\left( \boldsymbol{\theta }\right) =\boldsymbol{0}_{2}$ it is trivial that%
	\begin{align*}
		\boldsymbol{S}_{1,\alpha }\left( \boldsymbol{\theta }\right) & =\boldsymbol{J%
		}_{1,\alpha }\left( \boldsymbol{\theta }\right) , \\
		\boldsymbol{K}_{1,\alpha }\left( \boldsymbol{\theta }\right) & =\boldsymbol{J%
		}_{1,2\alpha }\left( \boldsymbol{\theta }\right) .
	\end{align*}%
	On the other hand,%
	\begin{align*}
		\boldsymbol{S}_{2,\alpha }\left( \boldsymbol{\theta }\right) & =\boldsymbol{J%
		}_{2,\alpha }\left( \boldsymbol{\theta }\right) -\kappa _{\alpha }(%
		\boldsymbol{\theta })\boldsymbol{c}_{2,\alpha }\left( \boldsymbol{\theta }%
		\right) \boldsymbol{c}_{2,\alpha }^{T}\left( \boldsymbol{\theta }\right) \\
		& =\frac{1}{k^{\alpha }(\boldsymbol{\theta })(\alpha +1)^{3}}\boldsymbol{D}%
		_{2,\sigma _{1},\sigma _{2}}^{-1}\boldsymbol{J}_{2,\alpha }(\rho )%
		\boldsymbol{D}_{2,\sigma _{1},\sigma _{2}}^{-1} \\
		& -\frac{\alpha ^{2}}{k^{\alpha }(\boldsymbol{\theta })\left( \alpha
			+1\right) ^{3}}\boldsymbol{D}_{2,\sigma _{1},\sigma _{2}}^{-1}\boldsymbol{S}%
		_{2,2}\left( \rho \right) \boldsymbol{D}_{2,\sigma _{1},\sigma _{2}}^{-1} \\
		& =\frac{1}{k^{\alpha }(\boldsymbol{\theta })(\alpha +1)^{3}}\boldsymbol{D}%
		_{2,\sigma _{1},\sigma _{2}}^{-1}\left[ \boldsymbol{J}_{2,\alpha }(\rho
		)-\alpha ^{2}\boldsymbol{S}_{2,2}\left( \rho \right) \right] \boldsymbol{D}%
		_{2,\sigma _{1},\sigma _{2}}^{-1},
	\end{align*}%
	where%
	\begin{align*}
		& \boldsymbol{J}_{2,\alpha }(\rho )-\alpha ^{2}\boldsymbol{S}_{2,2}\left(
		\rho \right) \\
		& =\frac{1}{1-\rho ^{2}}\left[ 
		\begin{pmatrix}
			\alpha ^{2}-\rho ^{2}(\alpha ^{2}+1)+2 & \alpha ^{2}-\rho ^{2}(\alpha ^{2}+1)
			& -\rho (\alpha ^{2}+1) \\ 
			\alpha ^{2}-\rho ^{2}(\alpha ^{2}+1) & \alpha ^{2}-\rho ^{2}(\alpha ^{2}+1)+2
			& -\rho (\alpha ^{2}+1) \\ 
			-\rho (\alpha ^{2}+1) & -\rho (\alpha ^{2}+1) & \frac{\rho ^{2}(\alpha
				^{2}+1)+1}{1-\rho ^{2}}%
		\end{pmatrix}%
		\right. \\
		& \left. -\alpha ^{2}%
		\begin{pmatrix}
			1-\rho ^{2} & 1-\rho ^{2} & -\rho \\ 
			1-\rho ^{2} & 1-\rho ^{2} & -\rho \\ 
			-\rho & -\rho & \frac{\rho ^{2}}{1-\rho ^{2}}%
		\end{pmatrix}%
		\right] \\
		& =\frac{1}{1-\rho ^{2}}%
		\begin{pmatrix}
			2-\rho ^{2} & -\rho ^{2} & -\rho \\ 
			-\rho ^{2} & 2-\rho ^{2} & -\rho \\ 
			-\rho & -\rho & \frac{\rho ^{2}+1}{1-\rho ^{2}}%
		\end{pmatrix}%
		=\boldsymbol{S}_{2,1}(\rho ).
	\end{align*}
	
	In addition,%
	\begin{equation*}
		\boldsymbol{c}_{2\alpha }(\boldsymbol{\theta })-\boldsymbol{c}_{\alpha }(%
		\boldsymbol{\theta })=\frac{\alpha }{(2\alpha +1)(\alpha +1)}\boldsymbol{D}%
		_{2,\sigma _{1},\sigma _{2}}^{-1}%
		\begin{pmatrix}
			-1 \\ 
			-1 \\ 
			\frac{\rho }{1-\rho ^{2}}%
		\end{pmatrix}%
	\end{equation*}%
	\begin{equation*}
		\left( \boldsymbol{c}_{2\alpha }(\boldsymbol{\theta })-\boldsymbol{c}%
		_{\alpha }(\boldsymbol{\theta })\right) \left( \boldsymbol{c}_{2\alpha }(%
		\boldsymbol{\theta })-\boldsymbol{c}_{\alpha }(\boldsymbol{\theta })\right)
		^{T}=\frac{\alpha ^{2}}{(2\alpha +1)^{2}(\alpha +1)^{2}}\boldsymbol{D}%
		_{2,\sigma _{1},\sigma _{2}}^{-1}\boldsymbol{S}_{2,2}\left( \rho \right) 
		\boldsymbol{D}_{2,\sigma _{1},\sigma _{2}}^{-1}
	\end{equation*}%
	\begin{align*}
		\boldsymbol{K}_{\alpha }\left( \boldsymbol{\theta }\right) & =\boldsymbol{S}%
		_{2\alpha }\left( \boldsymbol{\theta }\right) +\kappa _{2\alpha }(%
		\boldsymbol{\theta })\left( \boldsymbol{c}_{2\alpha }(\boldsymbol{\theta })-%
		\boldsymbol{c}_{\alpha }(\boldsymbol{\theta })\right) \left( \boldsymbol{c}%
		_{2\alpha }(\boldsymbol{\theta })-\boldsymbol{c}_{\alpha }(\boldsymbol{%
			\theta })\right) ^{T} \\
		& =\frac{1}{k^{2\alpha }(\boldsymbol{\theta })(2\alpha +1)^{3}}\boldsymbol{D}%
		_{2,\sigma _{1},\sigma _{2}}^{-1}\boldsymbol{S}_{2,1}(\rho )\boldsymbol{D}%
		_{2,\sigma _{1},\sigma _{2}}^{-1}+\kappa _{2\alpha }(\boldsymbol{\theta })%
		\frac{\alpha ^{2}}{(2\alpha +1)^{2}(\alpha +1)^{2}}\boldsymbol{D}_{2,\sigma
			_{1},\sigma _{2}}^{-1}\boldsymbol{S}_{2,2}\left( \rho \right) \boldsymbol{D}%
		_{2,\sigma _{1},\sigma _{2}}^{-1} \\
		& =\frac{1}{k^{2\alpha }(\boldsymbol{\theta })(2\alpha +1)^{3}}\boldsymbol{D}%
		_{2,\sigma _{1},\sigma _{2}}^{-1}\boldsymbol{S}_{2,1}(\rho )\boldsymbol{D}%
		_{2,\sigma _{1},\sigma _{2}}^{-1}+\frac{\alpha ^{2}}{k^{2\alpha }(%
			\boldsymbol{\theta })(2\alpha +1)^{3}(\alpha +1)^{2}}\boldsymbol{D}%
		_{2,\sigma _{1},\sigma _{2}}^{-1}\boldsymbol{S}_{2,2}\left( \rho \right) 
		\boldsymbol{D}_{2,\sigma _{1},\sigma _{2}}^{-1} \\
		& =\frac{1}{k^{2\alpha }(\boldsymbol{\theta })(2\alpha +1)^{3}}\boldsymbol{D}%
		_{2,\sigma _{1},\sigma _{2}}^{-1}\left( (\alpha +1)^{2}\boldsymbol{S}%
		_{2,1}(\rho )+\alpha ^{2}\boldsymbol{S}_{2,2}\left( \rho \right) \right) 
		\boldsymbol{D}_{2,\sigma _{1},\sigma _{2}}^{-1}.
	\end{align*}
	
	\section{Proof of Theorem \ref{Thm7} \label{app:Thm7}}
	
	Let ${\boldsymbol{\theta }}_{0}\in \Theta _{0}$ be the true value of ${%
		\boldsymbol{\theta }}$. Using a Taylor series expansion we get 
	\begin{align}
		\boldsymbol{m}(\widehat{\boldsymbol{\theta }}_{R,\alpha })& =\boldsymbol{m}({%
			\boldsymbol{\theta }}_{0})+\boldsymbol{M}^{T}({\boldsymbol{\theta }}_{0})(%
		\widehat{\boldsymbol{\theta }}_{R,\alpha }-{\boldsymbol{\theta }}%
		_{0})+o_{p}\left( \left\Vert \widehat{\boldsymbol{\theta }}_{R,\alpha }-{%
			\boldsymbol{\theta }}_{0}\right\Vert \right)  \notag \\
		& =\boldsymbol{M}^{T}({\boldsymbol{\theta }}_{0})(\widehat{\boldsymbol{%
				\theta }}_{R,\alpha }-{\boldsymbol{\theta }}_{0})+o_{p}\left( \left\Vert 
		\widehat{\boldsymbol{\theta }}_{R,\alpha }-{\boldsymbol{\theta }}%
		_{0}\right\Vert \right) ,  \label{m}
	\end{align}%
	because from equation (\ref{2.5}) we have $\boldsymbol{m}({\boldsymbol{%
			\theta }}_{0})=\boldsymbol{0}_{r}$. Now, under $H_{0}$, 
	\begin{equation*}
		n^{1/2}(\widehat{\boldsymbol{\theta }}_{R,\alpha }-{\boldsymbol{\theta }}%
		_{0})\underset{n\rightarrow \infty }{\overset{\mathcal{L}}{\longrightarrow }}%
		\mathcal{N}(\boldsymbol{0}_{p},\boldsymbol{V}_{\alpha }\left( {\boldsymbol{%
				\theta }}_{0}\right) ).
	\end{equation*}%
	Therefore, from equation (\ref{m}) we get, under $H_{0}$, 
	\begin{equation*}
		n^{1/2}\boldsymbol{m}(\widehat{\boldsymbol{\theta }}_{R,\alpha })\underset{%
			n\rightarrow \infty }{\overset{\mathcal{L}}{\longrightarrow }}\mathcal{N}(%
		\boldsymbol{0}_{r},\boldsymbol{M}^{T}\left( {\boldsymbol{\theta }}%
		_{0}\right) \boldsymbol{V}_{\alpha }\left( {\boldsymbol{\theta }}_{0}\right) 
		\boldsymbol{M}\left( \boldsymbol{\theta }_{0}\right) ).
	\end{equation*}%
	As rank$\left( \boldsymbol{M}({\boldsymbol{\theta }})\right) =r$, we get 
	\begin{equation*}
		n\boldsymbol{m}^{T}(\widehat{\boldsymbol{\theta }}_{R,\alpha })\left( 
		\boldsymbol{M}^{T}({\boldsymbol{\theta }}_{0})\boldsymbol{V}_{\alpha }\left( 
		{\boldsymbol{\theta }}_{0}\right) \boldsymbol{M}({\boldsymbol{\theta }}%
		_{0})\right) ^{-1}\boldsymbol{m}(\widehat{\boldsymbol{\theta }}_{R,\alpha })%
		\underset{n\rightarrow \infty }{\overset{\mathcal{L}}{\longrightarrow }}\chi
		_{r}^{2}.
	\end{equation*}%
	Now $\boldsymbol{M}^{T}(\widehat{\boldsymbol{\theta }}_{R,\alpha })%
	\boldsymbol{V}_{\alpha }(\widehat{\boldsymbol{\theta }}_{R,\alpha })%
	\boldsymbol{M}(\widehat{\boldsymbol{\theta }}_{R,\alpha })$ is a consistent
	estimator of 
	\begin{equation*}
		\boldsymbol{M}^{T}({\boldsymbol{\theta }}_{0})\boldsymbol{V}_{\alpha }\left( 
		{\boldsymbol{\theta }}_{0}\right) \boldsymbol{M}({\boldsymbol{\theta }}_{0}).
	\end{equation*}%
	Hence, under $H_{0}$, 
	\begin{equation*}
		W_{n,\alpha }(\widehat{\boldsymbol{\theta }}_{R,\alpha })\underset{%
			n\rightarrow \infty }{\overset{\mathcal{L}}{\longrightarrow }}\chi _{r}^{2}.
	\end{equation*}
	
	\section{Proof of the formulas of the inner iterations of Algorithm \protect
		\ref{Algo}\label{A3}}
	
	Taking into account Theorem \ref{Th5} and the components of $\boldsymbol{u}_{%
		\boldsymbol{\theta }}(x,y)-\boldsymbol{c}_{\alpha }(\boldsymbol{\theta })$,
	given by 
	\begin{align*}
		u_{\mu _{1}}(x,y)-c_{\alpha }(\mu _{1})=& \tfrac{1}{\sigma _{1}\left( 1-\rho
			^{2}\right) }\left[ \tfrac{x-\mu _{1}}{\sigma _{1}}-\rho \left( \tfrac{y-\mu
			_{2}}{\sigma _{2}}\right) \right] , \\
		u_{\mu _{2}}(x,y)-c_{\alpha }(\mu _{2})=& \tfrac{1}{\sigma _{2}\left( 1-\rho
			^{2}\right) }\left[ \tfrac{y-\mu _{2}}{\sigma _{2}}-\rho \left( \tfrac{x-\mu
			_{1}}{\sigma _{1}}\right) \right] , \\
		u_{\sigma _{1}}(x,y)-c_{\alpha }(\sigma _{1})=& -\tfrac{1}{\sigma _{1}}%
		\left\{ \tfrac{1}{\alpha +1}+\tfrac{1}{1-\rho ^{2}}\left[ \rho \left( \tfrac{%
			x-\mu _{1}}{\sigma _{1}}\right) \left( \tfrac{y-\mu _{2}}{\sigma _{2}}%
		\right) -\left( \tfrac{x-\mu _{1}}{\sigma _{1}}\right) ^{2}\right] \right\} ,
		\\
		u_{\sigma _{2}}(x,y)-c_{\alpha }(\sigma _{2})=& -\tfrac{1}{\sigma _{2}}%
		\left\{ \tfrac{1}{\alpha +1}+\tfrac{1}{1-\rho ^{2}}\left[ \rho \left( \tfrac{%
			x-\mu _{1}}{\sigma _{1}}\right) \left( \tfrac{y-\mu _{2}}{\sigma _{2}}%
		\right) -\left( \tfrac{y-\mu _{2}}{\sigma _{2}}\right) ^{2}\right] \right\} ,
		\\
		u_{\rho }(x,y)-c_{\alpha }(\rho )=& \tfrac{1}{(1-\rho ^{2})^{2}}\left\{
		(1+\rho ^{2})\left( \tfrac{x-\mu _{1}}{\sigma _{1}}\right) \left( \tfrac{%
			y-\mu _{2}}{\sigma _{2}}\right) -\rho \left[ \tfrac{1-\rho ^{2}}{\alpha +1}%
		+\left( \tfrac{x-\mu _{1}}{\sigma _{1}}\right) ^{2}+\left( \tfrac{y-\mu _{2}%
		}{\sigma _{2}}\right) ^{2}\right] \right\} ,
	\end{align*}%
	the estimating equations are%
	\begin{align}
		\sum_{i=1}^{n}w_{i,\boldsymbol{\theta }}^{-\alpha }\widetilde{X}_{i}-\rho
		\sum_{i=1}^{n}w_{i,\boldsymbol{\theta }}^{-\alpha }\widetilde{Y}_{i}& =0,
		\label{aa} \\
		\sum_{i=1}^{n}w_{i,\boldsymbol{\theta }}^{-\alpha }\widetilde{Y}_{i}-\rho
		\sum_{i=1}^{n}w_{i,\boldsymbol{\theta }}^{-\alpha }\widetilde{X}_{i}& =0,
		\label{bb} \\
		\frac{1-\rho ^{2}}{\alpha +1}\sum_{i=1}^{n}w_{i,\boldsymbol{\theta }%
		}^{-\alpha }+\rho \sum_{i=1}^{n}w_{i,\boldsymbol{\theta }}^{-\alpha }%
		\widetilde{X}_{i}\widetilde{Y}_{i}-\sum_{i=1}^{n}w_{i,\boldsymbol{\theta }%
		}^{-\alpha }\widetilde{X}_{i}^{2}& =0,  \label{ccc} \\
		\frac{1-\rho ^{2}}{\alpha +1}\sum_{i=1}^{n}w_{i,\boldsymbol{\theta }%
		}^{-\alpha }+\rho \sum_{i=1}^{n}w_{i,\boldsymbol{\theta }}^{-\alpha }%
		\widetilde{X}_{i}\widetilde{Y}_{i}-\sum_{i=1}^{n}w_{i,\boldsymbol{\theta }%
		}^{-\alpha }\widetilde{Y}_{i}^{2}& =0,  \label{dd} \\
		(1+\rho ^{2})\sum_{i=1}^{n}w_{i,\boldsymbol{\theta }}^{-\alpha }\widetilde{X}%
		_{i}\widetilde{Y}_{i}+\rho \frac{1-\rho ^{2}}{\alpha +1}\sum_{i=1}^{n}w_{i,%
			\boldsymbol{\theta }}^{-\alpha }-\rho \sum_{i=1}^{n}w_{i,\boldsymbol{\theta }%
		}^{-\alpha }\widetilde{X}_{i}^{2}-\rho \sum_{i=1}^{n}w_{i,\boldsymbol{\theta 
		}}^{-\alpha }\widetilde{Y}_{i}^{2}& =0,  \label{ee}
	\end{align}%
	\begin{align*}
		w_{i,\boldsymbol{\theta }}& =\exp \left\{ \tfrac{1}{2(1-\rho ^{2})}\left[ 
		\widetilde{X}_{i}^{2}+\widetilde{Y}_{i}^{2}-2\rho \widetilde{X}_{i}%
		\widetilde{Y}_{i}\right] \right\} , \\
		\widetilde{X}_{i}& =\tfrac{X_{i}-\mu _{1}}{\sigma _{1}},\qquad \widetilde{Y}%
		_{i}=\tfrac{Y_{i}-\mu _{2}}{\sigma _{2}}.
	\end{align*}%
	Since $\rho \in (-1,1)$, from (\ref{aa})-(\ref{bb}) it holds%
	\begin{eqnarray}
		\sum_{i=1}^{n}w_{i,\boldsymbol{\theta }}^{-\alpha }\widetilde{X}_{i} &=&0
		\label{bb2} \\
		\sum_{i=1}^{n}w_{i,\boldsymbol{\theta }}^{-\alpha }\widetilde{Y}_{i} &=&0,
	\end{eqnarray}%
	from (\ref{ccc})-(\ref{dd}) 
	\begin{equation}
		\sum_{i=1}^{n}w_{i,\boldsymbol{\theta }}^{-\alpha }\widetilde{X}%
		_{i}^{2}=\sum_{i=1}^{n}w_{i,\boldsymbol{\theta }}^{-\alpha }\widetilde{Y}%
		_{i}^{2},  \label{cd2}
	\end{equation}%
	with%
	\begin{align*}
		\sum_{i=1}^{n}w_{i,\boldsymbol{\theta }}^{-\alpha }\widetilde{X}_{i}^{2}& =%
		\frac{1-\rho ^{2}}{\alpha +1}\sum_{i=1}^{n}w_{i,\boldsymbol{\theta }%
		}^{-\alpha }+\rho \sum_{i=1}^{n}w_{i,\boldsymbol{\theta }}^{-\alpha }%
		\widetilde{X}_{i}\widetilde{Y}_{i}, \\
		\sum_{i=1}^{n}w_{i,\boldsymbol{\theta }}^{-\alpha }\widetilde{Y}_{i}^{2}& =%
		\frac{1-\rho ^{2}}{\alpha +1}\sum_{i=1}^{n}w_{i,\boldsymbol{\theta }%
		}^{-\alpha }+\rho \sum_{i=1}^{n}w_{i,\boldsymbol{\theta }}^{-\alpha }%
		\widetilde{X}_{i}\widetilde{Y}_{i}.
	\end{align*}%
	Replacing both in (\ref{ee}) we get%
	\begin{equation*}
		(1+\rho ^{2})\sum_{i=1}^{n}w_{i,\boldsymbol{\theta }}^{-\alpha }\widetilde{X}%
		_{i}\widetilde{Y}_{i}+\rho \frac{1-\rho ^{2}}{\alpha +1}\sum_{i=1}^{n}w_{i,%
			\boldsymbol{\theta }}^{-\alpha }-2\rho \left( \frac{1-\rho ^{2}}{\alpha +1}%
		\sum_{i=1}^{n}w_{i,\boldsymbol{\theta }}^{-\alpha }+\rho \sum_{i=1}^{n}w_{i,%
			\boldsymbol{\theta }}^{-\alpha }\widetilde{X}_{i}\widetilde{Y}_{i}\right) =0,
	\end{equation*}%
	i.e.%
	\begin{equation}
		\rho =\frac{(\alpha +1)\sum\limits_{i=1}^{n}w_{i,\boldsymbol{\theta }%
			}^{-\alpha }\widetilde{X}_{i}\widetilde{Y}_{i}}{\sum\limits_{i=1}^{n}w_{i,%
				\boldsymbol{\theta }}^{-\alpha }}.  \label{e2}
	\end{equation}%
	Hence%
	\begin{align}
		\frac{1-\rho ^{2}}{\alpha +1}\sum_{i=1}^{n}w_{i,\boldsymbol{\theta }%
		}^{-\alpha }+\rho \sum_{i=1}^{n}w_{i,\boldsymbol{\theta }}^{-\alpha }%
		\widetilde{X}_{i}\widetilde{Y}_{i}& =\sum_{i=1}^{n}w_{i,\boldsymbol{\theta }%
		}^{-\alpha }\widetilde{X}_{i}^{2},  \notag \\
		\frac{1-\rho ^{2}}{\alpha +1}\sum_{i=1}^{n}w_{i,\boldsymbol{\theta }%
		}^{-\alpha }+\frac{\rho ^{2}}{\alpha +1}\sum\limits_{i=1}^{n}w_{i,%
			\boldsymbol{\theta }}^{-\alpha }& =\sum_{i=1}^{n}w_{i,\boldsymbol{\theta }%
		}^{-\alpha }\widetilde{X}_{i}^{2},  \notag \\
		\frac{1}{\alpha +1}\sum\limits_{i=1}^{n}w_{i,\boldsymbol{\theta }}^{-\alpha
		}& =\sum_{i=1}^{n}w_{i,\boldsymbol{\theta }}^{-\alpha }\widetilde{X}_{i}^{2},
		\label{c3}
	\end{align}%
	and%
	\begin{align}
		\frac{1-\rho ^{2}}{\alpha +1}\sum_{i=1}^{n}w_{i,\boldsymbol{\theta }%
		}^{-\alpha }+\rho \sum_{i=1}^{n}w_{i,\boldsymbol{\theta }}^{-\alpha }%
		\widetilde{X}_{i}\widetilde{Y}_{i}& =\sum_{i=1}^{n}w_{i,\boldsymbol{\theta }%
		}^{-\alpha }\widetilde{Y}_{i}^{2},  \notag \\
		\frac{1-\rho ^{2}}{\alpha +1}\sum_{i=1}^{n}w_{i,\boldsymbol{\theta }%
		}^{-\alpha }+\frac{\rho ^{2}}{\alpha +1}\sum\limits_{i=1}^{n}w_{i,%
			\boldsymbol{\theta }}^{-\alpha }& =\sum_{i=1}^{n}w_{i,\boldsymbol{\theta }%
		}^{-\alpha }\widetilde{Y}_{i}^{2},  \notag \\
		\frac{1}{\alpha +1}\sum\limits_{i=1}^{n}w_{i,\boldsymbol{\theta }}^{-\alpha
		}& =\sum_{i=1}^{n}w_{i,\boldsymbol{\theta }}^{-\alpha }\widetilde{Y}_{i}^{2}.
		\label{d3}
	\end{align}%
	Finally,%
	\begin{align*}
		\mu _{1}& =\frac{\sum\limits_{i=1}^{n}w_{i,\boldsymbol{\theta }}^{-\alpha
			}X_{i}}{\sum\limits_{i=1}^{n}w_{i,\boldsymbol{\theta }}^{-\alpha }},\quad
		\mu _{2}=\frac{\sum\limits_{i=1}^{n}w_{i,\boldsymbol{\theta }}^{-\alpha
			}Y_{i}}{\sum\limits_{i=1}^{n}w_{i,\boldsymbol{\theta }}^{-\alpha }}, \\
		\frac{\sigma _{1}^{2}}{\alpha +1}& =\frac{ \sum\limits_{i=1}^{n}w_{i,%
				\boldsymbol{\theta }}^{-\alpha }\left( X_{i}-\mu _{1}\right) ^{2}}{%
			\sum\limits_{i=1}^{n}w_{i,\boldsymbol{\theta }}^{-\alpha }},\quad \frac{%
			\sigma _{2}^{2}}{\alpha +1}=\frac{ \sum\limits_{i=1}^{n}w_{i,\boldsymbol{%
					\theta }}^{-\alpha }\left( Y_{i}-\mu _{2}\right) ^{2}}{ \sum%
			\limits_{i=1}^{n}w_{i,\boldsymbol{\theta }}^{-\alpha }}, \\
		\rho & =\frac{(\alpha +1) \sum\limits_{i=1}^{n}w_{i,\boldsymbol{\theta }%
			}^{-\alpha }\frac{X_{i}-\mu _{1}}{\sigma _{1}}\frac{Y_{i}-\mu _{2}}{\sigma
				_{2}}}{ \sum\limits_{i=1}^{n}w_{i,\boldsymbol{\theta }}^{-\alpha }},
	\end{align*}%
	from which are derived the main formulas of the inner iterations of the
	Iteratively Reweighted Moments Algorithm.
	
	\section{Proof of Theorem \protect\ref{th:IF}\label{A4}}
	
	From Theorem 5 in \cite{Broniatowski2012}, the IF associated to the
	MRPDE of $\boldsymbol{\theta }$ is given by%
	\begin{equation*}
		\mathcal{IF}((x,y)^{T},\boldsymbol{T}_{\alpha },F_{\boldsymbol{\theta }})=%
		\boldsymbol{S}_{\alpha }^{-1}\left( \boldsymbol{\theta }\right) f_{%
			\boldsymbol{\theta }}^{\alpha }(x,y)\left[ \boldsymbol{u}_{\boldsymbol{%
				\theta }}(x,y)-\boldsymbol{c}_{\alpha }(\boldsymbol{\theta })\right] ,
	\end{equation*}%
	where 
	\begin{equation*}
		\boldsymbol{S}_{\alpha }^{-1}\left( \boldsymbol{\theta }\right) =%
		\begin{pmatrix}
			\boldsymbol{S}_{1,\alpha }^{-1}(\boldsymbol{\boldsymbol{\theta }}) & 
			\boldsymbol{0}_{2\times 3} \\ 
			\boldsymbol{0}_{3\times 2} & \boldsymbol{S}_{2,\alpha }^{-1}(\boldsymbol{%
				\boldsymbol{\theta }})%
		\end{pmatrix}%
	\end{equation*}%
	is the inverse of the matrix, $\boldsymbol{S}_{\alpha }\left( \boldsymbol{%
		\theta }\right) $, defined in Theorem \ref{Th4}. The expressions of $%
	\boldsymbol{S}_{1,\alpha }^{-1}(\boldsymbol{\boldsymbol{\theta }})$\ and $%
	\boldsymbol{S}_{2,\alpha }^{-1}(\boldsymbol{\boldsymbol{\theta }})$\ are
	given in (\ref{s1}) and (\ref{s2}) respectively and the ones of $\boldsymbol{%
		u}_{\boldsymbol{\theta }}(x,y)-\boldsymbol{c}_{\alpha }(\boldsymbol{\theta }%
	) $\ in (\ref{aa})-(\ref{dd}). On one hand%
	\begin{equation*}
		\begin{pmatrix}
			\mathcal{IF}_{\alpha }(\mu _{1}) \\ 
			\mathcal{IF}_{\alpha }(\mu _{2})%
		\end{pmatrix}%
		=\frac{k^{\alpha }(\boldsymbol{\theta })\left( \alpha +1\right) ^{2}}{1-\rho
			^{2}}f_{\boldsymbol{\theta }}^{\alpha }(x,y)\left( 
		\begin{array}{cc}
			\sigma _{1} & \sigma _{1}\rho \\ 
			\sigma _{2}\rho & \sigma _{2}%
		\end{array}%
		\right) 
		\begin{pmatrix}
			\tfrac{x-\mu _{1}}{\sigma _{1}}-\rho \left( \tfrac{y-\mu _{2}}{\sigma _{2}}%
			\right) \\ 
			\tfrac{y-\mu _{2}}{\sigma _{2}}-\rho \left( \tfrac{x-\mu _{1}}{\sigma _{1}}%
			\right)%
		\end{pmatrix}%
		,
	\end{equation*}%
	where%
	\begin{equation*}
		\left( 
		\begin{array}{cc}
			\sigma _{1} & \sigma _{1}\rho \\ 
			\sigma _{2}\rho & \sigma _{2}%
		\end{array}%
		\right) 
		\begin{pmatrix}
			\tfrac{x-\mu _{1}}{\sigma _{1}}-\rho \left( \tfrac{y-\mu _{2}}{\sigma _{2}}%
			\right) \\ 
			\tfrac{y-\mu _{2}}{\sigma _{2}}-\rho \left( \tfrac{x-\mu _{1}}{\sigma _{1}}%
			\right)%
		\end{pmatrix}%
		=(1-\rho ^{2})%
		\begin{pmatrix}
			x-\mu _{1} \\ 
			y-\mu _{2}%
		\end{pmatrix}%
		,
	\end{equation*}%
	and 
	\begin{align*}
		f_{\boldsymbol{\theta }}^{\alpha }(x,y)& =\frac{w_{\alpha ,\boldsymbol{%
					\theta }}^{-\alpha /(1-\rho ^{2})}(x,y)}{k^{\alpha }(\boldsymbol{\theta })}
		\\
		& =\frac{1}{k^{\alpha }(\boldsymbol{\theta })}\exp \left\{ -\tfrac{\alpha }{%
			2(1-\rho ^{2})}\left[ (\tfrac{x-\mu _{1}}{\sigma _{1}})^{2}+(\tfrac{y-\mu
			_{2}}{\sigma _{2}})^{2}-2\rho \tfrac{x-\mu _{1}}{\sigma _{1}}\tfrac{y-\mu
			_{2}}{\sigma _{2}}\right] \right\} ,
	\end{align*}%
	hence we obtain (\ref{IF1})-(\ref{IF2}). On the other hand, we have%
	\begin{align*}
		\begin{pmatrix}
			\mathcal{IF}_{\alpha }(\sigma _{1}) \\ 
			\mathcal{IF}_{\alpha }(\sigma _{2})%
		\end{pmatrix}%
		& =k^{\alpha }(\boldsymbol{\theta })\frac{\left( \alpha +1\right) ^{3}}{2}f_{%
			\boldsymbol{\theta }}^{\alpha }(x,y)%
		\begin{pmatrix}
			\sigma _{1} & \sigma _{1}\rho ^{2} & \sigma _{1}\rho (1-\rho ^{2}) \\ 
			\sigma _{2}\rho ^{2} & \sigma _{2} & \sigma _{2}\rho (1-\rho ^{2})%
		\end{pmatrix}
		\\
		& \hspace{-1.5cm} \times 
		\begin{pmatrix}
			-\left\{ \tfrac{1}{\alpha +1}+\tfrac{1}{1-\rho ^{2}}\left[ \rho \tfrac{%
				(x-\mu _{1})(y-\mu _{2})}{\sigma _{1}\sigma _{2}}-\left( \tfrac{x-\mu _{1}}{%
				\sigma _{1}}\right) ^{2}\right] \right\} \\ 
			-\left\{ \tfrac{1}{\alpha +1}+\tfrac{1}{1-\rho ^{2}}\left[ \rho \tfrac{%
				(x-\mu _{1})(y-\mu _{2})}{\sigma _{1}\sigma _{2}}-\left( \tfrac{y-\mu _{2}}{%
				\sigma _{2}}\right) ^{2}\right] \right\} \\ 
			\tfrac{1}{(1-\rho ^{2})^{2}}\left\{ (1+\rho ^{2})\left( \tfrac{x-\mu _{1}}{%
				\sigma _{1}}\right) \left( \tfrac{y-\mu _{2}}{\sigma _{2}}\right) -\rho %
			\left[ \tfrac{1-\rho ^{2}}{\alpha +1}+\left( \tfrac{x-\mu _{1}}{\sigma _{1}}%
			\right) ^{2}+\left( \tfrac{y-\mu _{2}}{\sigma _{2}}\right) ^{2}\right]
			\right\}%
		\end{pmatrix}%
		,
	\end{align*}%
	where%
	\begin{eqnarray*}
		&&%
		\begin{pmatrix}
			\sigma _{1} & \sigma _{1}\rho ^{2} & \sigma _{1}\rho (1-\rho ^{2}) \\ 
			\sigma _{2}\rho ^{2} & \sigma _{2} & \sigma _{2}\rho (1-\rho ^{2})%
		\end{pmatrix}\\
		&&
		\cdot \begin{pmatrix}
			-\left\{ \tfrac{1}{\alpha +1}+\tfrac{1}{1-\rho ^{2}}\left[ \rho \tfrac{%
				(x-\mu _{1})(y-\mu _{2})}{\sigma _{1}\sigma _{2}}-\left( \tfrac{x-\mu _{1}}{%
				\sigma _{1}}\right) ^{2}\right] \right\} \\ 
			-\left\{ \tfrac{1}{\alpha +1}+\tfrac{1}{1-\rho ^{2}}\left[ \rho \tfrac{%
				(x-\mu _{1})(y-\mu _{2})}{\sigma _{1}\sigma _{2}}-\left( \tfrac{y-\mu _{2}}{%
				\sigma _{2}}\right) ^{2}\right] \right\} \\ 
			\tfrac{1}{(1-\rho ^{2})^{2}}\left\{ (1+\rho ^{2})\left( \tfrac{x-\mu _{1}}{%
				\sigma _{1}}\right) \left( \tfrac{y-\mu _{2}}{\sigma _{2}}\right) -\rho %
			\left[ \tfrac{1-\rho ^{2}}{\alpha +1}+\left( \tfrac{x-\mu _{1}}{\sigma _{1}}%
			\right) ^{2}+\left( \tfrac{y-\mu _{2}}{\sigma _{2}}\right) ^{2}\right]
			\right\}%
		\end{pmatrix}
		\\
		&=&%
		\begin{pmatrix}
			\tfrac{\sigma _{1}}{1-\rho ^{2}}\left[ \left( \tfrac{x-\mu _{1}}{\sigma _{1}}%
			\right) ^{2}-\rho ^{2}\left( \tfrac{y-\mu _{2}}{\sigma _{2}}\right)
			^{2}-(1-\rho ^{2})(1+2\rho ^{2})\tfrac{1}{\alpha +1}\right] \\ 
			\tfrac{\sigma _{2}}{1-\rho ^{2}}\left[ \left( \tfrac{y-\mu _{2}}{\sigma _{2}}%
			\right) ^{2}-\rho ^{2}\left( \tfrac{x-\mu _{1}}{\sigma _{1}}\right)
			^{2}-(1-\rho ^{2})(1+2\rho ^{2})\tfrac{1}{\alpha +1}\right]%
		\end{pmatrix}%
		.
	\end{eqnarray*}%
	Hence, we obtain (\ref{IF3})-(\ref{IF4}). Finally,%
	\begin{align*}
		\mathcal{IF}_{\alpha }(\rho )& =k^{\alpha }(\boldsymbol{\theta })\frac{%
			\left( \alpha +1\right) ^{3}}{2}f_{\boldsymbol{\theta }}^{\alpha }(x,y)%
		\begin{pmatrix}
			\rho (1-\rho ^{2}) & \rho (1-\rho ^{2}) & 2(1-\rho ^{2})^{2}%
		\end{pmatrix}
		\\
		& \hspace{-1cm} \times 
		\begin{pmatrix}
			-\left\{ \tfrac{1}{\alpha +1}+\tfrac{1}{1-\rho ^{2}}\left[ \rho \tfrac{x-\mu
				_{1}}{\sigma _{1}}\tfrac{y-\mu _{2}}{\sigma _{2}}-\left( \tfrac{x-\mu _{1}}{%
				\sigma _{1}}\right) ^{2}\right] \right\} \\ 
			-\left\{ \tfrac{1}{\alpha +1}+\tfrac{1}{1-\rho ^{2}}\left[ \rho \tfrac{x-\mu
				_{1}}{\sigma _{1}}\tfrac{y-\mu _{2}}{\sigma _{2}}-\left( \tfrac{y-\mu _{2}}{%
				\sigma _{2}}\right) ^{2}\right] \right\} \\ 
			\tfrac{1}{(1-\rho ^{2})^{2}}\left\{ (1+\rho ^{2})\left( \tfrac{x-\mu _{1}}{%
				\sigma _{1}}\right) \left( \tfrac{y-\mu _{2}}{\sigma _{2}}\right) -\rho %
			\left[ \tfrac{1-\rho ^{2}}{\alpha +1}+\left( \tfrac{x-\mu _{1}}{\sigma _{1}}%
			\right) ^{2}+\left( \tfrac{y-\mu _{2}}{\sigma _{2}}\right) ^{2}\right]
			\right\}%
		\end{pmatrix}%
	\end{align*}%
	where%
	\begin{align*}
		& 
		\begin{pmatrix}
			\rho (1-\rho ^{2}) & \rho (1-\rho ^{2}) & 2(1-\rho ^{2})^{2}%
		\end{pmatrix}\\
		&
		\cdot \begin{pmatrix}
			-\left\{ \tfrac{1}{\alpha +1}+\tfrac{1}{1-\rho ^{2}}\left[ \rho \tfrac{x-\mu
				_{1}}{\sigma _{1}}\tfrac{y-\mu _{2}}{\sigma _{2}}-\left( \tfrac{x-\mu _{1}}{%
				\sigma _{1}}\right) ^{2}\right] \right\} \\ 
			-\left\{ \tfrac{1}{\alpha +1}+\tfrac{1}{1-\rho ^{2}}\left[ \rho \tfrac{x-\mu
				_{1}}{\sigma _{1}}\tfrac{y-\mu _{2}}{\sigma _{2}}-\left( \tfrac{y-\mu _{2}}{%
				\sigma _{2}}\right) ^{2}\right] \right\} \\ 
			\tfrac{1}{(1-\rho ^{2})^{2}}\left\{ (1+\rho ^{2})\left( \tfrac{x-\mu _{1}}{%
				\sigma _{1}}\right) \left( \tfrac{y-\mu _{2}}{\sigma _{2}}\right) -\rho %
			\left[ \tfrac{1-\rho ^{2}}{\alpha +1}+\left( \tfrac{x-\mu _{1}}{\sigma _{1}}%
			\right) ^{2}+\left( \tfrac{y-\mu _{2}}{\sigma _{2}}\right) ^{2}\right]
			\right\}%
		\end{pmatrix}
		\\
		& =-\rho \left[ \left( \tfrac{x-\mu _{1}}{\sigma _{1}}\right) ^{2}+\left( 
		\tfrac{x-\mu _{1}}{\sigma _{1}}\right) ^{2}\right] +2\left( \tfrac{x-\mu _{1}%
		}{\sigma _{1}}\right) \left( \tfrac{y-\mu _{2}}{\sigma _{2}}\right) ,
	\end{align*}%
	from which it follows (\ref{IF5}).
	
	\section{Proofs of  cases of testing problems (Section \protect\ref%
		{secWald})\label{A5}}
	
	\noindent \textbf{Case \ref{subSec1}\ (Comparing means of two
		dependent populations with normal distributions)}.\textbf{\medskip }\texttt{%
		\newline
	}
	If we consider the function 
	\begin{equation*}
		m\left( \mu _{1},\mu _{2},\sigma _{1},\sigma _{2},\rho \right) =\mu _{1}-\mu
		_{2},
	\end{equation*}%
	the null hypothesis can be given by $m\left( \mu _{1},\mu _{2},\sigma
	_{1},\sigma _{2},\rho \right) =0$. In this case 
	\begin{align*}
		\boldsymbol{M}^{T}\left( \mu _{1},\mu _{2},\sigma _{1},\sigma _{2},\rho
		\right) & =\left( 
		\begin{array}{ccccc}
			1 & -1 & 0 & 0 & 0%
		\end{array}%
		\right) , \\
		\left( \boldsymbol{M}^{T}(\widehat{\boldsymbol{\theta }}_{R,\alpha })%
		\boldsymbol{V}_{\alpha }(\widehat{\boldsymbol{\theta }}_{R,\alpha })%
		\boldsymbol{M}(\widehat{\boldsymbol{\theta }}_{R,\alpha })\right) ^{-1}&
		=\left( \left( 
		\begin{array}{cc}
			1 & -1%
		\end{array}%
		\right) \boldsymbol{V}_{1,\alpha }(\widehat{\boldsymbol{\theta }}_{R,\alpha
		})\left( 
		\begin{array}{c}
			1 \\ 
			-1%
		\end{array}%
		\right) \right) ^{-1} \\
		& =\frac{\left( 2\alpha +1\right) ^{2}}{\left( \alpha +1\right) ^{4}} (%
		\widehat{\sigma }_{1,R,\alpha }^{2}-2\widehat{\rho }_{R,\alpha }\widehat{%
			\sigma }_{1,R,\alpha }\widehat{\sigma }_{2,R,\alpha }+\widehat{\sigma }%
		_{2,R,\alpha }^{2})^{-1} \\
		& =\frac{\left( 2\alpha +1\right) ^{2}}{\left( \alpha +1\right) ^{4}}  \big[ (%
		\widehat{\sigma }_{1,R,\alpha }-\widehat{\sigma }_{2,R,\alpha })^{2}\\
		& \hspace{1.8cm}+2(1-%
		\widehat{\rho }_{R,\alpha })\widehat{\sigma }_{1,R,\alpha },\widehat{\sigma }%
		_{2,R,\alpha }\big] ,
	\end{align*}%
	where $\boldsymbol{V}_{1,\alpha }(\cdot )$ is given by (\ref{v1}).
	Therefore, (\ref{waldCase1})\ is obtained.
	
	\noindent \textbf{Case \ref{subSec2}\ (Comparing variances of two
		dependent populations with normal distributions)}.\textbf{\medskip }\texttt{%
		\newline
	}
	If we consider the function 
	\begin{equation*}
		m\left( \mu _{1},\mu _{2},\sigma _{1},\sigma _{2},\rho \right) =\sigma
		_{1}-\sigma _{2},
	\end{equation*}%
	the null hypothesis can be given by $m\left( \mu _{1},\mu _{2},\sigma
	_{1},\sigma _{2},\rho \right) =0$. In this case 
	\begin{equation*}
		\boldsymbol{M}^{T}\left( \mu _{1},\mu _{2},\sigma _{1},\sigma _{2},\rho
		\right) =\left( 
		\begin{array}{ccccc}
			0 & 0 & 1 & -1 & 0%
		\end{array}%
		\right)
	\end{equation*}%
	and taking%
	\begin{align*}
		& \boldsymbol{M}^{T}\left( \mu _{1},\mu _{2},\sigma _{1},\sigma _{2},\rho
		\right) \boldsymbol{V}_{\alpha }(\boldsymbol{\theta })\boldsymbol{M}\left(
		\mu _{1},\mu _{2},\sigma _{1},\sigma _{2},\rho \right) \\
		& =\left( 
		\begin{array}{ccc}
			\sigma _{1} & -\sigma _{2} & 0%
		\end{array}%
		\right) \boldsymbol{V}_{2,\alpha }(\rho )%
		\begin{pmatrix}
			\sigma _{1} \\ 
			-\sigma _{2} \\ 
			0%
		\end{pmatrix}%
		=\dfrac{\left( \alpha +1\right) ^{4}}{\left( 2\alpha +1\right) ^{3}}%
		b_{\alpha }(\boldsymbol{\theta }),
	\end{align*}%
	we denote%
	\begin{align}
		b_{\alpha }(\boldsymbol{\theta })& =\left( \alpha +1\right) ^{2}b_{1,\alpha
		}(\boldsymbol{\theta })+\alpha ^{2}b_{2,\alpha }(\boldsymbol{\theta })
		\label{b} \\
		& =\frac{2\left( \alpha +1\right) ^{2}+\alpha ^{2}}{4}(\sigma _{1}-\sigma
		_{2})^{2}+\left( \alpha +1\right) ^{2}(1-\rho ^{2})\sigma _{1}\sigma _{2} 
		\notag \\
		& =\frac{1}{4}\left( \alpha +1\right) ^{2}\left\{ \left[ 2+\left( \tfrac{%
			\alpha }{\alpha +1}\right) ^{2}\right] (\sigma _{1}-\sigma
		_{2})^{2}+4(1-\rho ^{2})\sigma _{1}\sigma _{2}\right\} ,  \notag
	\end{align}%
	where%
	\begin{align*}
		b_{1,\alpha }(\boldsymbol{\theta })& =\frac{1}{2}\left( 
		\begin{array}{cc}
			\sigma _{1} & -\sigma _{2}%
		\end{array}%
		\right) 
		\begin{pmatrix}
			1 & \rho ^{2} \\ 
			\rho ^{2} & 1%
		\end{pmatrix}%
		\left( 
		\begin{array}{c}
			\sigma _{1} \\ 
			-\sigma _{2}%
		\end{array}%
		\right) \\
		& =\frac{1}{2}\left( \sigma _{1}^{2}+\sigma _{2}^{2})-\rho ^{2}\sigma
		_{1}\sigma _{2}\right) \\
		& =\frac{1}{2}(\sigma _{1}-\sigma _{2})^{2}+(1-\rho ^{2})\sigma _{1}\sigma
		_{2},
	\end{align*}%
	\begin{align*}
		b_{2,\alpha }(\boldsymbol{\theta })& =\frac{1}{4}\tfrac{1}{1-\rho ^{2}}%
		\left( 
		\begin{array}{ccc}
			\sigma _{1} & -\sigma _{2} & 0%
		\end{array}%
		\right) 
		\begin{pmatrix}
			1 & \rho ^{2} & \rho (1-\rho ^{2}) \\ 
			\rho ^{2} & 1 & \rho (1-\rho ^{2}) \\ 
			\rho (1-\rho ^{2}) & \rho (1-\rho ^{2}) & 2(1-\rho ^{2})^{2}%
		\end{pmatrix}
		\\
		& \times 
		\begin{pmatrix}
			1-\rho ^{2} & 1-\rho ^{2} & -\rho \\ 
			1-\rho ^{2} & 1-\rho ^{2} & -\rho \\ 
			-\rho & -\rho & \frac{\rho ^{2}}{1-\rho ^{2}}%
		\end{pmatrix}%
		\begin{pmatrix}
			1 & \rho ^{2} & \rho (1-\rho ^{2}) \\ 
			\rho ^{2} & 1 & \rho (1-\rho ^{2}) \\ 
			\rho (1-\rho ^{2}) & \rho (1-\rho ^{2}) & 2(1-\rho ^{2})^{2}%
		\end{pmatrix}%
		\begin{pmatrix}
			\sigma _{1} \\ 
			-\sigma _{2} \\ 
			0%
		\end{pmatrix}
		\\
		& =\frac{1}{4}\left( 
		\begin{array}{cc}
			\sigma _{1} & -\sigma _{2}%
		\end{array}%
		\right) \boldsymbol{1}_{2}\boldsymbol{1}_{2}^{T}%
		\begin{pmatrix}
			\sigma _{1} \\ 
			-\sigma _{2}%
		\end{pmatrix}%
		=\frac{1}{4}(\sigma _{1}-\sigma _{2})^{2}.
	\end{align*}%
	Finally, denoting $b_{\alpha }(\widehat{\boldsymbol{\theta }}_{R,\alpha
	})=\left( \alpha +1\right) ^{2}\beta _{\alpha }(\widehat{\boldsymbol{\theta }%
	}_{R,\alpha })$, we get (\ref{TS}).
	
	\noindent \textbf{Case \ref{subSec3}\ (Fixing a value of the correlation coefficient of two
		dependent populations with normal distributions)}.\textbf{\medskip }\texttt{%
		\newline
	}
	If we consider the function 
	\begin{equation*}
		m\left( \mu _{1},\mu _{2},\sigma _{1},\sigma _{2},\rho \right) =\rho -\rho
		_{0},
	\end{equation*}%
	the null hypothesis can be given by $m\left( \mu _{1},\mu _{2},\sigma
	_{1},\sigma _{2},\rho \right) =0$. In this case $\boldsymbol{M}^{T}(\mu
	_{1},\mu _{2},\sigma _{1},\sigma _{2},\rho )=\left( 
	\begin{array}{ccccc}
		0 & 0 & 0 & 0 & 1%
	\end{array}%
	\right) $ and we have%
	\begin{equation}
		\left( \boldsymbol{M}^{T}(\widehat{\boldsymbol{\theta }}_{R,\alpha })%
		\boldsymbol{V}_{\alpha }(\widehat{\boldsymbol{\theta }}_{R,\alpha })%
		\boldsymbol{M}(\widehat{\boldsymbol{\theta }}_{R,\alpha })\right) ^{-1}=%
		\frac{\left( 2\alpha +1\right) ^{3}}{\left( \alpha +1\right) ^{6}}\frac{1}{%
			(1-\widehat{\rho }_{R,\alpha }^{2})^{2}}.  \label{key}
	\end{equation}%
	Therefore, we get (\ref{TS2}).

	\noindent \textbf{Case \ref{subSec4}\ (Comparing means and variances of two
		dependent populations with normal distribution)}.\textbf{\medskip }\texttt{%
		\newline
	}If we consider the function%
	\begin{equation*}
		\boldsymbol{m}^{T}\left( \mu _{1},\mu _{2},\sigma _{1},\sigma _{2},\rho
		\right) =\left( \mu _{1}-\mu _{2},\sigma _{1}-\sigma _{2}\right) ,
	\end{equation*}%
	the null hypothesis can be written by $\boldsymbol{m}\left( \mu _{1},\mu
	_{2},\sigma _{1},\sigma _{2},\rho \right) =\boldsymbol{0}_{2}$. In this case,%
	\begin{equation*}
		\boldsymbol{M}^{T}(\widehat{\boldsymbol{\theta }}_{R,\alpha })=\left( 
		\begin{array}{ccccc}
			1 & -1 & 0 & 0 & 0 \\ 
			0 & 0 & 1 & -1 & 0%
		\end{array}%
		\right)
	\end{equation*}%
	and 
	\begin{equation*}
		\boldsymbol{M}^{T}\left( \boldsymbol{\theta }\right) \boldsymbol{V}_{\alpha
		}\left( \boldsymbol{\theta }\right) \boldsymbol{M}\left( \boldsymbol{\theta }%
		\right) =\frac{\left( \alpha +1\right) ^{4}}{\left( 2\alpha +1\right) ^{2}}%
		\left( 
		\begin{array}{cc}
			\sigma _{1}^{2}-2\rho \sigma _{1}\sigma _{2}+\sigma _{2}^{2} & 0 \\ 
			0 & \frac{\left( \alpha +1\right) ^{2}}{2\alpha +1}\beta _{\alpha }(%
			\boldsymbol{\theta })%
		\end{array}%
		\right) ,
	\end{equation*}%
	with $\beta _{\alpha }(\boldsymbol{\theta })$ given in (\ref{b}). Therefore,
	we get (\ref{waldCase2}).\textbf{\medskip }
	
	\noindent \textbf{Case \ref{subSec5}\ (Fixing a value for covariance of two
		normal populations)}.\textbf{\medskip }\texttt{\newline
	}If we consider the function%
	\begin{equation*}
		m\left( \mu _{1},\mu _{2},\sigma _{1},\sigma _{2},\rho \right) =\sigma
		_{1}\sigma _{2}\rho -\sigma _{12,0},
	\end{equation*}%
	the null hypothesis can be written as $m\left( \mu _{1},\mu _{2},\sigma
	_{1},\sigma _{2},\rho \right) =0$ and 
	\begin{eqnarray*}
		\boldsymbol{M}^{T}\left( \boldsymbol{\theta }\right) \boldsymbol{V}_{\alpha
		}\left( \boldsymbol{\theta }\right) \boldsymbol{M}\left( \boldsymbol{\theta }%
		\right) &=&\sigma _{1}^{2}\sigma _{2}^{2}\left( 
		\begin{array}{ccc}
			\rho & \rho & 1%
		\end{array}%
		\right) \boldsymbol{D}_{1,\sigma _{1},\sigma _{2}}\boldsymbol{V}_{2,\alpha
		}(\rho )\boldsymbol{D}_{1,\sigma _{1},\sigma _{2}}\left( 
		\begin{array}{c}
			\rho \\ 
			\rho \\ 
			1%
		\end{array}%
		\right) \\
		&=&\sigma _{1}^{2}\sigma _{2}^{2}\left[ \left( \alpha +1\right) ^{2}(%
		\widehat{\rho }_{R,\alpha }^{2}+1)+\frac{\alpha ^{2}}{2}\widehat{\rho }%
		_{R,\alpha }^{2}\right] ,
	\end{eqnarray*}%
	where $\boldsymbol{V}_{2,\alpha }(\rho )$ is given by (\ref{v2rho}%
	).Therefore, we get (\ref{chi}).\textbf{\medskip }
	
	\noindent \textbf{Case \ref{subSec6}\ (Fixing values for means of two
		dependent populations with normal distribution)}.\textbf{\medskip }\texttt{%
		\newline
	}If we consider the function%
	\begin{equation*}
		\boldsymbol{m}^{T}(\mu _{1},\mu _{2},\sigma _{1},\sigma _{2},\rho )=\left(
		\mu _{1}-\mu _{1,0},\mu _{2}-\mu _{2,0}\right) ,
	\end{equation*}%
	the null hypothesis can be written by $m\left( \mu _{1},\mu _{2},\sigma
	_{1},\sigma _{2},\rho \right) =0$. It is clear that 
	\begin{equation*}
		\boldsymbol{M}^{T}\left( \mu _{1},\mu _{2},\sigma _{1},\sigma _{2},\rho
		\right) =\left( 
		\begin{array}{ccccc}
			1 & 0 & 0 & 0 & 0 \\ 
			0 & 1 & 0 & 0 & 0%
		\end{array}%
		\right) .
	\end{equation*}%
	Therefore, we get (\ref{w}) since 
	\begin{align*}
		\left( \boldsymbol{M}^{T}\left( \boldsymbol{\theta }\right) \boldsymbol{V}%
		_{\alpha }\left( \boldsymbol{\theta }\right) \boldsymbol{M}\left( 
		\boldsymbol{\theta }\right) \right) ^{-1}& =\boldsymbol{V}_{1,\alpha }^{-1}%
		\boldsymbol{\left( {\boldsymbol{\theta }}\right) } \\
		& =\left( \frac{\left( \alpha +1\right) ^{4}}{\left( 2\alpha +1\right) ^{2}}%
		\left( 
		\begin{array}{cc}
			\sigma _{1}^{2} & \rho \sigma _{1}\sigma _{2} \\ 
			\rho \sigma _{1}\sigma _{2} & \sigma _{2}^{2}%
		\end{array}%
		\right) \right) ^{-1} \\
		& =\frac{\left( 2\alpha +1\right) ^{2}}{(1-\rho ^{2})\left( \alpha +1\right)
			^{4}}\left( 
		\begin{array}{cc}
			\frac{1}{\sigma _{1}^{2}} & -\frac{\rho }{\sigma _{1}\sigma _{2}} \\ 
			-\frac{\rho }{\sigma _{1}\sigma _{2}} & \frac{1}{\sigma _{2}^{2}}%
		\end{array}%
		\right) ,
	\end{align*}%
	\begin{align*}
		& \boldsymbol{m}^{T}\left( \mu _{1},\mu _{2},\sigma _{1},\sigma _{2},\rho 
		\text{ }\right) \left( \boldsymbol{M}^{T}\left( \boldsymbol{\theta }\right) 
		\boldsymbol{V}_{\alpha }\left( \boldsymbol{\theta }\right) \boldsymbol{M}%
		\left( \boldsymbol{\theta }\right) \right) ^{-1}\boldsymbol{m}\left( \mu
		_{1},\mu _{2},\sigma _{1},\sigma _{2},\rho \text{ }\right) \\
		& =\frac{\left( 2\alpha +1\right) ^{2}}{\left( \alpha +1\right)^{4}}
		\bigg[\frac{(%
			\widehat{\mu }_{1,R,\alpha }-\mu _{1,0})^{2}\widehat{\sigma }_{2,R,\alpha
			}^{2}}{\widehat{\sigma }_{1,R,\alpha }^{2}%
			\widehat{\sigma }_{2,R,\alpha }^{2}(1-\widehat{\rho }_{R,\alpha }^{2})} \\
		&  \hspace{0.5cm} + \frac{
			-2\widehat{\rho }_{R,\alpha }(\widehat{\mu }_{1,R,\alpha }-\mu _{1,0})(%
			\widehat{\mu }_{2,R,\alpha }-\mu _{2,0})\widehat{\sigma }_{1,R,\alpha }%
			\widehat{\sigma }_{2,R,\alpha }
			+
			(\widehat{\mu }_{2,R,\alpha }-\mu _{2,0})^{2}%
			\widehat{\sigma }_{1,R,\alpha }^{2}}{\widehat{\sigma }_{1,R,\alpha }^{2}%
			\widehat{\sigma }_{2,R,\alpha }^{2}(1-\widehat{\rho }_{R,\alpha }^{2})}\bigg] \\
		& =\frac{\left( 2\alpha +1\right) ^{2}}{\left( \alpha +1\right) ^{4}}
		\Bigg[\frac{%
			\left( \frac{\widehat{\mu }_{1,R,\alpha }-\mu _{1,0}}{\widehat{\sigma }%
				_{1,R,\alpha }}\right) ^{2}-2\widehat{\rho }_{R,\alpha }\left(\frac{\widehat{\mu }%
				_{1,R,\alpha }-\mu _{1,0}}{\widehat{\sigma }_{1,R,\alpha }}\right)\left(\frac{\widehat{%
					\mu }_{2,R,\alpha }-\mu _{2,0}}{\widehat{\sigma }_{2,R,\alpha }}\right)}{1-\widehat{\rho }_{R,\alpha }^{2}} \\
		&	\hspace{2.3cm} + \frac{%
			\left( 
			\dfrac{\widehat{\mu }_{2,R,\alpha }-\mu _{2,0}}{\widehat{\sigma }_{2,R,\alpha }}\right) ^{2}}{1-\widehat{\rho }_{R,\alpha }^{2}} \Bigg] \\
		& =\dfrac{\left( 2\alpha +1\right) ^{2}}{\left( \alpha +1\right) ^{4}}
		\bigg[\frac{%
			\left( \dfrac{\widehat{\mu }_{1,R,\alpha }-\mu _{1,0}}{\widehat{\sigma }%
				_{1,R,\alpha }}-\dfrac{\widehat{\mu }_{2,R,\alpha }-\mu _{2,0}}{\widehat{%
					\sigma }_{2,R,\alpha }}\right) ^{2}}{1-\widehat{\rho }_{R,\alpha }^{2}}\\
		&	\hspace{2.3cm} +\frac{+2(1-\widehat{\rho }_{R,\alpha })\left(\dfrac{%
				\widehat{\mu }_{1,R,\alpha }-\mu _{1,0}}{\widehat{\sigma }_{1,R,\alpha }}\right)
			\left(\dfrac{\widehat{\mu }_{2,R,\alpha }-\mu _{2,0}}{\widehat{\sigma }%
				_{2,R,\alpha }}\right)}{1-\widehat{\rho }_{R,\alpha }^{2}} \Bigg]
	\end{align*}%
	\textbf{\medskip }
	
	\noindent \textbf{Case \ref{SecVarCov}\ (Fixing values for variances and
		covariance of two dependent populations with normal distribution)}.\textbf{%
		\medskip }\texttt{\newline
	}If we consider the function%
	\begin{equation*}
		\boldsymbol{m}^{T}\left( \boldsymbol{\theta }\right) =\left( \sigma
		_{1}-\sigma _{1,0},\sigma _{2}-\sigma _{2,0},\sigma _{1}\sigma _{2}\rho
		-\sigma _{12,0}\right) ,
	\end{equation*}%
	the null hypothesis can be written by $\boldsymbol{m}\left( \boldsymbol{%
		\theta }\right) =\boldsymbol{0}_{3}$. Therefore,%
	\begin{equation*}
		\boldsymbol{M}^{T}(\boldsymbol{\theta })=(\boldsymbol{0}_{3\times 2},%
		\boldsymbol{M}_{22}^{T}(\boldsymbol{\theta })),\qquad \boldsymbol{M}%
		_{22}^{T}(\boldsymbol{\theta })=\left( 
		\begin{array}{ccc}
			1 & 0 & 0 \\ 
			0 & 1 & 0 \\ 
			\sigma _{2}\rho & \sigma _{1}\rho & \sigma _{1}\sigma _{2}%
		\end{array}%
		\right)
	\end{equation*}%
	and we denote by 
	\begin{align*}
		\boldsymbol{M}^{T}(\boldsymbol{\theta })\boldsymbol{V}_{\alpha }\left( 
		\boldsymbol{\theta }\right) \boldsymbol{M}(\boldsymbol{\theta }) =%
		\boldsymbol{M}_{22}^{T}(\boldsymbol{\theta })\boldsymbol{V}_{2,\alpha }%
		\boldsymbol{\left( {\boldsymbol{\theta }}\right) M}_{22}(\boldsymbol{\theta }%
		), 
	\end{align*}
	and
	\begin{align*}
		\left( \boldsymbol{M}^{T}(\boldsymbol{\theta })\boldsymbol{V}_{\alpha
		}\left( \boldsymbol{\theta }\right) \boldsymbol{M}(\boldsymbol{\theta }%
		)\right) ^{-1} & =\left( \boldsymbol{M}_{22}^{T}(\boldsymbol{\theta })%
		\boldsymbol{V}_{2,\alpha }\boldsymbol{\left( {\boldsymbol{\theta }}\right) M}%
		_{22}(\boldsymbol{\theta })\right) ^{-1} \\
		& 
		\hspace{-1.5cm}=\boldsymbol{M}_{22}^{-1}(\boldsymbol{\theta })\boldsymbol{D}_{2,\widehat{%
				\sigma }_{1,R,\alpha },\widehat{\sigma }_{2,R,\alpha }}^{-1}\boldsymbol{V}%
		_{2,\alpha }^{-1}\left( \rho \right) \boldsymbol{D}_{2,\widehat{\sigma }%
			_{1,R,\alpha },\widehat{\sigma }_{2,R,\alpha }}^{-1}\left( \boldsymbol{M}%
		_{22}^{T}(\boldsymbol{\theta })\right) ^{-1},
	\end{align*}%
	and%
	\begin{equation*}
		\boldsymbol{m}^{T}(\widehat{\boldsymbol{\theta }}_{R,\alpha })=(\widehat{%
			\sigma }_{1,R,\alpha }-\sigma _{1,0},\widehat{\sigma }_{2,R,\alpha }-\sigma
		_{2,0},\widehat{\sigma }_{1,R,\alpha }\widehat{\sigma }_{2,R,\alpha }%
		\widehat{\rho }_{R,\alpha }-\sigma _{12,0}).
	\end{equation*}%
	Therefore, we get (\ref{chi2}). \newpage
	
	\section{Complementary tables for Simulation (Section \protect\ref{sim}) 
		\label{A6}}
	
	\begin{center}
		\begin{table}[htbp]  \tabcolsep2.8pt  \centering%
			\begin{tabular}{llllllllllll}
				\hline
				&  &  & \multicolumn{3}{c}{slighly} & \multicolumn{3}{c}{regular} & 
				\multicolumn{3}{c}{heavily} \\ 
				$\rho $ & $\alpha $ & pure & $0.05$ & $0.10$ & $0.20$ & $0.05$ & $0.10$ & $%
				0.20$ & $0.05$ & $0.10$ & $0.20$ \\ \hline
				$0$ & $0$ & 0.225 & 0.232 & 0.239 & 0.249 & 0.229 & 0.234 & 0.245 & 0.305 & 
				0.372 & 0.499 \\ 
				& $0.1$ & 0.227 & 0.231 & 0.236 & 0.243 & 0.229 & 0.232 & 0.239 & 0.268 & 
				0.316 & 0.438 \\ 
				& $0.2$ & 0.237 & 0.239 & 0.244 & 0.251 & 0.239 & 0.241 & 0.245 & 0.250 & 
				0.276 & 0.355 \\ 
				& $0.3$ & 0.258 & 0.259 & 0.266 & 0.270 & 0.259 & 0.262 & 0.264 & 0.264 & 
				0.280 & 0.328 \\ 
				& $0.5$ & 0.347 & 0.360 & 0.361 & 0.368 & 0.353 & 0.363 & 0.353 & 0.364 & 
				0.366 & 0.396 \\ 
				& $0.7$ & 0.596 & 0.580 & 0.761 & 0.597 & 0.612 & 0.558 & 0.597 & 0.575 & 
				0.620 & 0.709 \\ \hline
				$0.3$ & $0$ & 0.214 & 0.226 & 0.230 & 0.237 & 0.220 & 0.225 & 0.234 & 0.299
				& 0.371 & 0.496 \\ 
				& $0.1$ & 0.215 & 0.224 & 0.227 & 0.232 & 0.219 & 0.223 & 0.229 & 0.258 & 
				0.311 & 0.434 \\ 
				& $0.2$ & 0.224 & 0.232 & 0.235 & 0.240 & 0.228 & 0.231 & 0.236 & 0.242 & 
				0.267 & 0.350 \\ 
				& $0.3$ & 0.244 & 0.251 & 0.254 & 0.261 & 0.247 & 0.251 & 0.256 & 0.255 & 
				0.269 & 0.322 \\ 
				& $0.5$ & 0.322 & 0.335 & 0.338 & 0.351 & 0.330 & 0.338 & 0.342 & 0.337 & 
				0.387 & 0.387 \\ 
				& $0.7$ & 0.510 & 0.555 & 0.557 & 0.615 & 0.550 & 0.573 & 0.550 & 0.556 & 
				0.627 & 0.641 \\ \hline
				$0.6$ & $0$ & 0.182 & 0.188 & 0.193 & 0.200 & 0.185 & 0.189 & 0.195 & 0.275
				& 0.362 & 0.494 \\ 
				& $0.1$ & 0.184 & 0.187 & 0.190 & 0.195 & 0.185 & 0.187 & 0.191 & 0.223 & 
				0.285 & 0.415 \\ 
				& $0.2$ & 0.192 & 0.193 & 0.197 & 0.201 & 0.192 & 0.193 & 0.196 & 0.205 & 
				0.233 & 0.317 \\ 
				& $0.3$ & 0.209 & 0.210 & 0.212 & 0.217 & 0.208 & 0.208 & 0.213 & 0.215 & 
				0.231 & 0.281 \\ 
				& $0.5$ & 0.280 & 0.285 & 0.291 & 0.294 & 0.282 & 0.279 & 0.290 & 0.299 & 
				0.298 & 0.332 \\ 
				& $0.7$ & 0.468 & 0.460 & 0.485 & 0.496 & 0.460 & 0.468 & 0.455 & 0.468 & 
				0.512 & 0.513 \\ \hline
				$0.9$ & $0$ & 0.098 & 0.103 & 0.105 & 0.109 & 0.101 & 0.103 & 0.106 & 0.230
				& 0.339 & 0.492 \\ 
				& $0.1$ & 0.098 & 0.102 & 0.104 & 0.107 & 0.101 & 0.102 & 0.104 & 0.131 & 
				0.188 & 0.336 \\ 
				& $0.2$ & 0.102 & 0.106 & 0.107 & 0.110 & 0.105 & 0.106 & 0.107 & 0.112 & 
				0.131 & 0.205 \\ 
				& $0.3$ & 0.112 & 0.115 & 0.116 & 0.119 & 0.114 & 0.114 & 0.115 & 0.118 & 
				0.128 & 0.166 \\ 
				& $0.5$ & 0.149 & 0.153 & 0.155 & 0.161 & 0.153 & 0.151 & 0.157 & 0.159 & 
				0.169 & 0.195 \\ 
				& $0.7$ & 0.247 & 0.263 & 0.248 & 0.247 & 0.251 & 0.249 & 0.299 & 0.271 & 
				0.280 & 0.311 \\ \hline
			\end{tabular}
			\caption{Simulated mean square error of the MRPDE for ratio of variances,
				$\protect\widehat{\gamma}_{R,\alpha}$, when $n=15$\label{table2}} 
		\end{table}%
		
		\begin{table}[htbp]  \tabcolsep2.8pt  \centering%
			\begin{tabular}{llllllllllll}
				\hline
				&  &  & \multicolumn{3}{c}{slighly} & \multicolumn{3}{c}{regular} & 
				\multicolumn{3}{c}{heavily} \\ 
				$\rho $ & $\alpha $ & pure & $0.05$ & $0.10$ & $0.20$ & $0.05$ & $0.10$ & $%
				0.20$ & $0.05$ & $0.10$ & $0.20$ \\ \hline
				$0$ & $0$ & 0.064 & 0.071 & 0.085 & 0.091 & 0.071 & 0.074 & 0.089 & 0.270 & 
				0.426 & 0.687 \\ 
				& $0.1$ & 0.062 & 0.065 & 0.075 & 0.082 & 0.067 & 0.066 & 0.077 & 0.169 & 
				0.290 & 0.563 \\ 
				& $0.2$ & 0.067 & 0.068 & 0.075 & 0.085 & 0.069 & 0.067 & 0.077 & 0.103 & 
				0.167 & 0.354 \\ 
				& $0.3$ & 0.080 & 0.080 & 0.087 & 0.097 & 0.083 & 0.081 & 0.088 & 0.098 & 
				0.131 & 0.249 \\ 
				& $0.5$ & 0.141 & 0.147 & 0.151 & 0.162 & 0.145 & 0.145 & 0.149 & 0.158 & 
				0.172 & 0.235 \\ 
				& $0.7$ & 0.253 & 0.263 & 0.267 & 0.280 & 0.255 & 0.257 & 0.269 & 0.273 & 
				0.288 & 0.338 \\ \hline
				& \multicolumn{1}{c}{MP} & 0.049 & 0.057 & 0.066 & 0.073 & 0.058 & 0.058 & 
				0.071 & 0.250 & 0.405 & 0.666 \\ \hline
				$0.3$ & $0$ & 0.064 & 0.080 & 0.084 & 0.097 & 0.074 & 0.079 & 0.090 & 0.270
				& 0.441 & 0.692 \\ 
				& $0.1$ & 0.062 & 0.074 & 0.076 & 0.086 & 0.068 & 0.071 & 0.080 & 0.167 & 
				0.299 & 0.567 \\ 
				& $0.2$ & 0.064 & 0.075 & 0.077 & 0.085 & 0.069 & 0.074 & 0.078 & 0.108 & 
				0.169 & 0.363 \\ 
				& $0.3$ & 0.078 & 0.084 & 0.090 & 0.099 & 0.082 & 0.086 & 0.091 & 0.100 & 
				0.136 & 0.255 \\ 
				& $0.5$ & 0.138 & 0.147 & 0.150 & 0.164 & 0.142 & 0.145 & 0.151 & 0.157 & 
				0.174 & 0.236 \\ 
				& $0.7$ & 0.250 & 0.255 & 0.268 & 0.277 & 0.257 & 0.260 & 0.267 & 0.265 & 
				0.290 & 0.326 \\ \hline
				& \multicolumn{1}{c}{MP} & 0.049 & 0.063 & 0.065 & 0.074 & 0.057 & 0.060 & 
				0.071 & 0.249 & 0.417 & 0.672 \\ \hline
				$0.6$ & $0$ & 0.070 & 0.081 & 0.083 & 0.100 & 0.075 & 0.080 & 0.092 & 0.286
				& 0.474 & 0.735 \\ 
				& $0.1$ & 0.068 & 0.074 & 0.076 & 0.089 & 0.071 & 0.074 & 0.081 & 0.163 & 
				0.303 & 0.582 \\ 
				& $0.2$ & 0.073 & 0.077 & 0.081 & 0.088 & 0.073 & 0.076 & 0.079 & 0.106 & 
				0.166 & 0.353 \\ 
				& $0.3$ & 0.085 & 0.088 & 0.092 & 0.099 & 0.087 & 0.087 & 0.092 & 0.102 & 
				0.134 & 0.248 \\ 
				& $0.5$ & 0.146 & 0.150 & 0.158 & 0.163 & 0.148 & 0.149 & 0.154 & 0.162 & 
				0.179 & 0.237 \\ 
				& $0.7$ & 0.260 & 0.262 & 0.271 & 0.280 & 0.255 & 0.262 & 0.265 & 0.273 & 
				0.293 & 0.335 \\ \hline
				& \multicolumn{1}{c}{MP} & 0.050 & 0.061 & 0.065 & 0.077 & 0.056 & 0.061 & 
				0.071 & 0.262 & 0.449 & 0.711 \\ \hline
				$0.9$ & $0$ & 0.068 & 0.084 & 0.088 & 0.107 & 0.082 & 0.086 & 0.093 & 0.342
				& 0.547 & 0.805 \\ 
				& $0.1$ & 0.067 & 0.080 & 0.080 & 0.097 & 0.078 & 0.080 & 0.081 & 0.136 & 
				0.248 & 0.528 \\ 
				& $0.2$ & 0.070 & 0.082 & 0.083 & 0.099 & 0.079 & 0.082 & 0.083 & 0.092 & 
				0.124 & 0.276 \\ 
				& $0.3$ & 0.083 & 0.095 & 0.093 & 0.111 & 0.092 & 0.093 & 0.094 & 0.101 & 
				0.113 & 0.193 \\ 
				& $0.5$ & 0.143 & 0.155 & 0.158 & 0.172 & 0.156 & 0.154 & 0.155 & 0.162 & 
				0.181 & 0.225 \\ 
				& $0.7$ & 0.257 & 0.268 & 0.273 & 0.285 & 0.270 & 0.273 & 0.268 & 0.279 & 
				0.300 & 0.345 \\ \hline
				& \multicolumn{1}{c}{MP} & 0.047 & 0.061 & 0.063 & 0.081 & 0.060 & 0.061 & 
				0.068 & 0.312 & 0.519 & 0.786 \\ \hline
			\end{tabular}
			\caption{Simulated significance level for testing equal variances through
				$W_{n,\alpha }(\widehat{\gamma }_{R,\alpha },\widehat{\rho }_{R,\alpha })$ given by
				(\ref{simW1}) and the Morgan-Pitman test, when $n=15$\label{table1}} 
		\end{table}%
		
		\begin{table}[htbp]  \tabcolsep2.8pt  \centering%
			\begin{tabular}{llllllllllll}
				\hline
				&  &  & \multicolumn{3}{c}{slighly} & \multicolumn{3}{c}{regular} & 
				\multicolumn{3}{c}{heavily} \\ 
				$\rho $ & $\alpha $ & pure & $0.05$ & $0.10$ & $0.20$ & $0.05$ & $0.10$ & $%
				0.20$ & $0.05$ & $0.10$ & $0.20$ \\ \hline
				$0$ & $0$ & 0.217 & 0.224 & 0.230 & 0.239 & 0.221 & 0.224 & 0.233 & 0.338 & 
				0.431 & 0.596 \\ 
				& $0.1$ & 0.219 & 0.223 & 0.227 & 0.234 & 0.221 & 0.223 & 0.229 & 0.284 & 
				0.353 & 0.516 \\ 
				& $0.2$ & 0.227 & 0.230 & 0.233 & 0.240 & 0.229 & 0.230 & 0.234 & 0.253 & 
				0.290 & 0.403 \\ 
				& $0.3$ & 0.244 & 0.247 & 0.251 & 0.257 & 0.246 & 0.247 & 0.250 & 0.258 & 
				0.282 & 0.352 \\ 
				& $0.5$ & 0.306 & 0.313 & 0.314 & 0.325 & 0.309 & 0.312 & 0.314 & 0.320 & 
				0.332 & 0.374 \\ 
				& $0.7$ & 0.409 & 0.420 & 0.421 & 0.434 & 0.412 & 0.417 & 0.423 & 0.426 & 
				0.438 & 0.476 \\ \hline
				$0.3$ & $0$ & 0.217 & 0.226 & 0.231 & 0.238 & 0.221 & 0.226 & 0.234 & 0.338
				& 0.439 & 0.601 \\ 
				& $0.1$ & 0.218 & 0.225 & 0.228 & 0.234 & 0.220 & 0.224 & 0.229 & 0.282 & 
				0.357 & 0.518 \\ 
				& $0.2$ & 0.226 & 0.232 & 0.235 & 0.240 & 0.228 & 0.231 & 0.235 & 0.253 & 
				0.292 & 0.404 \\ 
				& $0.3$ & 0.243 & 0.248 & 0.252 & 0.258 & 0.244 & 0.249 & 0.252 & 0.259 & 
				0.282 & 0.354 \\ 
				& $0.5$ & 0.305 & 0.311 & 0.316 & 0.325 & 0.307 & 0.309 & 0.315 & 0.319 & 
				0.334 & 0.373 \\ 
				& $0.7$ & 0.410 & 0.413 & 0.424 & 0.432 & 0.412 & 0.417 & 0.422 & 0.421 & 
				0.441 & 0.468 \\ \hline
				$0.6$ & $0$ & 0.218 & 0.226 & 0.231 & 0.240 & 0.222 & 0.226 & 0.232 & 0.347
				& 0.462 & 0.628 \\ 
				& $0.1$ & 0.220 & 0.225 & 0.228 & 0.235 & 0.222 & 0.224 & 0.228 & 0.279 & 
				0.360 & 0.527 \\ 
				& $0.2$ & 0.228 & 0.232 & 0.235 & 0.240 & 0.230 & 0.231 & 0.234 & 0.252 & 
				0.289 & 0.400 \\ 
				& $0.3$ & 0.246 & 0.250 & 0.251 & 0.256 & 0.247 & 0.246 & 0.251 & 0.259 & 
				0.281 & 0.348 \\ 
				& $0.5$ & 0.311 & 0.314 & 0.319 & 0.324 & 0.311 & 0.311 & 0.317 & 0.322 & 
				0.335 & 0.376 \\ 
				& $0.7$ & 0.415 & 0.423 & 0.432 & 0.439 & 0.416 & 0.417 & 0.424 & 0.431 & 
				0.448 & 0.478 \\ \hline
				$0.9$ & $0$ & 0.216 & 0.226 & 0.230 & 0.239 & 0.223 & 0.228 & 0.232 & 0.377
				& 0.507 & 0.688 \\ 
				& $0.1$ & 0.217 & 0.225 & 0.227 & 0.235 & 0.223 & 0.226 & 0.227 & 0.258 & 
				0.325 & 0.503 \\ 
				& $0.2$ & 0.226 & 0.233 & 0.234 & 0.243 & 0.231 & 0.233 & 0.233 & 0.239 & 
				0.261 & 0.354 \\ 
				& $0.3$ & 0.244 & 0.251 & 0.253 & 0.264 & 0.250 & 0.251 & 0.252 & 0.255 & 
				0.267 & 0.319 \\ 
				& $0.5$ & 0.323 & 0.333 & 0.338 & 0.353 & 0.333 & 0.333 & 0.333 & 0.340 & 
				0.355 & 0.397 \\ 
				& $0.7$ & 0.459 & 0.472 & 0.476 & 0.497 & 0.468 & 0.473 & 0.471 & 0.480 & 
				0.497 & 0.539 \\ \hline
			\end{tabular}%
			\caption{Simulated mean square error of the MRPDE for correlation coefficient,
				$\protect\widehat{\rho}_{R,\alpha}$, when $n=15$\label{table2b}} 
		\end{table}%
		
		\begin{table}[htbp]  \tabcolsep2.8pt  \centering%
			\begin{tabular}{llllllllllll}
				\hline
				&  &  & \multicolumn{3}{c}{slighly} & \multicolumn{3}{c}{regular} & 
				\multicolumn{3}{c}{heavily} \\ 
				$\rho $ & $\alpha $ & pure & $0.05$ & $0.10$ & $0.20$ & $0.05$ & $0.10$ & $%
				0.20$ & $0.05$ & $0.10$ & $0.20$ \\ \hline
				$0$ & $0$ & 0.054 & 0.061 & 0.072 & 0.079 & 0.063 & 0.063 & 0.076 & 0.061 & 
				0.072 & 0.079 \\ 
				& $0.1$ & 0.052 & 0.055 & 0.063 & 0.070 & 0.057 & 0.056 & 0.066 & 0.055 & 
				0.063 & 0.070 \\ 
				& $0.2$ & 0.056 & 0.055 & 0.064 & 0.072 & 0.058 & 0.056 & 0.064 & 0.055 & 
				0.064 & 0.072 \\ 
				& $0.3$ & 0.068 & 0.065 & 0.075 & 0.083 & 0.070 & 0.067 & 0.074 & 0.065 & 
				0.075 & 0.083 \\ 
				& $0.5$ & 0.122 & 0.128 & 0.133 & 0.143 & 0.127 & 0.128 & 0.130 & 0.128 & 
				0.133 & 0.143 \\ 
				& $0.7$ & 0.229 & 0.238 & 0.240 & 0.254 & 0.228 & 0.231 & 0.241 & 0.238 & 
				0.240 & 0.254 \\ \hline
				& \multicolumn{1}{c}{MP} & 0.049 & 0.057 & 0.066 & 0.073 & 0.058 & 0.058 & 
				0.071 & 0.250 & 0.405 & 0.666 \\ \hline
				$0.3$ & $0$ & 0.053 & 0.067 & 0.070 & 0.081 & 0.061 & 0.065 & 0.077 & 0.067
				& 0.070 & 0.081 \\ 
				& $0.1$ & 0.051 & 0.062 & 0.063 & 0.070 & 0.057 & 0.060 & 0.066 & 0.062 & 
				0.063 & 0.070 \\ 
				& $0.2$ & 0.052 & 0.062 & 0.064 & 0.071 & 0.058 & 0.060 & 0.065 & 0.062 & 
				0.064 & 0.071 \\ 
				& $0.3$ & 0.065 & 0.071 & 0.074 & 0.085 & 0.067 & 0.073 & 0.076 & 0.071 & 
				0.074 & 0.085 \\ 
				& $0.5$ & 0.119 & 0.127 & 0.131 & 0.144 & 0.124 & 0.126 & 0.132 & 0.127 & 
				0.131 & 0.144 \\ 
				& $0.7$ & 0.224 & 0.230 & 0.243 & 0.253 & 0.231 & 0.233 & 0.241 & 0.230 & 
				0.243 & 0.253 \\ \hline
				& \multicolumn{1}{c}{MP} & 0.049 & 0.063 & 0.065 & 0.074 & 0.057 & 0.060 & 
				0.071 & 0.249 & 0.417 & 0.672 \\ \hline
				$0.6$ & $0$ & 0.056 & 0.065 & 0.069 & 0.083 & 0.062 & 0.066 & 0.076 & 0.065
				& 0.069 & 0.083 \\ 
				& $0.1$ & 0.054 & 0.060 & 0.062 & 0.072 & 0.058 & 0.059 & 0.066 & 0.060 & 
				0.062 & 0.072 \\ 
				& $0.2$ & 0.058 & 0.060 & 0.065 & 0.072 & 0.060 & 0.061 & 0.064 & 0.060 & 
				0.065 & 0.072 \\ 
				& $0.3$ & 0.069 & 0.071 & 0.075 & 0.082 & 0.072 & 0.072 & 0.076 & 0.071 & 
				0.075 & 0.082 \\ 
				& $0.5$ & 0.127 & 0.129 & 0.138 & 0.142 & 0.129 & 0.130 & 0.133 & 0.129 & 
				0.138 & 0.142 \\ 
				& $0.7$ & 0.233 & 0.241 & 0.249 & 0.258 & 0.231 & 0.236 & 0.242 & 0.241 & 
				0.249 & 0.258 \\ \hline
				& \multicolumn{1}{c}{MP} & 0.050 & 0.061 & 0.065 & 0.077 & 0.056 & 0.061 & 
				0.071 & 0.262 & 0.449 & 0.711 \\ \hline
				$0.9$ & $0$ & 0.051 & 0.067 & 0.069 & 0.087 & 0.065 & 0.067 & 0.072 & 0.067
				& 0.069 & 0.087 \\ 
				& $0.1$ & 0.051 & 0.062 & 0.063 & 0.076 & 0.061 & 0.062 & 0.063 & 0.062 & 
				0.063 & 0.076 \\ 
				& $0.2$ & 0.053 & 0.063 & 0.064 & 0.078 & 0.063 & 0.062 & 0.063 & 0.063 & 
				0.064 & 0.078 \\ 
				& $0.3$ & 0.067 & 0.077 & 0.076 & 0.094 & 0.076 & 0.075 & 0.076 & 0.077 & 
				0.076 & 0.094 \\ 
				& $0.5$ & 0.147 & 0.160 & 0.164 & 0.185 & 0.159 & 0.155 & 0.158 & 0.160 & 
				0.164 & 0.185 \\ 
				& $0.7$ & 0.296 & 0.313 & 0.317 & 0.347 & 0.304 & 0.312 & 0.309 & 0.313 & 
				0.317 & 0.347 \\ \hline
				& \multicolumn{1}{c}{MP} & 0.047 & 0.061 & 0.063 & 0.081 & 0.060 & 0.061 & 
				0.068 & 0.312 & 0.519 & 0.786 \\ \hline
			\end{tabular}%
			\caption{Simulated significance level for testing null correlation coefficient through
				$W_{n,\alpha }^{\prime }(\widehat{\rho }_{UV,R,\alpha })$ given by
				(\ref{simW2}) and the Morgan-Pitman test, when $n=15$\label{table1b}} 
		\end{table}%
		
		\begin{table}[htbp]  \tabcolsep2.8pt  \centering%
			\begin{tabular}{llllllllllll}
				\hline
				&  &  & \multicolumn{3}{c}{slighly} & \multicolumn{3}{c}{regular} & 
				\multicolumn{3}{c}{heavily} \\ 
				$\rho $ & $\alpha $ & pure & $0.05$ & $0.10$ & $0.20$ & $0.05$ & $0.10$ & $%
				0.20$ & $0.05$ & $0.10$ & $0.20$ \\ \hline
				$0$ & $0$ & 0.115 & 0.123 & 0.128 & 0.131 & 0.121 & 0.123 & 0.131 & 0.278 & 
				0.403 & 0.552 \\ 
				& $0.1$ & 0.117 & 0.120 & 0.123 & 0.126 & 0.119 & 0.118 & 0.121 & 0.160 & 
				0.251 & 0.440 \\ 
				& $0.2$ & 0.121 & 0.123 & 0.125 & 0.128 & 0.122 & 0.122 & 0.124 & 0.134 & 
				0.168 & 0.296 \\ 
				& $0.3$ & 0.127 & 0.129 & 0.130 & 0.133 & 0.127 & 0.127 & 0.130 & 0.133 & 
				0.148 & 0.217 \\ 
				& $0.5$ & 0.144 & 0.145 & 0.147 & 0.151 & 0.144 & 0.144 & 0.148 & 0.147 & 
				0.153 & 0.179 \\ 
				& $0.7$ & 0.168 & 0.171 & 0.173 & 0.178 & 0.171 & 0.168 & 0.175 & 0.172 & 
				0.176 & 0.193 \\ \hline
				$0.3$ & $0$ & 0.110 & 0.117 & 0.121 & 0.124 & 0.114 & 0.119 & 0.124 & 0.274
				& 0.402 & 0.552 \\ 
				& $0.1$ & 0.111 & 0.114 & 0.116 & 0.119 & 0.112 & 0.114 & 0.116 & 0.155 & 
				0.242 & 0.438 \\ 
				& $0.2$ & 0.115 & 0.116 & 0.118 & 0.121 & 0.116 & 0.117 & 0.118 & 0.127 & 
				0.158 & 0.290 \\ 
				& $0.3$ & 0.121 & 0.122 & 0.123 & 0.127 & 0.121 & 0.122 & 0.124 & 0.126 & 
				0.139 & 0.210 \\ 
				& $0.5$ & 0.137 & 0.137 & 0.139 & 0.143 & 0.138 & 0.137 & 0.140 & 0.139 & 
				0.143 & 0.171 \\ 
				& $0.7$ & 0.159 & 0.161 & 0.163 & 0.168 & 0.161 & 0.161 & 0.164 & 0.162 & 
				0.166 & 0.184 \\ \hline
				$0.6$ & $0$ & 0.093 & 0.097 & 0.102 & 0.106 & 0.097 & 0.100 & 0.105 & 0.270
				& 0.402 & 0.553 \\ 
				& $0.1$ & 0.094 & 0.095 & 0.098 & 0.101 & 0.095 & 0.096 & 0.097 & 0.134 & 
				0.224 & 0.430 \\ 
				& $0.2$ & 0.097 & 0.098 & 0.100 & 0.102 & 0.098 & 0.098 & 0.099 & 0.108 & 
				0.138 & 0.262 \\ 
				& $0.3$ & 0.101 & 0.102 & 0.104 & 0.106 & 0.102 & 0.103 & 0.103 & 0.107 & 
				0.121 & 0.180 \\ 
				& $0.5$ & 0.114 & 0.115 & 0.118 & 0.120 & 0.116 & 0.116 & 0.116 & 0.118 & 
				0.125 & 0.146 \\ 
				& $0.7$ & 0.134 & 0.135 & 0.139 & 0.141 & 0.136 & 0.136 & 0.137 & 0.137 & 
				0.144 & 0.157 \\ \hline
				$0.9$ & $0$ & 0.050 & 0.053 & 0.056 & 0.057 & 0.052 & 0.054 & 0.057 & 0.263
				& 0.401 & 0.555 \\ 
				& $0.1$ & 0.051 & 0.052 & 0.054 & 0.054 & 0.051 & 0.052 & 0.053 & 0.068 & 
				0.128 & 0.365 \\ 
				& $0.2$ & 0.053 & 0.053 & 0.054 & 0.055 & 0.053 & 0.053 & 0.054 & 0.056 & 
				0.066 & 0.138 \\ 
				& $0.3$ & 0.055 & 0.056 & 0.057 & 0.057 & 0.055 & 0.055 & 0.056 & 0.058 & 
				0.062 & 0.086 \\ 
				& $0.5$ & 0.063 & 0.063 & 0.064 & 0.065 & 0.063 & 0.062 & 0.063 & 0.064 & 
				0.067 & 0.077 \\ 
				& $0.7$ & 0.073 & 0.074 & 0.075 & 0.076 & 0.074 & 0.074 & 0.074 & 0.075 & 
				0.079 & 0.088 \\ \hline
			\end{tabular}
			\caption{Simulated mean square error of the MRPDE for ratio of variances,
				$\protect\widehat{\gamma}_{R,\alpha}$, when $n=50$\label{table6}} 
		\end{table}%
		
		\begin{table}[htbp]  \tabcolsep2.8pt  \centering%
			\begin{tabular}{llllllllllll}
				\hline
				&  &  & \multicolumn{3}{c}{slighly} & \multicolumn{3}{c}{regular} & 
				\multicolumn{3}{c}{heavily} \\ 
				$\rho $ & $\alpha $ & pure & $0.05$ & $0.10$ & $0.20$ & $0.05$ & $0.10$ & $%
				0.20$ & $0.05$ & $0.10$ & $0.20$ \\ \hline
				$0$ & $0$ & 0.054 & 0.072 & 0.081 & 0.089 & 0.063 & 0.071 & 0.091 & 0.535 & 
				0.809 & 0.980 \\ 
				& $0.1$ & 0.054 & 0.061 & 0.066 & 0.068 & 0.057 & 0.056 & 0.063 & 0.206 & 
				0.485 & 0.885 \\ 
				& $0.2$ & 0.055 & 0.058 & 0.061 & 0.064 & 0.057 & 0.053 & 0.061 & 0.100 & 
				0.215 & 0.593 \\ 
				& $0.3$ & 0.055 & 0.059 & 0.061 & 0.065 & 0.058 & 0.055 & 0.062 & 0.079 & 
				0.131 & 0.348 \\ 
				& $0.5$ & 0.062 & 0.064 & 0.071 & 0.076 & 0.065 & 0.063 & 0.071 & 0.072 & 
				0.097 & 0.182 \\ 
				& $0.7$ & 0.079 & 0.082 & 0.090 & 0.096 & 0.083 & 0.081 & 0.091 & 0.088 & 
				0.107 & 0.151 \\ \hline
				& \multicolumn{1}{c}{MP} & 0.050 & 0.067 & 0.077 & 0.084 & 0.061 & 0.068 & 
				0.086 & 0.529 & 0.805 & 0.979 \\ \hline
				$0.3$ & $0$ & 0.057 & 0.067 & 0.077 & 0.087 & 0.063 & 0.074 & 0.089 & 0.542
				& 0.817 & 0.981 \\ 
				& $0.1$ & 0.057 & 0.058 & 0.065 & 0.071 & 0.055 & 0.059 & 0.063 & 0.206 & 
				0.485 & 0.890 \\ 
				& $0.2$ & 0.056 & 0.056 & 0.061 & 0.066 & 0.054 & 0.056 & 0.061 & 0.099 & 
				0.205 & 0.592 \\ 
				& $0.3$ & 0.058 & 0.060 & 0.060 & 0.066 & 0.057 & 0.057 & 0.063 & 0.080 & 
				0.119 & 0.345 \\ 
				& $0.5$ & 0.065 & 0.065 & 0.067 & 0.073 & 0.067 & 0.066 & 0.070 & 0.076 & 
				0.088 & 0.175 \\ 
				& $0.7$ & 0.082 & 0.081 & 0.086 & 0.092 & 0.082 & 0.082 & 0.090 & 0.090 & 
				0.099 & 0.146 \\ \hline
				& \multicolumn{1}{c}{MP} & 0.052 & 0.063 & 0.072 & 0.081 & 0.059 & 0.070 & 
				0.084 & 0.536 & 0.814 & 0.980 \\ \hline
				$0.6$ & $0$ & 0.058 & 0.067 & 0.079 & 0.092 & 0.066 & 0.079 & 0.090 & 0.580
				& 0.847 & 0.986 \\ 
				& $0.1$ & 0.056 & 0.057 & 0.068 & 0.072 & 0.058 & 0.059 & 0.061 & 0.197 & 
				0.487 & 0.891 \\ 
				& $0.2$ & 0.056 & 0.056 & 0.065 & 0.066 & 0.058 & 0.057 & 0.060 & 0.093 & 
				0.199 & 0.565 \\ 
				& $0.3$ & 0.056 & 0.059 & 0.064 & 0.069 & 0.058 & 0.058 & 0.060 & 0.074 & 
				0.119 & 0.315 \\ 
				& $0.5$ & 0.065 & 0.066 & 0.072 & 0.073 & 0.066 & 0.064 & 0.067 & 0.075 & 
				0.090 & 0.161 \\ 
				& $0.7$ & 0.080 & 0.083 & 0.090 & 0.094 & 0.084 & 0.080 & 0.084 & 0.090 & 
				0.104 & 0.143 \\ \hline
				& \multicolumn{1}{c}{MP} & 0.053 & 0.062 & 0.074 & 0.085 & 0.060 & 0.074 & 
				0.084 & 0.574 & 0.843 & 0.985 \\ \hline
				$0.9$ & $0$ & 0.054 & 0.068 & 0.085 & 0.088 & 0.064 & 0.075 & 0.090 & 0.679
				& 0.907 & 0.996 \\ 
				& $0.1$ & 0.054 & 0.058 & 0.069 & 0.071 & 0.056 & 0.060 & 0.064 & 0.113 & 
				0.322 & 0.806 \\ 
				& $0.2$ & 0.054 & 0.055 & 0.066 & 0.066 & 0.057 & 0.056 & 0.061 & 0.066 & 
				0.103 & 0.339 \\ 
				& $0.3$ & 0.056 & 0.055 & 0.067 & 0.066 & 0.060 & 0.058 & 0.062 & 0.061 & 
				0.079 & 0.164 \\ 
				& $0.5$ & 0.067 & 0.067 & 0.074 & 0.074 & 0.068 & 0.064 & 0.070 & 0.070 & 
				0.083 & 0.112 \\ 
				& $0.7$ & 0.084 & 0.082 & 0.090 & 0.095 & 0.084 & 0.083 & 0.085 & 0.089 & 
				0.100 & 0.125 \\ \hline
				& \multicolumn{1}{c}{MP} & 0.048 & 0.062 & 0.079 & 0.081 & 0.057 & 0.070 & 
				0.084 & 0.673 & 0.904 & 0.996 \\ \hline
			\end{tabular}%
			\caption{Simulated significance level for testing equal variances through
				$W_{n,\alpha }(\widehat{\gamma }_{R,\alpha },\widehat{\rho }_{R,\alpha })$ given by
				(\ref{simW1}) and the Morgan-Pitman test, when $n=50$\label{table5}} 
		\end{table}%
		
		\begin{table}[htbp]  \tabcolsep2.8pt  \centering%
			\begin{tabular}{llllllllllll}
				\hline
				&  &  & \multicolumn{3}{c}{slighly} & \multicolumn{3}{c}{regular} & 
				\multicolumn{3}{c}{heavily} \\ 
				$\rho $ & $\alpha $ & pure & $0.05$ & $0.10$ & $0.20$ & $0.05$ & $0.10$ & $%
				0.20$ & $0.05$ & $0.10$ & $0.20$ \\ \hline
				$0$ & $0$ & 0.114 & 0.121 & 0.126 & 0.130 & 0.119 & 0.122 & 0.129 & 0.322 & 
				0.479 & 0.665 \\ 
				& $0.1$ & 0.115 & 0.119 & 0.121 & 0.124 & 0.117 & 0.117 & 0.120 & 0.174 & 
				0.288 & 0.525 \\ 
				& $0.2$ & 0.119 & 0.121 & 0.123 & 0.126 & 0.120 & 0.121 & 0.123 & 0.139 & 
				0.184 & 0.345 \\ 
				& $0.3$ & 0.125 & 0.127 & 0.129 & 0.132 & 0.126 & 0.126 & 0.129 & 0.136 & 
				0.157 & 0.245 \\ 
				& $0.5$ & 0.141 & 0.142 & 0.145 & 0.148 & 0.141 & 0.142 & 0.146 & 0.147 & 
				0.157 & 0.193 \\ 
				& $0.7$ & 0.164 & 0.166 & 0.168 & 0.173 & 0.166 & 0.164 & 0.170 & 0.169 & 
				0.177 & 0.201 \\ \hline
				$0.3$ & $0$ & 0.115 & 0.121 & 0.125 & 0.128 & 0.118 & 0.123 & 0.129 & 0.326
				& 0.487 & 0.673 \\ 
				& $0.1$ & 0.116 & 0.118 & 0.120 & 0.123 & 0.117 & 0.118 & 0.120 & 0.175 & 
				0.285 & 0.532 \\ 
				& $0.2$ & 0.120 & 0.121 & 0.122 & 0.125 & 0.120 & 0.121 & 0.122 & 0.138 & 
				0.179 & 0.346 \\ 
				& $0.3$ & 0.125 & 0.126 & 0.127 & 0.131 & 0.126 & 0.126 & 0.128 & 0.135 & 
				0.153 & 0.244 \\ 
				& $0.5$ & 0.141 & 0.142 & 0.143 & 0.147 & 0.142 & 0.142 & 0.145 & 0.146 & 
				0.154 & 0.192 \\ 
				& $0.7$ & 0.164 & 0.164 & 0.166 & 0.171 & 0.165 & 0.165 & 0.168 & 0.167 & 
				0.174 & 0.200 \\ \hline
				$0.6$ & $0$ & 0.114 & 0.120 & 0.126 & 0.130 & 0.119 & 0.124 & 0.129 & 0.348
				& 0.518 & 0.704 \\ 
				& $0.1$ & 0.116 & 0.118 & 0.121 & 0.124 & 0.117 & 0.119 & 0.120 & 0.171 & 
				0.289 & 0.552 \\ 
				& $0.2$ & 0.119 & 0.121 & 0.123 & 0.125 & 0.121 & 0.122 & 0.122 & 0.136 & 
				0.177 & 0.338 \\ 
				& $0.3$ & 0.125 & 0.126 & 0.128 & 0.130 & 0.126 & 0.127 & 0.127 & 0.134 & 
				0.154 & 0.232 \\ 
				& $0.5$ & 0.141 & 0.142 & 0.145 & 0.146 & 0.142 & 0.142 & 0.143 & 0.146 & 
				0.156 & 0.186 \\ 
				& $0.7$ & 0.163 & 0.165 & 0.169 & 0.171 & 0.165 & 0.165 & 0.166 & 0.169 & 
				0.178 & 0.196 \\ \hline
				$0.9$ & $0$ & 0.114 & 0.121 & 0.127 & 0.129 & 0.118 & 0.123 & 0.129 & 0.411
				& 0.588 & 0.767 \\ 
				& $0.1$ & 0.115 & 0.118 & 0.122 & 0.123 & 0.116 & 0.118 & 0.120 & 0.141 & 
				0.228 & 0.539 \\ 
				& $0.2$ & 0.119 & 0.121 & 0.123 & 0.125 & 0.120 & 0.121 & 0.122 & 0.124 & 
				0.139 & 0.242 \\ 
				& $0.3$ & 0.125 & 0.126 & 0.128 & 0.130 & 0.125 & 0.126 & 0.127 & 0.128 & 
				0.135 & 0.171 \\ 
				& $0.5$ & 0.141 & 0.141 & 0.144 & 0.146 & 0.141 & 0.141 & 0.142 & 0.143 & 
				0.149 & 0.164 \\ 
				& $0.7$ & 0.164 & 0.164 & 0.169 & 0.171 & 0.165 & 0.165 & 0.166 & 0.167 & 
				0.174 & 0.188 \\ \hline
			\end{tabular}%
			\caption{Simulated mean square error of the MRPDE for correlation coefficient,
				$\protect\widehat{\rho}_{R,\alpha}$, when $n=50$\label{table6b}} 
		\end{table}%
		
		\begin{table}[htbp]  \tabcolsep2.8pt  \centering%
			\begin{tabular}{llllllllllll}
				\hline
				&  &  & \multicolumn{3}{c}{slighly} & \multicolumn{3}{c}{regular} & 
				\multicolumn{3}{c}{heavily} \\ 
				$\rho $ & $\alpha $ & pure & $0.05$ & $0.10$ & $0.20$ & $0.05$ & $0.10$ & $%
				0.20$ & $0.05$ & $0.10$ & $0.20$ \\ \hline
				$0$ & $0$ & 0.051 & 0.068 & 0.078 & 0.086 & 0.062 & 0.069 & 0.088 & 0.532 & 
				0.807 & 0.979 \\ 
				& $0.1$ & 0.052 & 0.058 & 0.063 & 0.065 & 0.055 & 0.053 & 0.061 & 0.199 & 
				0.480 & 0.883 \\ 
				& $0.2$ & 0.052 & 0.055 & 0.058 & 0.061 & 0.054 & 0.051 & 0.058 & 0.096 & 
				0.210 & 0.588 \\ 
				& $0.3$ & 0.052 & 0.057 & 0.057 & 0.062 & 0.056 & 0.052 & 0.060 & 0.075 & 
				0.126 & 0.343 \\ 
				& $0.5$ & 0.059 & 0.061 & 0.067 & 0.072 & 0.062 & 0.059 & 0.067 & 0.069 & 
				0.092 & 0.176 \\ 
				& $0.7$ & 0.074 & 0.078 & 0.085 & 0.092 & 0.079 & 0.075 & 0.085 & 0.084 & 
				0.101 & 0.145 \\ \hline
				& \multicolumn{1}{c}{MP} & 0.050 & 0.067 & 0.077 & 0.084 & 0.061 & 0.068 & 
				0.086 & 0.529 & 0.805 & 0.979 \\ \hline
				$0.3$ & $0$ & 0.053 & 0.064 & 0.074 & 0.083 & 0.060 & 0.071 & 0.086 & 0.538
				& 0.815 & 0.980 \\ 
				& $0.1$ & 0.054 & 0.055 & 0.062 & 0.069 & 0.051 & 0.056 & 0.060 & 0.201 & 
				0.480 & 0.888 \\ 
				& $0.2$ & 0.054 & 0.053 & 0.058 & 0.063 & 0.051 & 0.054 & 0.058 & 0.095 & 
				0.200 & 0.586 \\ 
				& $0.3$ & 0.055 & 0.056 & 0.057 & 0.064 & 0.054 & 0.054 & 0.058 & 0.076 & 
				0.114 & 0.339 \\ 
				& $0.5$ & 0.062 & 0.061 & 0.063 & 0.070 & 0.063 & 0.061 & 0.067 & 0.072 & 
				0.084 & 0.168 \\ 
				& $0.7$ & 0.077 & 0.077 & 0.080 & 0.088 & 0.078 & 0.077 & 0.084 & 0.085 & 
				0.094 & 0.139 \\ \hline
				& \multicolumn{1}{c}{MP} & 0.052 & 0.063 & 0.072 & 0.081 & 0.059 & 0.070 & 
				0.084 & 0.536 & 0.814 & 0.980 \\ \hline
				$0.6$ & $0$ & 0.054 & 0.063 & 0.076 & 0.086 & 0.061 & 0.076 & 0.086 & 0.575
				& 0.845 & 0.985 \\ 
				& $0.1$ & 0.052 & 0.054 & 0.064 & 0.067 & 0.055 & 0.056 & 0.058 & 0.191 & 
				0.481 & 0.889 \\ 
				& $0.2$ & 0.052 & 0.053 & 0.060 & 0.062 & 0.055 & 0.053 & 0.056 & 0.088 & 
				0.192 & 0.558 \\ 
				& $0.3$ & 0.053 & 0.055 & 0.061 & 0.064 & 0.055 & 0.054 & 0.055 & 0.070 & 
				0.114 & 0.306 \\ 
				& $0.5$ & 0.060 & 0.061 & 0.067 & 0.069 & 0.060 & 0.060 & 0.062 & 0.069 & 
				0.085 & 0.155 \\ 
				& $0.7$ & 0.074 & 0.077 & 0.083 & 0.087 & 0.078 & 0.074 & 0.078 & 0.084 & 
				0.098 & 0.134 \\ \hline
				& \multicolumn{1}{c}{MP} & 0.053 & 0.062 & 0.074 & 0.085 & 0.060 & 0.074 & 
				0.084 & 0.574 & 0.843 & 0.985 \\ \hline
				$0.9$ & $0$ & 0.050 & 0.064 & 0.080 & 0.082 & 0.059 & 0.071 & 0.085 & 0.674
				& 0.905 & 0.996 \\ 
				& $0.1$ & 0.050 & 0.053 & 0.064 & 0.065 & 0.051 & 0.055 & 0.059 & 0.108 & 
				0.315 & 0.804 \\ 
				& $0.2$ & 0.051 & 0.051 & 0.061 & 0.061 & 0.053 & 0.052 & 0.056 & 0.060 & 
				0.096 & 0.331 \\ 
				& $0.3$ & 0.051 & 0.051 & 0.062 & 0.061 & 0.056 & 0.052 & 0.058 & 0.056 & 
				0.073 & 0.155 \\ 
				& $0.5$ & 0.061 & 0.060 & 0.067 & 0.067 & 0.062 & 0.059 & 0.064 & 0.063 & 
				0.076 & 0.104 \\ 
				& $0.7$ & 0.077 & 0.075 & 0.083 & 0.088 & 0.079 & 0.076 & 0.079 & 0.082 & 
				0.093 & 0.118 \\ \hline
				& \multicolumn{1}{c}{MP} & 0.048 & 0.062 & 0.079 & 0.081 & 0.057 & 0.070 & 
				0.084 & 0.673 & 0.904 & 0.996 \\ \hline
			\end{tabular}%
			\caption{Simulated significance level for testing null correlation coefficient through
				$W_{n,\alpha }^{\prime }(\widehat{\rho }_{UV,R,\alpha })$ given by
				(\ref{simW2}) and the Morgan-Pitman test, when $n=50$\label{table5b}} 
		\end{table}%
		
		\begin{table}[htbp]  \tabcolsep2.8pt  \centering%
			\begin{tabular}{llllllllllll}
				\hline
				&  &  & \multicolumn{3}{c}{slighly} & \multicolumn{3}{c}{regular} & 
				\multicolumn{3}{c}{heavily} \\ 
				$\rho $ & $\alpha $ & pure & $0.05$ & $0.10$ & $0.20$ & $0.05$ & $0.10$ & $%
				0.20$ & $0.05$ & $0.10$ & $0.20$ \\ \hline
				$0$ & $0$ & 0.081 & 0.086 & 0.089 & 0.093 & 0.085 & 0.089 & 0.094 & 0.288 & 
				0.429 & 0.569 \\ 
				& $0.1$ & 0.082 & 0.084 & 0.085 & 0.088 & 0.083 & 0.084 & 0.085 & 0.133 & 
				0.241 & 0.451 \\ 
				& $0.2$ & 0.085 & 0.086 & 0.087 & 0.089 & 0.086 & 0.086 & 0.086 & 0.100 & 
				0.141 & 0.292 \\ 
				& $0.3$ & 0.089 & 0.090 & 0.090 & 0.092 & 0.090 & 0.089 & 0.089 & 0.095 & 
				0.114 & 0.197 \\ 
				& $0.5$ & 0.099 & 0.100 & 0.100 & 0.103 & 0.100 & 0.099 & 0.099 & 0.102 & 
				0.108 & 0.140 \\ 
				& $0.7$ & 0.111 & 0.113 & 0.113 & 0.116 & 0.112 & 0.111 & 0.112 & 0.114 & 
				0.117 & 0.135 \\ \hline
				$0.3$ & $0$ & 0.077 & 0.082 & 0.086 & 0.088 & 0.080 & 0.084 & 0.090 & 0.288
				& 0.427 & 0.569 \\ 
				& $0.1$ & 0.078 & 0.079 & 0.082 & 0.084 & 0.078 & 0.080 & 0.081 & 0.128 & 
				0.234 & 0.449 \\ 
				& $0.2$ & 0.081 & 0.081 & 0.083 & 0.084 & 0.080 & 0.082 & 0.082 & 0.094 & 
				0.134 & 0.284 \\ 
				& $0.3$ & 0.084 & 0.085 & 0.086 & 0.087 & 0.084 & 0.085 & 0.085 & 0.090 & 
				0.109 & 0.190 \\ 
				& $0.5$ & 0.093 & 0.094 & 0.096 & 0.097 & 0.093 & 0.094 & 0.095 & 0.096 & 
				0.104 & 0.135 \\ 
				& $0.7$ & 0.105 & 0.107 & 0.108 & 0.110 & 0.105 & 0.107 & 0.106 & 0.108 & 
				0.113 & 0.130 \\ \hline
				$0.6$ & $0$ & 0.064 & 0.070 & 0.072 & 0.074 & 0.068 & 0.071 & 0.076 & 0.289
				& 0.427 & 0.571 \\ 
				& $0.1$ & 0.065 & 0.068 & 0.069 & 0.070 & 0.066 & 0.067 & 0.068 & 0.109 & 
				0.210 & 0.443 \\ 
				& $0.2$ & 0.067 & 0.069 & 0.070 & 0.070 & 0.068 & 0.068 & 0.069 & 0.079 & 
				0.112 & 0.255 \\ 
				& $0.3$ & 0.070 & 0.072 & 0.073 & 0.073 & 0.071 & 0.071 & 0.072 & 0.076 & 
				0.091 & 0.161 \\ 
				& $0.5$ & 0.077 & 0.080 & 0.081 & 0.081 & 0.079 & 0.079 & 0.080 & 0.081 & 
				0.088 & 0.114 \\ 
				& $0.7$ & 0.087 & 0.090 & 0.091 & 0.092 & 0.089 & 0.089 & 0.090 & 0.091 & 
				0.095 & 0.111 \\ \hline
				$0.9$ & $0$ & 0.035 & 0.038 & 0.039 & 0.040 & 0.037 & 0.039 & 0.041 & 0.286
				& 0.428 & 0.570 \\ 
				& $0.1$ & 0.035 & 0.037 & 0.037 & 0.038 & 0.036 & 0.037 & 0.037 & 0.049 & 
				0.106 & 0.387 \\ 
				& $0.2$ & 0.036 & 0.037 & 0.038 & 0.039 & 0.037 & 0.037 & 0.038 & 0.040 & 
				0.049 & 0.112 \\ 
				& $0.3$ & 0.038 & 0.039 & 0.039 & 0.040 & 0.039 & 0.039 & 0.039 & 0.040 & 
				0.044 & 0.064 \\ 
				& $0.5$ & 0.042 & 0.043 & 0.044 & 0.045 & 0.043 & 0.043 & 0.044 & 0.044 & 
				0.046 & 0.055 \\ 
				& $0.7$ & 0.048 & 0.049 & 0.049 & 0.050 & 0.049 & 0.049 & 0.050 & 0.050 & 
				0.052 & 0.058 \\ \hline
			\end{tabular}
			\caption{Simulated mean square error of the MRPDE for ratio of variances,
				$\protect\widehat{\gamma}_{R,\alpha}$, when $n=100$\label{table8}} 
		\end{table}%
		
		\begin{table}[htbp]  \tabcolsep2.8pt  \centering%
			\begin{tabular}{llllllllllll}
				\hline
				&  &  & \multicolumn{3}{c}{slighly} & \multicolumn{3}{c}{regular} & 
				\multicolumn{3}{c}{heavily} \\ 
				$\rho $ & $\alpha $ & pure & $0.05$ & $0.10$ & $0.20$ & $0.05$ & $0.10$ & $%
				0.20$ & $0.05$ & $0.10$ & $0.20$ \\ \hline
				$0$ & $0$ & 0.053 & 0.068 & 0.080 & 0.087 & 0.067 & 0.080 & 0.096 & 0.748 & 
				0.963 & 1.000 \\ 
				& $0.1$ & 0.054 & 0.059 & 0.063 & 0.067 & 0.057 & 0.056 & 0.061 & 0.280 & 
				0.684 & 0.983 \\ 
				& $0.2$ & 0.053 & 0.058 & 0.059 & 0.062 & 0.057 & 0.055 & 0.057 & 0.120 & 
				0.303 & 0.795 \\ 
				& $0.3$ & 0.054 & 0.058 & 0.060 & 0.063 & 0.056 & 0.055 & 0.058 & 0.086 & 
				0.166 & 0.500 \\ 
				& $0.5$ & 0.057 & 0.059 & 0.063 & 0.066 & 0.057 & 0.058 & 0.059 & 0.069 & 
				0.097 & 0.227 \\ 
				& $0.7$ & 0.062 & 0.064 & 0.069 & 0.072 & 0.065 & 0.062 & 0.063 & 0.072 & 
				0.089 & 0.153 \\ \hline
				& \multicolumn{1}{c}{MP} & 0.052 & 0.066 & 0.077 & 0.085 & 0.065 & 0.078 & 
				0.094 & 0.746 & 0.962 & 1.000 \\ \hline
				$0.3$ & $0$ & 0.051 & 0.069 & 0.079 & 0.086 & 0.063 & 0.077 & 0.095 & 0.760
				& 0.964 & 1.000 \\ 
				& $0.1$ & 0.051 & 0.057 & 0.063 & 0.069 & 0.050 & 0.057 & 0.057 & 0.280 & 
				0.678 & 0.988 \\ 
				& $0.2$ & 0.052 & 0.055 & 0.058 & 0.066 & 0.052 & 0.054 & 0.054 & 0.117 & 
				0.291 & 0.793 \\ 
				& $0.3$ & 0.052 & 0.055 & 0.058 & 0.065 & 0.051 & 0.055 & 0.056 & 0.083 & 
				0.161 & 0.494 \\ 
				& $0.5$ & 0.053 & 0.057 & 0.060 & 0.068 & 0.055 & 0.059 & 0.058 & 0.068 & 
				0.096 & 0.223 \\ 
				& $0.7$ & 0.059 & 0.065 & 0.065 & 0.073 & 0.061 & 0.063 & 0.065 & 0.070 & 
				0.087 & 0.148 \\ \hline
				& \multicolumn{1}{c}{MP} & 0.049 & 0.067 & 0.077 & 0.084 & 0.060 & 0.074 & 
				0.093 & 0.759 & 0.964 & 1.000 \\ \hline
				$0.6$ & $0$ & 0.049 & 0.071 & 0.084 & 0.088 & 0.066 & 0.078 & 0.098 & 0.803
				& 0.975 & 1.000 \\ 
				& $0.1$ & 0.047 & 0.059 & 0.068 & 0.070 & 0.055 & 0.059 & 0.062 & 0.267 & 
				0.661 & 0.987 \\ 
				& $0.2$ & 0.049 & 0.058 & 0.065 & 0.063 & 0.054 & 0.056 & 0.056 & 0.102 & 
				0.262 & 0.773 \\ 
				& $0.3$ & 0.049 & 0.059 & 0.063 & 0.062 & 0.056 & 0.057 & 0.056 & 0.076 & 
				0.146 & 0.464 \\ 
				& $0.5$ & 0.053 & 0.059 & 0.066 & 0.068 & 0.059 & 0.061 & 0.058 & 0.066 & 
				0.095 & 0.207 \\ 
				& $0.7$ & 0.058 & 0.065 & 0.070 & 0.074 & 0.063 & 0.066 & 0.065 & 0.070 & 
				0.089 & 0.144 \\ \hline
				& \multicolumn{1}{c}{MP} & 0.047 & 0.069 & 0.081 & 0.085 & 0.063 & 0.076 & 
				0.096 & 0.801 & 0.974 & 1.000 \\ \hline
				$0.9$ & $0$ & 0.051 & 0.070 & 0.079 & 0.091 & 0.068 & 0.078 & 0.096 & 0.881
				& 0.994 & 1.000 \\ 
				& $0.1$ & 0.050 & 0.058 & 0.062 & 0.072 & 0.055 & 0.057 & 0.060 & 0.131 & 
				0.423 & 0.948 \\ 
				& $0.2$ & 0.050 & 0.058 & 0.059 & 0.065 & 0.054 & 0.056 & 0.059 & 0.068 & 
				0.118 & 0.443 \\ 
				& $0.3$ & 0.050 & 0.059 & 0.058 & 0.065 & 0.054 & 0.056 & 0.059 & 0.063 & 
				0.084 & 0.201 \\ 
				& $0.5$ & 0.056 & 0.062 & 0.058 & 0.067 & 0.057 & 0.057 & 0.063 & 0.064 & 
				0.071 & 0.116 \\ 
				& $0.7$ & 0.062 & 0.068 & 0.066 & 0.075 & 0.062 & 0.064 & 0.070 & 0.069 & 
				0.073 & 0.106 \\ \hline
				& \multicolumn{1}{c}{MP} & 0.047 & 0.067 & 0.076 & 0.088 & 0.065 & 0.075 & 
				0.093 & 0.880 & 0.993 & 1.000 \\ \hline
			\end{tabular}%
			\caption{Simulated significance level for testing equal variances through
				$W_{n,\alpha }(\widehat{\gamma }_{R,\alpha },\widehat{\rho }_{R,\alpha })$ given by
				(\ref{simW1}) and the Morgan-Pitman test, when $n=100$\label{table7}} 
		\end{table}%
		
		\begin{table}[htbp]  \tabcolsep2.8pt  \centering%
			\begin{tabular}{llllllllllll}
				\hline
				&  &  & \multicolumn{3}{c}{slighly} & \multicolumn{3}{c}{regular} & 
				\multicolumn{3}{c}{heavily} \\ 
				$\rho $ & $\alpha $ & pure & $0.05$ & $0.10$ & $0.20$ & $0.05$ & $0.10$ & $%
				0.20$ & $0.05$ & $0.10$ & $0.20$ \\ \hline
				$0$ & $0$ & 0.081 & 0.086 & 0.089 & 0.092 & 0.085 & 0.089 & 0.093 & 0.334 & 
				0.510 & 0.685 \\ 
				& $0.1$ & 0.082 & 0.084 & 0.085 & 0.087 & 0.083 & 0.083 & 0.084 & 0.144 & 
				0.274 & 0.538 \\ 
				& $0.2$ & 0.085 & 0.086 & 0.086 & 0.088 & 0.085 & 0.085 & 0.085 & 0.104 & 
				0.154 & 0.337 \\ 
				& $0.3$ & 0.088 & 0.089 & 0.090 & 0.092 & 0.089 & 0.089 & 0.089 & 0.098 & 
				0.121 & 0.221 \\ 
				& $0.5$ & 0.098 & 0.099 & 0.099 & 0.102 & 0.099 & 0.098 & 0.098 & 0.103 & 
				0.112 & 0.152 \\ 
				& $0.7$ & 0.110 & 0.111 & 0.112 & 0.115 & 0.111 & 0.110 & 0.111 & 0.114 & 
				0.120 & 0.144 \\ \hline
				$0.3$ & $0$ & 0.081 & 0.085 & 0.089 & 0.092 & 0.084 & 0.088 & 0.093 & 0.342
				& 0.517 & 0.695 \\ 
				& $0.1$ & 0.082 & 0.083 & 0.085 & 0.087 & 0.081 & 0.083 & 0.084 & 0.143 & 
				0.274 & 0.545 \\ 
				& $0.2$ & 0.084 & 0.084 & 0.086 & 0.088 & 0.084 & 0.085 & 0.085 & 0.103 & 
				0.152 & 0.337 \\ 
				& $0.3$ & 0.088 & 0.088 & 0.090 & 0.091 & 0.087 & 0.089 & 0.089 & 0.097 & 
				0.120 & 0.220 \\ 
				& $0.5$ & 0.097 & 0.098 & 0.099 & 0.101 & 0.097 & 0.098 & 0.098 & 0.102 & 
				0.112 & 0.152 \\ 
				& $0.7$ & 0.109 & 0.110 & 0.112 & 0.114 & 0.109 & 0.111 & 0.110 & 0.113 & 
				0.120 & 0.143 \\ \hline
				$0.6$ & $0$ & 0.080 & 0.086 & 0.089 & 0.092 & 0.085 & 0.088 & 0.094 & 0.373
				& 0.550 & 0.726 \\ 
				& $0.1$ & 0.081 & 0.084 & 0.086 & 0.087 & 0.083 & 0.083 & 0.084 & 0.140 & 
				0.271 & 0.569 \\ 
				& $0.2$ & 0.083 & 0.085 & 0.087 & 0.088 & 0.085 & 0.085 & 0.086 & 0.101 & 
				0.144 & 0.330 \\ 
				& $0.3$ & 0.087 & 0.089 & 0.091 & 0.091 & 0.089 & 0.089 & 0.089 & 0.096 & 
				0.116 & 0.208 \\ 
				& $0.5$ & 0.096 & 0.099 & 0.100 & 0.100 & 0.098 & 0.098 & 0.099 & 0.102 & 
				0.111 & 0.147 \\ 
				& $0.7$ & 0.108 & 0.111 & 0.113 & 0.114 & 0.111 & 0.110 & 0.112 & 0.114 & 
				0.120 & 0.141 \\ \hline
				$0.9$ & $0$ & 0.079 & 0.086 & 0.088 & 0.092 & 0.085 & 0.089 & 0.094 & 0.446
				& 0.626 & 0.786 \\ 
				& $0.1$ & 0.080 & 0.084 & 0.085 & 0.087 & 0.082 & 0.084 & 0.085 & 0.104 & 
				0.194 & 0.572 \\ 
				& $0.2$ & 0.083 & 0.085 & 0.086 & 0.088 & 0.084 & 0.085 & 0.086 & 0.088 & 
				0.104 & 0.205 \\ 
				& $0.3$ & 0.086 & 0.089 & 0.089 & 0.092 & 0.088 & 0.089 & 0.090 & 0.090 & 
				0.097 & 0.132 \\ 
				& $0.5$ & 0.096 & 0.099 & 0.099 & 0.102 & 0.098 & 0.098 & 0.100 & 0.099 & 
				0.103 & 0.118 \\ 
				& $0.7$ & 0.109 & 0.111 & 0.112 & 0.115 & 0.110 & 0.111 & 0.113 & 0.112 & 
				0.115 & 0.126 \\ \hline
			\end{tabular}%
			\caption{Simulated mean square error of the MRPDE for correlation coefficient,
				$\protect\widehat{\rho}_{R,\alpha}$, when $n=100$\label{table8b}} 
		\end{table}%
		
		\begin{table}[htbp]  \tabcolsep2.8pt  \centering%
			\begin{tabular}{llllllllllll}
				\hline
				&  &  & \multicolumn{3}{c}{slighly} & \multicolumn{3}{c}{regular} & 
				\multicolumn{3}{c}{heavily} \\ 
				$\rho $ & $\alpha $ & pure & $0.05$ & $0.10$ & $0.20$ & $0.05$ & $0.10$ & $%
				0.20$ & $0.05$ & $0.10$ & $0.20$ \\ \hline
				$0$ & $0$ & 0.053 & 0.067 & 0.078 & 0.086 & 0.066 & 0.079 & 0.095 & 0.747 & 
				0.962 & 1.000 \\ 
				& $0.1$ & 0.052 & 0.058 & 0.062 & 0.066 & 0.056 & 0.055 & 0.060 & 0.278 & 
				0.682 & 0.983 \\ 
				& $0.2$ & 0.052 & 0.057 & 0.058 & 0.061 & 0.056 & 0.053 & 0.056 & 0.118 & 
				0.301 & 0.794 \\ 
				& $0.3$ & 0.053 & 0.056 & 0.058 & 0.061 & 0.055 & 0.053 & 0.057 & 0.085 & 
				0.164 & 0.497 \\ 
				& $0.5$ & 0.055 & 0.057 & 0.061 & 0.064 & 0.056 & 0.056 & 0.057 & 0.068 & 
				0.095 & 0.224 \\ 
				& $0.7$ & 0.061 & 0.063 & 0.067 & 0.070 & 0.063 & 0.060 & 0.062 & 0.070 & 
				0.087 & 0.150 \\ \hline
				& \multicolumn{1}{c}{MP} & 0.052 & 0.066 & 0.077 & 0.085 & 0.065 & 0.078 & 
				0.094 & 0.746 & 0.962 & 1.000 \\ \hline
				$0.3$ & $0$ & 0.049 & 0.068 & 0.077 & 0.085 & 0.061 & 0.076 & 0.093 & 0.759
				& 0.964 & 1.000 \\ 
				& $0.1$ & 0.050 & 0.055 & 0.061 & 0.068 & 0.049 & 0.056 & 0.055 & 0.276 & 
				0.675 & 0.988 \\ 
				& $0.2$ & 0.051 & 0.053 & 0.057 & 0.065 & 0.050 & 0.053 & 0.052 & 0.116 & 
				0.289 & 0.791 \\ 
				& $0.3$ & 0.050 & 0.054 & 0.056 & 0.063 & 0.050 & 0.054 & 0.054 & 0.081 & 
				0.159 & 0.490 \\ 
				& $0.5$ & 0.052 & 0.056 & 0.059 & 0.066 & 0.053 & 0.057 & 0.056 & 0.066 & 
				0.094 & 0.220 \\ 
				& $0.7$ & 0.057 & 0.063 & 0.064 & 0.071 & 0.059 & 0.060 & 0.063 & 0.067 & 
				0.085 & 0.145 \\ \hline
				& \multicolumn{1}{c}{MP} & 0.049 & 0.067 & 0.077 & 0.084 & 0.060 & 0.074 & 
				0.093 & 0.759 & 0.964 & 1.000 \\ \hline
				$0.6$ & $0$ & 0.048 & 0.070 & 0.082 & 0.086 & 0.064 & 0.077 & 0.096 & 0.801
				& 0.974 & 1.000 \\ 
				& $0.1$ & 0.046 & 0.058 & 0.066 & 0.068 & 0.053 & 0.058 & 0.061 & 0.264 & 
				0.659 & 0.987 \\ 
				& $0.2$ & 0.047 & 0.056 & 0.062 & 0.060 & 0.053 & 0.054 & 0.054 & 0.099 & 
				0.259 & 0.771 \\ 
				& $0.3$ & 0.047 & 0.058 & 0.060 & 0.060 & 0.054 & 0.056 & 0.054 & 0.074 & 
				0.143 & 0.460 \\ 
				& $0.5$ & 0.051 & 0.057 & 0.063 & 0.065 & 0.057 & 0.059 & 0.056 & 0.063 & 
				0.093 & 0.204 \\ 
				& $0.7$ & 0.055 & 0.061 & 0.068 & 0.071 & 0.061 & 0.063 & 0.062 & 0.068 & 
				0.086 & 0.140 \\ \hline
				& \multicolumn{1}{c}{MP} & 0.047 & 0.069 & 0.081 & 0.085 & 0.063 & 0.076 & 
				0.096 & 0.801 & 0.974 & 1.000 \\ \hline
				$0.9$ & $0$ & 0.048 & 0.068 & 0.076 & 0.089 & 0.065 & 0.076 & 0.094 & 0.880
				& 0.994 & 1.000 \\ 
				& $0.1$ & 0.048 & 0.056 & 0.060 & 0.069 & 0.053 & 0.055 & 0.058 & 0.128 & 
				0.418 & 0.947 \\ 
				& $0.2$ & 0.047 & 0.056 & 0.056 & 0.063 & 0.052 & 0.054 & 0.057 & 0.066 & 
				0.115 & 0.438 \\ 
				& $0.3$ & 0.048 & 0.056 & 0.056 & 0.062 & 0.052 & 0.053 & 0.057 & 0.061 & 
				0.081 & 0.197 \\ 
				& $0.5$ & 0.053 & 0.060 & 0.056 & 0.064 & 0.054 & 0.054 & 0.059 & 0.061 & 
				0.068 & 0.112 \\ 
				& $0.7$ & 0.059 & 0.065 & 0.062 & 0.071 & 0.059 & 0.061 & 0.067 & 0.066 & 
				0.069 & 0.101 \\ \hline
				& \multicolumn{1}{c}{MP} & 0.047 & 0.067 & 0.076 & 0.088 & 0.065 & 0.075 & 
				0.093 & 0.880 & 0.993 & 1.000 \\ \hline
			\end{tabular}%
			\caption{Simulated significance level for testing null correlation coefficient through
				$W_{n,\alpha }^{\prime }(\widehat{\rho }_{UV,R,\alpha })$ given by
				(\ref{simW2}) and the Morgan-Pitman test, when $n=100$\label{table7b}} 
		\end{table}%
	\end{center}
	
	
	
\end{appendices}


\bibliography{sn-bibliography}


\end{document}